\documentclass[11pt]{amsart}
\usepackage{amssymb,mathrsfs,amsmath}
 \usepackage{graphicx}
\usepackage{amscd,amsmath,amsopn,amssymb,amsthm,multicol}
\usepackage[color,matrix, all, 2cell]{xy}
\usepackage{amscd}
\usepackage{lscape}
\usepackage{slashed}
\usepackage{graphicx}
\usepackage{setspace}
\usepackage{upgreek}
\usepackage{textgreek}
\usepackage{enumerate}
\usepackage{color}
\usepackage{lscape}
\usepackage{tikz}
\usepackage{multirow}
\usepackage{cancel}
\usepackage{soul}
\usepackage{comment}
\usepackage{wasysym}
\usepackage{mathrsfs}
\usepackage{young}
\usepackage{mathtools}
\usepackage{bbm}
\usepackage{marginnote}
\usepackage{diagbox}

\numberwithin{equation}{section}

 \newcommand{\R}{\mathbb{R}}
\newcommand{\Z}{\mathbb{Z}}
\newcommand{\C}{\mathbb{C}}
\newcommand{\Hn}{\mathbb{H}}

\newcommand{\gl}{\mathfrak{gl}}
\renewcommand{\sl}{\mathfrak{sl}}
\renewcommand{\so}{\mathfrak{so}}
\renewcommand{\sp}{\mathfrak{sp}}
\renewcommand{\st}{\mathsf{stn}}
\renewcommand{\sc}{\mathsf{sc}}

\renewcommand{\ll}{\ell_{3}}

\renewcommand{\Re}{\mathbbm{Re}}
\renewcommand{\Im}{\mathbbm{Im}}
\DeclareMathOperator{\imm}{\mathsf{Im}}

\DeclareMathOperator{\Tor}{\mathsf{Tor}}
\DeclareMathAlphabet{\mathscrbf}{OMS}{mdugm}{b}{n}

\DeclareMathOperator{\SO}{\mathsf{SO}}
\DeclareMathOperator{\Sp}{\mathsf{Sp}}
 \DeclareMathOperator{\SU}{\mathsf{SU}}
    
\DeclareMathOperator{\U}{\mathsf{U}}

\DeclareMathAlphabet{\mathpzc}{OT1}{pzc}{m}{it}
\DeclareMathOperator{\Hh}{\mathsf{H}}

\DeclareMathOperator{\E}{\mathsf{E}}
\DeclareMathOperator{\Gl}{\mathsf{GL}}
\DeclareMathOperator{\Sl}{\mathsf{SL}}

\DeclareMathOperator{\Aut}{\mathsf{Aut}}
\DeclareMathOperator{\id}{\mathsf{id}}
\DeclareMathOperator{\Ss}{S}
\DeclareMathOperator{\Ed}{\mathsf{End}}

\DeclareMathOperator{\J}{\mathsf{J}}

\DeclareMathOperator{\K}{\mathsf{K}}
\DeclareMathOperator{\Ad}{\mathsf{Ad}}
\DeclareMathOperator{\vol}{\mathsf{vol}}
\DeclareMathOperator{\Id}{\mathsf{Id}}
\newcommand{\fr}{\mathfrak}
\newcommand{\al}{\alpha}
\newcommand{\be}{\beta}
\newcommand{\mc}{\mathcal}
\newcommand{\mf}{\mathsf}

\newcommand{\cc}{\big(}
\newcommand{\CC}{\Big(}
\newcommand{\rr}{\big)}
\newcommand{\RR}{\Big)}

\DeclareMathAlphabet{\mathscrbf}{OMS}{mdugm}{b}{n}

\DeclareMathOperator{\ad}{ad}

\DeclareMathOperator{\Tr}{\mathsf{Tr}}

\DeclareMathOperator{\Hom}{\mathsf{Hom}}
\DeclareMathOperator{\ke}{\mathsf{Ker}}

\DeclareMathOperator{\dd}{d}

\newcommand{\thickline}{\noalign{\hrule height 1pt}}

\newtheorem{theorem}{Theorem}[section]
\newtheorem{lem}[theorem]{Lemma}
\newtheorem{prop}[theorem]{Proposition}
\newtheorem{corol}[theorem]{Corollary}
\newtheorem{rem}[theorem]{Remark}

\theoremstyle{definition}
\newtheorem{defi}[theorem]{Definition}
\newtheorem{example}[theorem]{Example}
 
\theoremstyle{remark}

\numberwithin{equation}{section}

\definecolor{dark}{rgb}{0.18,0.18,0.68}
\definecolor{mydark}{rgb}{0.78,0.08,0.08}
\definecolor{crew}{rgb}{0.2,0.5,0.2}
\definecolor{mmg}{rgb}{0.31,0.50,0.23}
\definecolor{dblue}{rgb}{0.01,0.01,0.44}
\definecolor{red}{rgb}{0.57,0.11,0.15}
\definecolor{cobalt}{RGB}{61,89,171}
\usepackage[colorlinks,citecolor=cobalt,linkcolor=cobalt,urlcolor=cobalt,pdfpagemode=UseNone,backref = page]{hyperref}

 \input ulem.sty

\title[Differential geometry of $\SO^\ast(2n)$-type structures]{Differential geometry of $\SO^\ast(2n)$-type structures}

\author{Ioannis Chrysikos} 
\address{Faculty of Science, University of Hradec Kr\'alov\'e, Rokitanskeho 62, Hradec Kr\'alov\'e
50003, Czech Republic}
\email{ioannis.chrysikos@uhk.cz}

\author{Jan Gregorovi\v c} 
\address{Faculty of Science, University of Hradec Kr\'alov\'e, Rokitanskeho 62, Hradec Kr\'alov\'e
50003, Czech Republic}
\email{jan.gregorovic@seznam.cz}

\author{Henrik Winther} 
\address{Department of Mathematics and Statistics, Masaryk University, Kotl\'a\v{r}sk\'a 2, Brno 611 37, Czech Republic} 
\email{winther@math.muni.cz}

\begin{document}

\begin{abstract}
We study $4n$-dimensional smooth manifolds  admitting a $\SO^*(2n)$- or a $\SO^*(2n)\Sp(1)$-structure, where $\SO^*(2n)$ is the quaternionic real form of $\SO(2n, \C)$. We show that such $G$-structures, called    \textsf{almost hypercomplex/quaternionic skew-Hermitian structures},  form the symplectic analogue of the  better known almost hypercomplex/quaternionic-Hermitian structures (hH/qH for short). We present  several equivalent definitions of $\SO^*(2n)$- and $\SO^*(2n)\Sp(1)$-structures in terms of almost symplectic forms compatible with an almost hypercomplex/quaternionic structure,   a quaternionic skew-Hermitian form, or a   symmetric  4-tensor, the latter  establishing  the counterpart of the fundamental 4-form in almost hH/qH  geometries.  The intrinsic torsion of such structures is presented in terms of Salamon's $\E\Hh$-formalism, and  the  algebraic types of the corresponding geometries  are classified.   We construct explicit adapted connections to our $G$-structures and specify certain normalization conditions, under which these connections become minimal. Finally, we present the classification of symmetric spaces $K/L$ with $K$ semisimple admitting an invariant torsion-free $\SO^*(2n)\Sp(1)$-structure. This paper  is the first in a series aiming at the description of the differential geometry of   $\SO^*(2n)$- and $\SO^*(2n)\Sp(1)$-structures. 
 \end{abstract}

\maketitle

\tableofcontents


\section*{Introduction}\label{intro}
This article is the first in a series  studying  $4n$-dimensional manifolds $M$ $(n>1)$ admitting  a reduction of the frame bundle to the Lie subgroups $\SO^*(2n)$, or $\SO^*(2n)\Sp(1)$ of $\Gl(4n,\R)$.  Here, $\SO^*(2n)$ denotes the quaternionic real form $\SO(2n, \C)$, and $\SO^*(2n)\Sp(1)$ denotes the Lie group $\SO^*(2n)\times_{\Z_2}\Sp(1)$. Such structures lie inside of the realm of almost hypercomplex and almost quaternionic geometries, respectively.  The aim of this note is to highlight them as the symplectic analogue of the well-known almost hypercomplex-Hermitian (hH) structures, almost quaternionic-Hermitian (qH) structures and their pseudo-Riemannian counterparts, which have been examined by many leaders in differential geometry, see for example  \cite{Ale68,  Salamon86, Swann, AM, Cortes, Cabrera}.

Recall that given  an almost hypercomplex manifold $(M, H)$, a pseudo-Riemannian  metric $g$ which is $H$-Hermitian,  corresponds to  a reduction to $\Sp(p, q)\subset\Gl(n, \Hn)$.  Similarly, given an almost quaternionic manifold $(M, Q)$ a pseudo-Riemannian  metric $g$ which is $Q$-Hermitian, corresponds to a reduction to $\Sp(p, q)\Sp(1)\subset\Gl(n, \Hn)\Sp(1)$. Note that these pseudo-Riemannian metrics exist only if certain topological conditions are satisfied.  Here, as usual, we have interpreted  an almost hypercomplex structure $H=\{J_a : a=1, 2, 3\}$ as a $\Gl(n, \Hn)$-structure, and an almost quaternionic structure $Q\subset\Ed(TM)$ as a $\Gl(n, \Hn)\Sp(1)$-structure, respectively.  After fixing a $\Gl(n, \Hn)$- or a $\Gl(n, \Hn)\Sp(1)$-structure,  Riemannian metrics of the above type always exist and correspond to the reductions induced by  the maximal compact subgroups, i.e., the inclusions $\Sp(n)\subset\Gl(n, \Hn)$ and $\Sp(n)\Sp(1)\subset\Gl(n, \Hn)\Sp(1)$, respectively.

By comparison, the structures that we treat in this article arise from almost symplectic forms $\omega$ which are $H$-Hermitian, respectively  $Q$-Hermitian. Hence, it is  natural to refer to such $G$-structures by the terms  \textsf{almost hypercomplex skew-Hermitian structures}, denoted by  $(H, \omega)$, and  \textsf{almost  quaternionic skew-Hermitian structures}, denoted by $(Q, \omega)$, respectively, a terminology  which is also motivated by the discussion in  \cite{Harvey} of the eight types of inner product spaces.  It is also convenient to refer to such  non-degenerate  $\Sp(1)$-invariant  real-valued 2-forms by the term  \textsf{scalar 2-forms}. To provide some further motivation behind our considerations, recall that for an almost quaternionic manifold $(M, Q)$ the space $\Lambda^{2}T^*M$ of antisymmetric bilinear 2-forms admits the following $\Gl(n, \Hn)\Sp(1)$-equivariant decomposition into irreducible  submodules
	\[
		\Lambda^2T_{x}^*M = \Lambda^2_{\Im(\Hn)}T_{x}^*M  \oplus \Lambda^2_{\Re(\Hn)} T_{x}^*M\,,\quad \  \forall \ x\in M\,.
\]
In terms of bundles and the $\E\Hh$-formalism of Salamon   (see \cite{Salamon82}), we  may identify $\Lambda^2_{\Im(\Hn)} T_{x}^*M\cong [\Lambda^2\E]^*\otimes[S^2\Hh]^*$ and  this is the module where the  (local) K\"ahler forms in almost  qH geometry  take values. Hence, nowadays there is a rich variety of works related to such 2-forms, see  for example \cite{Cortes, Cabrera04, Cabrera} and the references therein.  The second module $\Lambda^2_{\Re(\Hn)} T_{x}^*M\cong [S^2\E]^*\otimes [\Lambda^2\Hh]^*$  is spanned by the scalar 2-forms and has not received the same attention yet.  Under the $\SO^*(2n)\Sp(1)$-action it takes the form $\Lambda^2_{\Re(\Hn)} T_{x}^*M\cong [S^2_0\E]^*\oplus\langle\omega_0\rangle$, where $\omega_0$ is the standard scalar 2-form on $[\E\Hh]\cong T_{x}M$. In this paper, we   will use the notion of scalar 2-forms  to facilitate a systematic treatment of the geometries under examination.  

From the viewpoint of holonomy theory, recall that by a classical result of Hano and Ozeki \cite{Hano},  a connected and simply connected smooth manifold $M^{4n}$  can always be equipped with an affine connection with torsion,  whose holonomy group will coincide with $\SO^*(2n)$ (or  with $\SO^*(2n)\Sp(1)$).  
 Therefore, there are proper examples of manifolds with non-integrable  $\SO^*(2n)$- and  $\SO^*(2n)\Sp(1)$-structures, which   make the examination of  such geometries a reasonable task.   
 In contrast,  for $\SO^*(2n)$ it is known by Bryant \cite{Bryant} that torsion-free affine connections with (irreducible) full holonomy group $\SO^*(2n)$ cannot exist. This is because Berger's first criterion fails for the Lie algebra $\fr{so}^*(2n)$.
   On the other hand,  there are well-known constructions of torsion-free connections with prescribed symplectic holonomy (see \cite{Bryant91, ChiMS1, ChiMS2})  and in particular there exists torsion-free connections with full holonomy $\SO^*(2n)\Sp(1)$ (see  \cite{Bryant, MS1, Schw, Schw2, CahS}).  In other words, $\SO^*(2n)\Sp(1)$ is a real non-symmetric Berger subgroup and  appears in  the list of \textsf{exotic holonomies}, see  \cite{Bryant} and see also \cite[Table 3]{MS1}. 
   In addition to these advances,  \v{C}ap and Sala\v{c} \cite{Cap, CapII} have recently  discussed  {special symplectic connections} within the more general framework of parabolic conformally symplectic structures.

   Since  much of our attention has been attracted by the general non-integrable case of $\SO^*(2n)$- and $\SO^*(2n)\Sp(1)$-structures,       
the approach  in this paper differs   from those which are mainly devoted to the torsion-free case (\cite{Bryant, MS1, Schw, Schw2,  CahS}).  In  particular, a main contribution of this work is based on the establishment of  the local geometry of  $\SO^*(2n)$- or $\SO^*(2n)\Sp(1)$-structures, via a  geometric approach  based   on  defining tensor fields, adapted frames, adapted connections, intrinsic torsion modules, minimal connections and normalization conditions.  This method allows us to proceed systematically with a differential-geometric treatment  of manifolds carrying such structures  and  highlight some parts of their intrinsic geometry. 

Let us summarize some basic properties of our $G$-structures, by fixing   an almost hypercomplex skew-Hermitian manifold $(M, H=\{I, J, K\}, \omega)$.  We show that such a manifold admits   three   pseudo-metric tensors $g_{I}, g_{J}, g_{K}$, which  are of signature $(2n, 2n)$ (but not of Norden type).  It turns out that any of $g_I, g_J, g_K$, or their linear combination, provides  an embedding $\SO^*(2n)\subset\U(n, n)\subset\SO(2n, 2n)$.  However, there is no natural way to pick a unique compatible metric among these embeddings, and in  particular for an almost  quaternionic skew-Hermitian structure $(Q, \omega)$,  these metrics exist only locally. 
We also introduce 
a symmetric 4-tensor $\Phi$, called the \textsf{fundamental 4-tensor}, given by
\[
\Phi:=g_{I}\odot g_{I}+g_{J}\odot g_{J}+g_{K}\odot g_{K}=\mf{Sym}(g_{I}\otimes g_{I}+g_{J}\otimes g_{J}+g_{K}\otimes g_{K})\,.
\]
For an almost  quaternionic skew-Hermitian structure $(Q, \omega)$, we show that  the fundamental tensor $\Phi$  is globally defined, hence it forms  the analogue of the fundamental 4-form on an almost  qH manifold. 
Note that $\Phi$ provides an equivalent definition of   almost quaternionic skew-Hermitian structures, while  a similar  characterization  occurs also in terms of  a \textsf{quaternionic skew-Hermitian form} $h$ on $M$, defined by
\[
 	h:=\omega\id_{TM}+ g_{I}I+g_{J}J+g_{K}K\in  \Gamma\big(T^{*}M\otimes T^{*}M\otimes \Ed(TM)\big)\,.
\]
These tensors   also occur on an almost hypercomplex skew-Hermitian manifold $(M, H, \omega)$, although we prove that  they are both stabilized by the larger group $\SO^*(2n)\Sp(1)$. They are important since  they have analogous applications as the fundamental 4-form on almost hH/qH  geometries, a fact which we thoroughly  investigate in the second part of this series.
 In this paper,  we use the Obata connection $\nabla^H$ related to an almost hypercomplex structure, or an Oproiu connection $\nabla^{Q}$ related to an almost quaternionic structure $Q$, to present   adapted $\SO^*(2n)$- and   $\SO^*(2n)\Sp(1)$-connections, denoted by $\nabla^{H, \omega}$ and $\nabla^{Q, \omega}$,  respectively.   
  
With the aim of exploring the underlying geometries by using these  adapted connections, we proceed by presenting  the intrinsic torsion of $\SO^*(2n)$- or $\SO^*(2n)\Sp(1)$-structures. This description is given  in a convenient way, in terms of   the $\E\Hh$-formalism of Salamon.  Therefore, we  compute the (second) Spencer cohomology $\mc{H}^{0, 2}(\fr{so}^*(2n))$ and $\mc{H}^{0, 2}(\fr{so}^*(2n)\oplus\fr{sp}(1))$ associated to the Lie algebras $\fr{so}^*(2n)$ and $\fr{so}^{*}(2n)\oplus\fr{sp}(1)$, respectively, and present the number of   algebraic types  of such geometric structures.  For  $n>3$ and for $\SO^*(2n)\Sp(1)$ we obtain five pure   types $\mc{X}_i$ $(i=1, \ldots,5)$ and a totality of $2^5$ algebraic types. 
 The case for $\SO^*(2n)$ is rather complicated, due to the appearance of some multiplicities.  However,  we  prove that  for $n>3$ the number of algebraic types of $\SO^*(2n)$-geometries is equal to $2^{10}$, and moreover we specify seven special $\Sp(1)$-invariant classes $\mc{X}_1,\ldots, \mc{X}_7$, determined in terms of $\Sp(1)$-invariant conditions. For the low-dimensional cases $n=2, 3$, we finally show that  both structures under investigation include some extra algebraic types.  We also obtain a characterization of geometries with intrinsic torsion a 3-form, or with intrinsic torsion of vectorial type. 
 
Another contribution of this first part is  the explicit description of certain normalization conditions, which allow us to regard the   adapted connections $\nabla^{H, \omega}$ and $\nabla^{Q, \omega}$  as  minimal  connections for our structures. This description is based on our intrinsic torsion decompositions and the theory that we establish about these two adapted connections. We then  rely on these minimal connections to answer the question of equivalence of $\SO^*(2n)$- or $\SO^*(2n)\Sp(1)$-structures. 
Note that in the generic case this is a  non-trivial task due to the multiplicities appearing in the torsion decomposition into irreducible submodules.

A final contribution of this work is the classification of symmetric spaces $K/L$ with $K$ semisimple, admitting a $K$-invariant torsion-free $\SO^*(2n)\Sp(1)$-structure.  To obtain this classification, we are based on previous results obtained by the second author in \cite{G13} and on the classification of pseudo-Wolf spaces given by Alekseevsky and Cort\'es in \cite{ACort}. We prove that only the following three series of symmetric spaces admit such an invariant torsion-free structure:
\[
\SO^*(2n+2)/\SO^*(2n)\U(1)\,,\quad 
\SU(2+p,q)/(\SU(2)\SU(p,q)\U(1))\,,\quad\ \Sl(n+1,\mathbb{H})/(\Gl(1,\mathbb{H})\Sl(n,\mathbb{H}))\,.
\]
Note that  the last two coset spaces belong to the list of pseudo-Wolf spaces, and  moreover the second family for $q=0$ gives rise to the compact Wolf space $\SU(2+p)/{\sf S}(\U(2)\times\U(p))$.

Let us now briefly introduce the main applications of the results obtained in this paper,  which are presented  in the second part \cite{CGWPartII} of this series. There, we 
\begin{itemize}
\item   use  the algebraic types $\mc{X}_{i_1\ldots i_j}$ of the intrinsic torsion to derive  1st-order integrability conditions  corresponding to  $\SO^*(2n)$- and $\SO^*(2n)\Sp(1)$-structures, for the underlying almost hypercomplex,  quaternionic, or symplectic geometries,   respectively;
\item  focus on the fundamental tensor $\Phi$ and examine its interaction with distinguished connections, such as the Obata connection $\nabla^{H}$, or the unimodular Oproiu connection $\nabla^{Q, \vol}$,  or an arbitrary almost symplectic connection $\nabla^{\omega}$, where $\omega$ is the scalar 2-form. This allows us to provide further geometric interpretations of  some classes $\mc{X}_{i_1\ldots i_j}$;
\item  provide some general constructions of such geometries and in particular illustrate many types of non-integrable $\SO^*(2n)$- and $\SO^*(2n)\Sp(1)$-structures via explicit examples.  
\item describe certain topological conditions which constrain the existence of such $G$-structures, and introduce the related ``spin structures''.
\end{itemize}
The content of our further investigation, which includes a description of the curvature invariants related to the $G$-structures under examination,  twistor constructions, and    other open tasks, is summarized in the last section of \cite{CGWPartII}. We plan to resolve some of these open problems in  the third part of this series.

The structure of the paper is given as follows. In Section \ref{linears} we recall some generalities of the Lie groups $\SO^\ast(2n)$ and $\SO^{\ast}(2n)\Sp(1)$ and next  introduce  the reader to the world of linear $\SO^\ast(2n)$-structures and linear $\SO^{\ast}(2n)\Sp(1)$-structures.  Based on the $\E\Hh$-formalism   we develop a theoretical framework for  linear $G$-structures  of this type, which allows us to  derive  the corresponding defining tensors  explicitly,   introduce the associated adapted bases and provide some details of  the symplectic point of view.   In  the Appendix  \ref{appendix}  we provide an alternative description of linear $\SO^\ast(2n)$-  and $\SO^{\ast}(2n)\Sp(1)$-structures in terms of right quaternionic vector spaces.  In Section \ref{sec2}   we introduce $\SO^\ast(2n)$-  and $\SO^{\ast}(2n)\Sp(1)$-structures on smooth manifolds and describe many of their basic features, as well  the fundamental 4-tensor $\Phi$ and the quaternionic skew-Hermitian form $h$.  In Section \ref{secintrinsic} we present the construction of the corresponding adapted connections,   analyze the associated intrinsic torsion modules and their related decompositions into irreducible submodules. In the following section, we present minimal connections and  the corresponding normalization conditions, and    solve the equivalence problem. Finally,  in Section \ref{tfexamples} we discuss torsion-free examples. We also present a few remarks about special symplectic holonomy.

\medskip
 \noindent {\bf Acknowledgments:}  
    I.C. and J.G.  acknowledge full support by Czech Science Foundation  via the project GA\v{C}R No.~19-14466Y.  H.W. thanks  the University of  Hradec Kr\'alov\'e for hospitality and 
    A. \v{C}ap for useful feedback. 
    
\section{Linear $\SO^\ast(2n)$-structures and linear $\SO^{\ast}(2n)\Sp(1)$-structures}\label{linears}
 In this section we introduce the notion of \textsf{linear  hypercomplex skew-Hermitian  structures} and \textsf{linear   quaternionic skew-Hermitian  structures}. These are linear geometric structures determined by the  action of the Lie groups $\SO^\ast(2n)$ and  $\SO^{\ast}(2n)\Sp(1)$, respectively, on some $4n$-dimensional vector space $V$.   We will show that similarly to the case of $\Sp(n)$- or $\Sp(n)\Sp(1)$-structures, there are  many invariant tensors that can be used to define such structures. 
 	
	\smallskip
Below we shall mainly use the $\E\Hh$-formalism  of Salamon (\cite{Salamon82, Salamon86}), instead of an arbitrary quaternionic vector space. However, before we begin with preliminaries about these groups and their representations,  it is convenient to include a short summary of  some classical   definitions.
\begin{defi}\label{basicsright}
	\textsf{1)} A \textsf{linear hypercomplex structure} $H$ on $V$ is a triple $\{I, J, K\}$ of linear complex structures $I, J, K\in \Ed(V)$ satisfying the quaternionic relations, i.e., $I^{2}=J^2=K^2=-\id=IJK$.\\
		\textsf{2)}  A \textsf{linear quaternionic structure}  $Q$ on $V$ is a 3-dimensional subspace of $\Ed(V)$ spanned by an arbitrary linear hypercomplex structure $H$, i.e., $Q=\langle H\rangle$. In this case $H$ is called an \textsf{admissible basis} of $Q$ and it is easy to see that any two admissible bases of $Q$ are related by an element in $\SO(3)$. \\
\textsf{3)}   
	 For the algebra $\Hn$ of quaternions we will denote by $\Re(\Hn):=\R$ (respectively, $\Im(\Hn):=\sp(1)$) its real (respectively, imaginary) part. We have the corresponding Lie group decomposition $\Hn^{\times}=\R^{\times}\Sp(1)$, where as usual we set $\Hn^{\times}:=\Hn\backslash\{0\}$ and $\R^{\times}:=\R\backslash\{0\}$. A choice of an admissible basis $H=\{I, J, K\}$ of a linear quaternionic structure $Q$ on $V$ provides an isomorphism $\sp(1) \cong Q,$  and moreover the diffeomorphisms $\Sp(1)\cong\{\mu_0\id_V+\mu_1I+\mu_2J+\mu_3K : \mu_0^2+\mu_1^2+\mu_2^2+\mu_3^2=1\}\cong\Ss^3$ and 
\[
\Sp(1)\cap\fr{sp}(1)\cong\Ss(Q):=\{\mu_1 I+\mu_2 J+\mu_3 K : \mu_1^{2}+\mu_2^{2}+\mu_3^{2}=1\}\,.
\]	 
So $\Ss(Q)$ is the space of complex structures in $Q$, i.e.,  the space of endomorphisms $\J\in Q$ such that $\J^2=-\id$. Note that $\Ss(Q)\cong\Ss^2$, where $\Ss^n$ will denote the $n$-sphere.
	 \\
	\textsf{4)}   
	 Given a linear complex structure $J\in\Ed(V)$, a real-valued bilinear form $f$ on $V$ is called  \textsf{Hermitian} if $f(Jx, Jy)=f(x, y)$ for any $x, y\in V$.\\ 
	\textsf{5)} 
Given a linear hypercomplex structure $H=\{J_{a} : a=1, 2, 3\}$ on $V$, a real-valued bilinear form $f$ on $V$ is called 
\textsf{Hermitian with respect to $H$} ($H$-Hermitian for short), if $f$ is a Hermitian bilinear form with respect to $J_{a}$ for any $a=1, 2, 3$, and 
 \textsf{Hermitian with respect to $Q$} ($Q$-Hermitian for short),  if $f(\J x, \J y)=f(x, y)$ for any $\J\in \Ss(Q)$ and $x, y\in V$. \\
\textsf{6)}   Let $V$ be a vector space with a linear quaternionic structure $Q=\langle H\rangle$, where $H$ is an admissible basis providing the isomorphism $\Hn\cong \R\oplus Q$. Traditionally,  a  \textsf{quaternionic skew-Hermitian form} is defined to be a $\R$-bilinear map $h : V\times V\to\Hn$, satisfying  
\[
h(xp, yq)=ph(x, y)\bar{q}, \quad h(x, y)=-\overline{h(y, x)}, \quad \forall \ x, y\in V, \ p, q\in\Hn.
\]
Note that the first condition actually says that the form $h:V\times V\to \Hn$ does {\it not} depend on the admissible basis $H$ providing the isomorphism $\Hn\cong \R\oplus Q$.   
In particular,   $h$ is a \textsf{sesquilinear form} and as we will show later (see Proposition \ref{usefrelskewherm}), up to (quaternionic) linear automorphisms, any finite dimensional quaternionic vector space admits a unique non-degenerate quaternionic  skew-Hermitian form, see also \cite[Chapter 2]{Harvey}.
\end{defi}

\subsection{Preliminaries about the Lie group $\SO^{\ast}(2n)$}
	We are mainly interested in cases with $n\geq 2$, and then the Lie group $\SO^{\ast}(2n)$ is a non-compact real form of $\SO(2n, \C)$, of real dimension $n(2n-1)$, which is semisimple for $n=2$ and simple for $n\geq 3$ (see \cite{Hel}). We shall refer to   $\SO^{\ast}(2n)$ by the term \textsf{quaternionic real form}. Be aware that there exist many different notations for the Lie group $\SO^{\ast}(2n)$, for example  $\SO(n, \Hn)$ in \cite{Bryant, Schw, Schw2}, $\U_n^*(\Hn)$ in \cite{Oni}, or $\mathsf{Sk}(n,\Hn)$ in \cite{Harvey}. Traditionally, $\SO^{\ast}(2n)$ is defined as the stabilizer of a quaternionic  skew-Hermitian form $h$, as above,  an interpretation that we will discuss in detail below.  If we view  $\SO(2n, \C)$ as the Lie group of complex linear transformations preserving  the standard complex Euclidean metric on $\E:=\C^{2n}$, then the real form $\SO^{\ast}(2n)$ of $\SO(2n, \C)$ is the fixed point set of the following involution:
\[
\sigma(\phi)=\begin{pmatrix}
		0 & -\id_{\C^n}\\
		\id_{\C^n} & 0\\
	\end{pmatrix}^{-1}
    (\bar{\phi}^t)^{-1} 
   \begin{pmatrix}
		0 & -\id_{\C^n}\\
		\id_{\C^n} & 0\\
	\end{pmatrix},
 \]
for $\phi\in \SO(2n, \C)$, i.e.,
\[
\SO^{\ast}(2n)=\{\phi\in \SO(2n, \C) : \sigma(\phi)=\phi\}\,.
\]
In this way $\E$ becomes the standard representation of $\SO^{\ast}(2n)$.  The Lie algebra $\fr{so}^{\ast}(2n)$  of $\SO^{\ast}(2n)$ is represented  by the following endomorphisms of $\E$:
\[
\left( 
    \begin{smallmatrix}
 Z_1 & -Z_2\\
 \overline{Z_2}&\overline{Z_1}
  \end{smallmatrix}
  \right),
\]
for $Z_1,Z_2$ complex $(n\times n)$-matrices satisfying $Z_1^t=-Z_1$ and $Z_2^t=\overline{Z_2}$, see \cite{Hel}.  There are other presentations of $\fr{so}^{\ast}(2n)$ related to identifications of $\E$ as a right quaternionic vector space, which we review in    Appendix \ref{appendix} (see Proposition \ref{gradbas}).

Recall now that  $\SO^\ast(2n)$ is connected, with $\pi_1(\SO^\ast(2n))=\Z$.      Therefore,  $\SO^*(2n)$ is not simply connected and there are further Lie groups with Lie algebra $\fr{so}^{\ast}(2n)$, which we describe in \cite[Appendix A]{CGWPartII}. 
In fact, since the restricted root system of   $\fr{so}^{\ast}(2n)$ depends on the parity of the quaternionic dimension $n$, there are two related \textsf{Satake diagrams}, which we present below (a detailed exposition of real forms and Satake diagrams can be found in Onishchik's book \cite{Oni}, see also \cite[pp.~214-223]{CS}).
{\small \[
\begin{tabular}{rcrc}
\begin{picture}(15,45)(0,0)
\put(2, 25){\makebox(0,0){$n=2m  \ :$}}
\end{picture} & 
 \begin{picture}(160,40)(-15,-23)
\put(0, 0){\circle*{4}}
\put(0,8){\makebox(0,0){{\tiny$\Lambda_1$}}}
\put(2, 0){\line(1,0){14}}
\put(18, 0){\circle{4}}
\put(18.5,8){\makebox(0,0){{\tiny$\Lambda_2$}}}
\put(20, 0){\line(1,0){13.5}}
\put(35.5, 0){\circle*{4}}
\put(36.5,8){\makebox(0,0){{\tiny$\Lambda_3$}}}
 \put(37.5, 0){\line(1,0){12}}
\put(59.2, 0){\makebox(0,0){$\ldots$}}
\put(67, 0){\line(1,0){13}}
 \put(82, 0){\circle*{4}}
\put(84, 0){\line(1,0){12}}
\put(100, 1){\line(2,1){10.5}}
\put(100, -1){\line(2,-1){11}}
\put(98, 0){\circle{4}}
\put(98,-7){\makebox(0,0){{\tiny$\Lambda_{n-2}$}}}
\put(85.5, 8){\makebox(0,0){{\tiny{$\Lambda_{n-3}$}}}}
\put(112.5, -7){\circle{4}}
\put(120, 13){\makebox(0,0){{\tiny$\Lambda_{n-1}$}}}
\put(111, -16){{\tiny$\Lambda_{n}$}}
\put(112.5, 7){\circle*{4}}
\end{picture}  & \quad 
\begin{picture}(15,45)(0,0)
\put(12, 25){\makebox(0,0){$n=2m+1  \ :$}}
\end{picture} \quad\quad\quad & 
 \begin{picture}(150,40)(-15,-23)
\put(0, 0){\circle*{4}}
\put(0,8){\makebox(0,0){{\tiny$\Lambda_1$}}}
\put(2, 0){\line(1,0){14}}
\put(18, 0){\circle{4}}
\put(18.5,8){\makebox(0,0){{\tiny$\Lambda_2$}}}
\put(20, 0){\line(1,0){13.5}}
\put(35.5, 0){\circle*{4}}
\put(36.5,8){\makebox(0,0){{\tiny$\Lambda_3$}}}
 \put(37.5, 0){\line(1,0){12}}
\put(59.2, 0){\makebox(0,0){$\ldots$}}
\put(67, 0){\line(1,0){13}}
 \put(82, 0){\circle{4}}
\put(84, 0){\line(1,0){12}}
\put(100, 1){\line(2,1){10.5}}
\put(100, -1){\line(2,-1){10.5}}
\put(98, 0){\circle*{4}}
\put(98,-7){\makebox(0,0){{\tiny$\Lambda_{n-2}$}}}
\put(85.5, 8){\makebox(0,0){{\tiny{$\Lambda_{n-3}$}}}}
\put(112.5, -7){\circle{4}}
\put(120, 13){\makebox(0,0){{\tiny$\Lambda_{n-1}$}}}
\put(111, -16){{\tiny$\Lambda_{n}$}}
\put(112.5, 7){\circle{4}}
\put(113, -5){\vector(0,1){9.5}}
\put(113, -4){\vector(0, -1){1.5}}
\end{picture}
\end{tabular}
\]}
  Therefore, for $n=2,3,4$ there are special isomorphisms with classical Lie algebras, and we present the related details in the second part of this work.   However, since $\E$ is not the standard representation of these classical Lie algebras, this viewpoint does not relate these geometric structures with geometric structures which have been studied earlier. 
   
\subsection{The $\E\Hh$-formalism adapted to $\SO^*(2n)$ and $\SO^*(2n)\Sp(1)$}\label{ehformsec}
The most convenient way to  visualize a \textsf{linear quaternionic structure} is by using the $\E\Hh$-formalism of Salamon \cite{Salamon82, Salamon86}. The $\E\Hh$-formalism is usually used for the standard representation $\E$ of $\Gl(n,\Hn)$ or $\Sp(n)$ (or $\Sp(p,q)$), and we adapt this formalism to $\SO^{\ast}(2n)$.  So, let us  denote by $\Hh$ the standard representation of $\Sp(1)$ on $\C^2$. This is of quaternionic type, and the same holds for the $\SO^{\ast}(2n)$-representation $\E$.  Thus, the maps defined by
 \[
\epsilon_{\E} : \left( 
    \begin{smallmatrix}
a \\
b 
  \end{smallmatrix}
  \right)
  \mapsto 
  \left( 
    \begin{smallmatrix}
-\bar{b} \\
\bar{a}
  \end{smallmatrix}
  \right)\,,\quad \epsilon_{\Hh} : \left( 
    \begin{smallmatrix}
a \\
b 
  \end{smallmatrix}
  \right)
  \mapsto 
  \left( 
    \begin{smallmatrix}
-\bar{b} \\
\bar{a}
  \end{smallmatrix}
  \right),
\]
for $a,b\in \C^n,a,b\in \C$, respectively, are  complex anti-linear involutions of $\E$ and $\Hh$, which commute with the actions of $\SO^{\ast}(2n)$ and $\Sp(1)$, respectively.  Let us also denote the group $\SO^{\ast}(2n)\times_{\Z_2}\Sp(1)=(\SO^{\ast}(2n)\times\Sp(1))/{\Z_2}$ by 
\[
\SO^{\ast}(2n)\Sp(1)\,,
\]
which is the image  of the product $\SO^{\ast}(2n)\times\Sp(1)$ in $\Ed(\E\otimes_\C \Hh)$, via the tensor product representation.   
 The standard quaternionic representation of $\SO^{\ast}(2n)\Sp(1)$ is the real form $[\E\Hh]$  inside $\E\otimes_\C \Hh$, fixed by the real structure $\epsilon_{\E}\otimes \epsilon_{\Hh}$. Note that the $\E\Hh$-formalism extends this description to all tensor products of $\E$ and $\Hh$,  where products of $\epsilon_{\E}$ and $\epsilon_{\Hh}$ still define a real structure, for which we shall maintain the same notation   $[\ ]$. 
 
 Now, since $\E$ is of quaternionic type, we have
\[
N_{\fr{gl}(\E)}(\fr{so}^{\ast}(2n))=\fr{so}^{\ast}(2n)\oplus \fr{sp}(1)\oplus\mathbb{R}\,, \quad C_{\fr{gl}(\E)}(\fr{so}^{\ast}(2n))=\fr{sp}(1)\oplus\mathbb{R}\,.
\] 
In particular, the  $\fr{sp}(1)$-action commutes with $\SO^{\ast}(2n)$ and defines a linear quaternionic structure   on $[\E\Hh]$, i.e., there is admissible basis $H=\{J_1,J_2,J_3\}$ of $\sp(1)$ such that 
\[
J_1^{2}=J_2^2=J_3^2 =-\Id=J_1J_2J_3\,.
\]
 
In this series of papers  we will maintain  the following conventions. 
 
\smallskip
 \noindent{\bf{\textsf{ Conventions.}}}  \textsf{(i)}   $Q_0$ will denote the  quaternionic structure  on $[\E\Hh]$, defined by the $\fr{sp}(1)$-action. \\
 \textsf{(ii)} The notation $R(\alpha)$ will encode a  complex irreducible module corresponding to the highest   weight $\alpha$, for some of the Lie algebras of $\SO^\ast(2n)$ or   $\SO^\ast(2n)\Sp(1)$. The letter  $\theta$ will be used for the  fundamental weight of $\sp(1)$, and $\pi_k$, $k=1,\dots,2n$ will be the fundamental weights of $\so^\ast(2n)$. For instance, for $n>2$ we have    $\E\Hh=R(\pi_1)\otimes_{\C} R(\theta)$.  It is also useful to mention the following $\SO^*(2n)$-equivariant isomorphisms
\begin{eqnarray*}
R(\pi_k)&=&\Lambda^{k}\E \ \text{for}  \ k<n-1\,, \\
   R(k\pi_1)&=&S^{k}_{0}\E\cong S^{k}\E/ S^{k-2}\E\,.
\end{eqnarray*}
Here, $S^{k}_{0}\E$ denotes the trace free part of $S^{k}\E$,   for any $k>0$, see next section   for more details. 
For $k=\text{even}$, the complex representation $R(k\pi_1)$ is of real type, and for $k=\text{odd}$ we see that $R(k\pi_1)$ is of quaternionic type, similarly to the $\Sp(n)$-action (or $\Sp(p, q)$-action).  Note however that $\Lambda^{k}\E$ is only irreducible for $\SO^*(2n)$, but reducible in the $\Sp(n)$-case.   For modules appearing in low-dimensional cases we have some exceptions which are indicated in Table \ref{Table1}, together with the corresponding dimensions (see also \cite{FH}).

    \begin{table}[ht]
\centering
\renewcommand\arraystretch{1.5}

\begin{tabular}{l|c|c|c|c}
& $n>4$ & $n=4$& $n=3$& $n=2$\\
\hline

$\E$ & $R(\pi_1)$ & $R(\pi_1)$& $R(\pi_1)$& $R(\pi_1+\pi_2)$\\
$\dim_{\C}$ & $2n$ & $8$ & $6$ & $4$\\
\hline
$\Lambda^2\E$ & $R(\pi_2)$ & $R(\pi_2)$& $R(\pi_2+\pi_3)$& $R(2\pi_1)\oplus R(2\pi_2)$\\
$\dim_{\C}$ & $n(2n-1)$ & $28$ & $15$ & $6$\\
\hline
$S^2_0\E$ & $R(2\pi_1)$ & $R(2\pi_1)$& $R(2\pi_1)$& $R(2\pi_1+2\pi_2)$\\
$\dim_{\C}$ & $2n^2+n-1$ & $35$ & $20$ & $9$ \\
\hline
$\K$& $R(\pi_1+\pi_2)$ & $R(\pi_1+\pi_2)$& $R(\pi_1+\pi_2+\pi_3)$& $R(\pi_1+3\pi_2)\oplus R(3\pi_1+\pi_2)$\\
$\dim_{\C}$ & $\frac{8}{3}(n^{3}-n)$  & $160$ & $64$ & $16$ \\
\hline
$\Lambda^3\E$& $R(\pi_3)$ & $R(\pi_3+\pi_4)$& $R(2\pi_2)\oplus R(2\pi_3)$& $R(\pi_1+\pi_2)$\\
$\dim_{\C}$ &  $ \frac{2n}{3} (2n-1) (n-1)$   & $56$ & $20$ & $4$ \\
\hline
$S^3_0\E$& $R(3\pi_1)$ & $R(3\pi_1)$& $R(3\pi_1)$& $R(3\pi_1\oplus 3\pi_2)$ \\
$\dim_{\C}$ &  $ \frac{2n}{3} (2n-1) (n+2)$ & $112$  &  $50$ & $16$ \\
\hline
\end{tabular}
\vspace{0.5cm}
\caption{\small Some $\SO^*(2n)$-modules for low  quaternionic dimensions.}\label{Table1}
\end{table}

Let us now recall the following basic result, which is useful for our considerations  (see also \cite{Salamon86}).
 \begin{lem}\label{invarena}
\textsf{1)} The module $[\Hh\Hh]^*$   admits  a 1-dimensional $\Sp(1)$-submodule $[\Lambda^2\Hh]^*$ spanned by a real form of the complex-valued volume form on $\Hh$. The latter is defined by
$
\omega_{\Hh}( \left( 
    \begin{smallmatrix}
a \\
b
  \end{smallmatrix}
  \right),\left( 
    \begin{smallmatrix}
c \\
d
  \end{smallmatrix}
  \right))
:=\frac12(ad-bc)\,,
$
for  $a,b,c,d\in \C$, and satisfies $\omega_{\Hh}(\epsilon_{\Hh} x,\epsilon_{\Hh} y)=\overline{\omega_{\Hh}(x,y)}$. In particular, $\omega_{\Hh}$ establishes   an isomorphism $\Hh^*\cong \Hh$.\\
\textsf{2)}   The  linear quaternionic structure $Q_0$ on $[\E\Hh]$ is isomorphic to $[S^2\Hh]^*$, that is $[S^2\Hh]^*\cong_{\omega_{\Hh}}Q_0=\sp(1)$. 
Moreover, $[\Hh\Hh]^*\cong_{\omega_{\Hh}} [\Hh\Hh^*]=\R\oplus  Q_0=\R\oplus \sp(1)$, and a choice of an admissible basis for $Q_0$  provides  an isomorphism $[\Hh\Hh]^*\cong \Hn$.\\
\textsf{3)} The module $[S^2\E]^*$     admits a 1-dimensional $\SO^{\ast}(2n)$-submodule spanned by a real form of the complex-valued metric on $\E$. The latter is defined by
$
g_{\E}(\left( 
    \begin{smallmatrix}
a \\
b
  \end{smallmatrix}
  \right),\left( 
    \begin{smallmatrix}
c \\
d
  \end{smallmatrix}
  \right)):=a^tc+b^td\,,
$
for $a,b,c,d\in \C^n$, and  satisfies $g_{\E}(\epsilon_{\E} x,\epsilon_{\E} y)=\overline{g_{\E}(x,y)}$. In particular, $g_{\E}$ establishes an isomorphism $\E^*\cong \E$.
\end{lem}

Let us point out some further similarities and differences between the $\E\Hh$-formalism adapted to $\Gl(n,\Hn)$- or $\Sp(p,q)$-actions in comparison to $\SO^{\ast}(2n)$-actions.

\begin{itemize}
\item  The $\SO^{\ast}(2n)$-action can be naturally extended to a $\Gl(n,\Hn)$-action, however the real form of $g_{\E}\in S^2\E^*$ will be no longer invariant.  Thus in this case there is no canonical isomorphism $\E^*\cong \E$. Nevertheless, independently of the actions, non-degenerate  elements of $[S^2\E]^*$ induce  skew-symmetric $\R$-bilinear forms (2-forms) on $[\E\Hh]$. Under the $\SO^{\ast}(2n)$- and  $\SO^{\ast}(2n)\Sp(1)$-action, we will study these 2-forms in a great detail in next subsection.
\item The invariant tensor of the $\Sp(p,q)$-action providing  the isomorphism $\E^*\cong \E$,  is   an element  of $\Lambda^2\E^*$   instead of $S^2\E^*$, which we may denote by $\omega_{\E}$, see also Remark \ref{zerorem}. In general, independently of the actions, non-degenerate elements of $[\Lambda^2\E]^*$  induce qH pseudo-Euclidean metrics on $[\E\Hh]$.
\item When decomposing low order tensor products of $\E$ with respect to $\Sp(p, q)$, it is customary to introduce the $\Sp(p, q)$-module $\K$, see      \cite{Swann2}. However, note that whenever we decompose  low order tensor products of $\E$ with respect to $\SO^*(2n)$,   the $\SO^*(2n)$-module $\K$ represents a different submodule in the corresponding decomposition. Nevertheless, $\K$ has the same dimension as in the $\Sp(p, q)$-case  and also the same expression in terms of fundamental weights (but for different Lie algebras).
\end{itemize}


\subsection{Invariants of actions of $\SO^{\ast}(2n)$ and $\SO^{\ast}(2n)\Sp(1)$}\label{defEH}
  Let us now  discuss important representations of  $\SO^{\ast}(2n)$ and $\SO^{\ast}(2n)\Sp(1)$ for $n>1$, in detail.   We first treat  the representations indicated in the previous section, and in particular in the above Lemma \ref{invarena}.

 \begin{lem}\label{IJKstandard}
The module $[\E\Hh]$ consists of the following elements of $\E\otimes_\C \Hh$:
\begin{align*}
a:= \left( 
    \begin{smallmatrix}
a \\
0
  \end{smallmatrix}
  \right)\otimes \left( 
    \begin{smallmatrix}
1 \\
0
  \end{smallmatrix}
  \right)+\left( 
    \begin{smallmatrix}
0 \\
\bar{a}
  \end{smallmatrix}
  \right)\otimes \left( 
    \begin{smallmatrix}
0 \\
1
  \end{smallmatrix}
  \right),\\
  bj:= \left( 
    \begin{smallmatrix}
0 \\
-b
  \end{smallmatrix}
  \right)\otimes \left( 
    \begin{smallmatrix}
1 \\
0
  \end{smallmatrix}
  \right)+ \left( 
    \begin{smallmatrix}
    \bar{b}\\
0
  \end{smallmatrix}
  \right)\otimes \left( 
    \begin{smallmatrix}
0 \\
1
  \end{smallmatrix}
  \right),
\end{align*}
for $a,b\in \C^n$. The  linear quaternionic structure $Q_{0}$ on $[\E\Hh]$ can be defined by the following admissible basis $H_{0}=\{\mc{J}_{a}\in\Ed([\E\Hh]) : a=1, 2, 3\}$:
\[
	 \mc{J}_{1}=i\id_{\C^{2n}} \otimes\begin{bmatrix}
		1 & 0\\
		0 & -1\\
	\end{bmatrix}\,,\quad \mc{J}_{2}=\id_{\C^{2n}} \otimes\begin{bmatrix}
		0 & -1\\
		1 & 0\\
	\end{bmatrix}\,,\quad \mc{J}_{3}=i\id_{\C^{2n}} \otimes\begin{bmatrix}
		0 & -1\\
		-1 & 0\\
	\end{bmatrix}
	\]
such that
\begin{gather*}
(a_1+a_2\mc{J}_1+a_3\mc{J}_2+a_4\mc{J}_3)(a+bj)=(a_1+a_2i+a_3j+a_4k)(a+bj)
\end{gather*}
for $a,b\in \C^n$, where on the right-hand side we view $a+bj\in [\E\Hh]$ as  an element of $\Hn^n$.
\end{lem}
\begin{rem}
\textnormal{Let us emphasize that the identification with $\Hn^n$ is possible only after the choice of  an admissible basis,  and that  in general  $a+bj\in [\E\Hh]$ is just notation.  However, there is also an identification $\bar{a}-bj\in [\E\Hh]$ with $a+bj \in \Hn^n$ as  right quaternionic vector spaces, and for convenience the details related to the view of  $\Hn^n$ as a left  or right quaternionic vector space  are described in Appendix \ref{appendix}. So, let us focus  on $[\E\Hh]$ from now on.}
\end{rem}
 \begin{proof}
It is a simple observation that the real structure on $[\E\Hh]$ defined by $\epsilon_{\E}\otimes\epsilon_{\Hh}$ fixes the elements $a, bj$, and that the real dimension is $4n$, as  it is expected. Now, the chosen representation of $\Sp(1)$ on $\Hh$ suggests using the following elements of $\Sp(1)\cap \sp(1)= S(Q)$ to define an admissible  basis $\{\mc{J}_{a} : a=1, 2, 3\}$:
\[
	 \begin{bmatrix}
		i & 0\\
		0 & -i\\
	\end{bmatrix},\quad \begin{bmatrix}
		0 & -1\\
		1 & 0\\
	\end{bmatrix},\quad \begin{bmatrix}
		0 & -i\\
		-i & 0\\
	\end{bmatrix}.
\]
Since we can always move $i$ from $\Hh$ to $\E$, it is simple computation to check that such endomorphisms preserve $[\E\Hh]$ and take the claimed form $\{\mc{J}_1, \mc{J}_2, \mc{J}_3\}$. Also,  it is not hard to prove that the elements in $H_{0}:=\{\mc{J}_{a}, a=1, 2, 3\}$ correspond to the left multiplication by $i, j, k$, respectively, when we view $a+bj$ as an element of $\Hn^n$.
 \end{proof}
 
 \begin{defi}
We shall refer to the admissible basis $H_{0}$ of the linear quaternionic structure $Q_{0}$ on  $[\E\Hh]$ presented in Lemma \ref{IJKstandard} be the term \textsf{standard admissible basis} of $Q_{0}$ on $[\E\Hh]$. 
\end{defi}

 Next, we prove that the elements $g_{\E}, \omega_{\Hh}$ introduced before, can be combined into a $\SO^{\ast}(2n)\Sp(1)$-invariant linear symplectic form on $[\E\Hh]$.
 
 \begin{prop}\label{usefrel}
 The expression $\omega_{0}:=g_{\E}\otimes \omega_{\Hh}$ defines  a  non-degenerate $\SO^{\ast}(2n)\Sp(1)$-invariant real 2-form  on $[\E\Hh]$, explicitly given by
\[
\omega_0(a+bj,c+dj)=\Re(a^t\bar{d}-b^t\bar{c})\,,\quad \forall \ a+bj, \ c+dj\in [\E\Hh]\,.
\]
\end{prop}
\begin{proof}
We   compute 
\begin{eqnarray*}
\omega_0(a+bj, c+dj)&:=&g_{\E}\otimes \omega_{\Hh}(\left( 
    \begin{smallmatrix}
a \\
-b
  \end{smallmatrix}
  \right)\otimes \left( 
    \begin{smallmatrix}
1 \\
0
  \end{smallmatrix}
  \right)+\left( 
    \begin{smallmatrix}
 \bar{b} \\
\bar{a}
  \end{smallmatrix}
  \right)\otimes \left( 
    \begin{smallmatrix}
0 \\
1
  \end{smallmatrix}
  \right),\left( 
    \begin{smallmatrix}
c \\
-d
  \end{smallmatrix}
  \right)\otimes \left( 
    \begin{smallmatrix}
1 \\
0
  \end{smallmatrix}
  \right)+\left( 
    \begin{smallmatrix}
 \bar{d} \\
\bar{c}
  \end{smallmatrix}
  \right)\otimes \left( 
    \begin{smallmatrix}
0 \\
1
  \end{smallmatrix}
  \right))\\
  &=&\frac12(a^t\bar{d}-b^t\bar{c}+\bar{a}^td-\bar{b}^tc)=\Re(a^t\bar{d}-b^t\bar{c}).
\end{eqnarray*}
This is clearly a 2-form, which is $\SO^*(2n)\Sp(1)$-invariant and non-degenerate by definition.
\end{proof}

Let us look for further invariants. We start with  the space $[\E\Hh]^*\otimes_\R [\E\Hh]^*$ of   real-valued bilinear forms on $[\E\Hh]$, for which we have the following   equivariant decompositions into submodules, which are irreducible for $n>2$.
 \begin{prop}\label{modules}
 
\[
\renewcommand\arraystretch{1.3}
\begin{tabular}{l l l}
$\Lambda^2[\E\Hh]^*$  &  $\cong_{\Gl(n,\Hn)\Sp(1)}$ & $[\Lambda^2 \E]^*\otimes_\R \sp(1)\oplus [S^2\E]^*\otimes_\R [\Lambda^2\Hh]^*$\,, \\  
                                      & $\cong_{\SO^{\ast}(2n)\Sp(1)}$ & $[\Lambda^2 \E]^*\otimes_\R \sp(1) \oplus [S^2_0\E]^*\oplus \langle\omega_0 \rangle$\,,\\  
                                      & $\cong_{\Gl(n, \Hn)}$ & $3[\Lambda^2 \E]^*  \oplus [S^2\E]^*$\,,  \\  
                                      & $\cong_{\SO^{\ast}(2n)}$ & $3[\Lambda^2 \E]^*  \oplus [S^2_0\E]^*\oplus \langle\omega_0 \rangle$\,. \\ 
$S^2[\E\Hh]^*$              & $\cong_{\Gl(n,\Hn)\Sp(1)}$ & $[\Lambda^2\E]^*\otimes_\R [\Lambda^2\Hh]^*\oplus [S^2\E]^*\otimes_\R \sp(1)$\,, \\ 
                                       & $\cong_{\SO^{\ast}(2n)\Sp(1)}$ & $[\Lambda^2\E]^*\oplus [S^2_0\E]^*\otimes_\R \sp(1)\oplus \langle\omega_0 \rangle \otimes_\R \sp(1)$\,, \\ 
                                         & $\cong_{\Gl(n, \Hn)}$  & $[\Lambda^2\E]^*\oplus 3[S^2\E]^*$\,, \\                                        & $\cong_{\SO^{\ast}(2n)}$ & $[\Lambda^2\E]^*\oplus 3[S^2_0\E]^* \oplus 3\R$\,. 
                                                                                \end{tabular}
                                       \]
\end{prop}
\begin{rem}\label{zerorem}
\textnormal{
\textsf{1)}  These results are obtained by applying basic representation theory, see for example \cite{FH}, or the summary given in \cite{CGW}.   The reader may feel more familiar with  the decompositions for the  $\Gl(n,\Hn)\Sp(1)$- and $\Gl(n, \Hn)$-actions,  see for example  \cite{Salamon86, Swann, Cabrera}.\\   
\textsf{2)} The $\Gl(n, \Hn)\Sp(1)$-module $[\Lambda^2\E]^*\subset S^2[\E\Hh]^*$ is spanned by the qH pseudo-Euclidean metrics on $[\E\Hh]$ of signature $(4p,4q)$, which  have a standard representative   given by  (see \cite{Salamon86, ACort}) 
\[
( \ , \ )_{0}:=\omega_{\E}\otimes\omega_{\Hh}\,.
\]
  At the same time, $[\Lambda^2\E]^*$ can be viewed as  the $\Gl(n, \Hn)$-module of the hH pseudo-Euclidean metrics on $[\E\Hh]$ of signature $(4p,4q)$.  The corresponding \textsf{fundamental   2-forms}  (skew-symmetric non-degenerate $\R$-valued bilinear forms)
\[
  \omega_{a}(\cdot\,,\cdot):=(\cdot\,, \mc{J}_{a}\cdot)_{0}\,,\quad a=1, 2, 3,  
\]
where  $H_{0}=\{\mc{J}_{a} : a=1, 2, 3\}$  is  the standard admissible basis,  are three distinguished elements of the $\Gl(n, \Hn)$-module $3[\Lambda^2 \E]^*$, inducing a $\Sp(1)$-submodule of the $\Gl(n, \Hn)\Sp(1)$-module $[\Lambda^2 \E]^*\otimes_\R [S^2\Hh]^*$ (independently of the choice of an admissible basis).\\
\textsf{3)} The irreducible $\Gl(n, \Hn)\Sp(1)$-module $[S^2\E]^*\otimes_\R [\Lambda^2\Hh]^*\subset \Lambda^2[\E\Hh]^*$, which is also an irreducible $\Gl(n, \Hn)$-module,
   consists of 2-forms of a distinguished type, playing an important role in our considerations related to linear $\SO^*(2n)$- or $\SO^{*}(2n)\Sp(1)$-structures. For such 2-forms,  we shall use the following terminology.}
\end{rem}
\begin{defi}\label{scalardef}
\textsf{1)} Consider $[\E\Hh]$ with the standard admissible basis $H_{0}$. Then,  a 2-form $\omega$ on $[\E\Hh]$ is called a \textsf{scalar 2-form} (with respect to $H_{0}$), if $\omega$ is  non-degenerate and $H_{0}$-Hermitian. \\
\textsf{2)} Consider $[\E\Hh]$ with the standard linear quaternionic structure  $Q_{0}$. Then, a  2-form $\omega$ on $[\E\Hh]$ is called a \textsf{scalar 2-form} (with respect to $Q_{0}$), if $\omega$ is  non-degenerate and $Q_{0}$-Hermitian. 
  \end{defi}
  Next we shall refer to the scalar 2-form $\omega_0$ on $[\E\Hh]$ introduced in Proposition \ref{usefrel} via the term  \textsf{standard scalar 2-form} on $[\E\Hh]$ (with respect to $Q_{0}=\langle H_0\rangle$). 
As an immediate consequence of  Proposition \ref{usefrel}  we conclude that 
\[
\Aut\big(H_{0}, \omega_0\big)=\SO^*(2n)\,,\quad \Aut\big(Q_{0}=\langle H_{0}\rangle, \omega_0\big)=\SO^*(2n)\Sp(1)\,,
\] 
respectively. Let us now provide a useful  characterization of scalar 2-forms.  

\begin{prop}\label{usefrel1}
The following conditions are equivalent for a non-degenerate 2-form $\omega$   on $[\E\Hh]${\rm :}
\begin{enumerate}
\item[$\sf{(1)}$] $\omega$ is conjugated to $\omega_0$ by an element of $\Gl(n,\Hn)$.
\item[$\sf{(2)}$] $\omega\in [S^2\E]^*=[S^2_0\E]^*\oplus\langle\omega_{0}\rangle$.  
\item[$\sf{(3)}$] $\omega$ is $\Sp(1)$-invariant, i.e., $\omega(A\cdot\,,A\cdot)=\omega(\cdot\,,\cdot)$, for all $A\in \Sp(1)$.
\item[$\sf{(4)}$] $\omega(A\cdot\,,\cdot)+\omega(\cdot\,,A\cdot)=0$, for all $A\in \sp(1)$.
\item[$\sf{(5)}$]  \label{winv}\label{defquaternion-symplectic2} $\omega(\J\cdot\,, \J\cdot)=\omega(\cdot\,, \cdot)$, for all $\J\in \Ss(Q_{0})$, which means that $\omega$ is scalar.
\item[$\sf{(6)}$]  $\omega(\J\cdot\,,\cdot)+\omega(\cdot\,, \J\cdot)=0$, for all $\J\in \Ss(Q_{0})$.
\item[$\sf{(7)}$] $\omega(J_a\cdot\,, J_a\cdot)=\omega(\cdot\,,\cdot)$, for any admissible  basis $H=\{J_a : a=1, 2, 3\}$ of $Q_{0}$.
\item[$\sf{(8)}$] $\omega(J_a\cdot\,, \cdot)+\omega(\cdot, J_a\cdot)=0$, for any  admissible  basis $H=\{J_a : a=1, 2, 3\}$ of $Q_{0}$.
\end{enumerate}
\end{prop}
\begin{proof}
Clearly, (2) is equivalent to (3), because $[\Lambda^2\Hh]^*$ is a trivial $\Sp(1)$-module and $[S^2\Hh]^*\cong \sp(1)$ is not.  Now, we may express the qH  pseudo-Euclidean metric  $( \ , \ )_{0}$  on $[\E\Hh]$ as
\[
(a+bj,c+dj)_{0}=\Re(a^t\bar{c}+b^t\bar{d})\,.
\]
It follows that $(x,Ay)_{0}=(Ay,x)_{0}=-(y,Ax)_{0}$, for any $A\in \sp(n)$ and $x,y\in [\E\Hh]$. Since $[S^2\E]^*\cong \sp(n)$, the scalar 2-forms correspond to invertible elements of $\sp(n)$. Moreover, the $\Sp(1)$-action on the 2-form $\omega(\cdot\,,\cdot):=(\cdot\,,A\cdot)_{0}$ commutes with $A\in \sp(n)$. Therefore, the claims (2)--(8) are equivalent due to the usual properties of $( \ , \ )_{0}$. 
Now, because $\omega_0$ is scalar, and the space $ [S^2\E]^*\otimes_\R [\Lambda^2\Hh]^*$ is $\Gl(n,\Hn)$-invariant, it remains to show that    for some   invertible $A\in \sp(n)$, the expression $(\cdot\,, A\cdot)_{0}$ is conjugated to $\omega_0$.   Indeed, for $B\in \Gl(n,\Hn)$, we get
\[
\omega(B\cdot\,,B\cdot)=(B\cdot\,,AB\cdot)_{0}=(\cdot\,,B^*AB\cdot)_{0}\,,
\]
 where $B^*$ is the conjugate transpose. Also, for $B\in \Sp(n)$ we see that $B^*AB=B^{-1}AB$, and it is  known that every element $A\in \sp(n)$ is conjugated to a diagonal purely imaginary quaternionic matrix (in a maximal torus in $\sp(n)$).  Now, since for any $a\in \sp(1)$ there is $b\in \Gl(1,\Hn)$, such that $\bar bab=j\in \Hn$, we may find $B\in \Gl(n,\Hn)$ such that 
\[
\omega(Bx,By)=(x,j\id_{[EH]}y)_{0}\,.
\]
But then for $x=a+bj,y=c+dj\in [\E\Hh]$ we obtain $j\id_{[EH]}y=d-cj$, and thus 
 \[
 \omega(Bx,By)=(a+bj,d-cj)_{0}=\Re(a^t\bar{d}-b^t\bar{c})=\omega_0(a+bj,c+dj)\,.
 \]
\end{proof}

Let us now focus on   invariant symmetric bilinear  forms on $[\E\Hh]$  induced by the module 
\[
 \langle\omega_0 \rangle \otimes_\R \sp(1)\subset S^2[\E\Hh]^*\,.
 \]
  This module is trivial under the $\SO^{\ast}(2n)$-action,  but nontrivial for the $\Sp(1)$-action. Let us show how these symmetric bilinear  forms can provide three $\SO^{\ast}(2n)$-invariant pseudo-Euclidean metrics,  which form an analogue of the usual fundamental (K\"ahler) 2-forms arising in the (linear) hH/qH setting.

\begin{prop}\label{signprop}
\textsf{1)}  The vector space $[\E\Hh]$   admits three  pseudo-Euclidean metrics  $g_{a}(\cdot\,,\cdot):=\omega_{0}(\cdot\,, \mc{J}_{a}\cdot)$   of signature $(2n, 2n)$, satisfying
\[
 g_{a}(\mc{J}_a\cdot\,, \mc{J}_a\cdot)=g_{a}(\cdot\,,\cdot)\,,\quad \forall \ a=1, 2, 3\,,
 \]
 where  $H_{0}=\{\mc{J}_{a} : a=1, 2, 3\}$ is the  standard admissible basis  and  $\omega_{0}$ is  the standard  scalar 2-form on $[\E\Hh]$. \\ 
 \textsf{2)}   Assume that   $H=\{I, J, K\}$ is an admissible basis of $Q_0$ and that  $\omega$ is a scalar 2-form on $[\E\Hh]$ with respect to $Q_0$. Then, the elements $g_{I}=\omega(\cdot\,, I\cdot)$, $g_{J}=\omega(\cdot\,, J\cdot)$ and $g_{K}=\omega(\cdot\,, K\cdot)$ are simultaneously conjugated to $g_1,g_2,g_3$ by an element in $\Gl(n,\Hn)\Sp(1)$, and thus they have the same properties.\\
  \textsf{3)}  For any $\J\in \Ss(Q_{0})$,  the tensor
   \[
\langle\cdot\,,\cdot\rangle_{\J}:=\omega_0(\cdot\,,\J\cdot)=g_{\E}\otimes \J\in \langle\omega_0 \rangle \otimes_\R \sp(1)
\]
is   a pseudo-Euclidean metric of signature $(2n, 2n)$  satisfying   $\langle\J x, \J y\rangle_{\J}=\langle x, y\rangle_{\J}$, for all $x,y\in [\E\Hh]$.\\
  \textsf{4)} For any $\J\in \Ss(Q_{0})$, the tensor
  \[
g_{\J}(\cdot\,, \cdot):=\omega(\cdot\,, \J\cdot)\in [S^2\E]^*\otimes \sp(1)
\]
is conjugated to $\langle\cdot\,,\cdot\rangle_{\J}$ by an element of $\Gl(n,\Hn)$, and thus has the same properties, too.
Note that $\langle\cdot\,,\cdot\rangle_{\mc{J}_a}=g_{a}(\cdot, \cdot)$, for $a=1, 2, 3$.
\end{prop}

\begin{proof}
It is not hard  to check that  $g_{\J }(\J\cdot\,, \J\cdot)=\omega(\cdot\,, -\J^3\cdot)=g_{\J}(\cdot\,, \cdot)$, for $\J\in \Ss(Q_0)$. Moreover, since $\omega$ is conjugated to $\omega_0$ by an element in $\Gl(n,\Hn)$, it suffices to prove the claims for $\langle \cdot, \cdot\rangle_{\J}$.  By using the standard admissible  basis $\{\mc{J}_{a} : a=1, 2, 3\}$ on $[\E\Hh]$ we may consider the linear complex structure $\J=\mu_1\mc{J}_1+\mu_2 \mc{J}_2+\mu_3 \mc{J}_3, \mu_1,\mu_2,\mu_3\in \R$, with $\sum_{a=1}^{3}\mu_a^2=1$.  Then, we obtain $\langle\cdot\,,\cdot\rangle_{\J}=\sum_{a=1}^{3}\mu_a\langle\cdot\,,\cdot\rangle_{\mc{J}_{a}}=\sum_{a=1}^{3}\mu_{a}g_{a}$, where  
 \begin{eqnarray*}
g_{1}(a+bj, c+dj)&=&\langle a+bj, c+dj\rangle_{\mc{J}_1}=\Re(-ia^t\bar{d}+ib^t\bar{c})=\Im(a^t\bar{d}-b^t\bar{c})\,,\\
g_2(a+bj, c+dj)&=&\langle a+bj, c+dj\rangle_{\mc{J}_2}=\omega_0(a+bj, -\bar{d}+\bar{c}j)=\Re(a^tc+b^td)\,,\\
g_{3}( a+bj, c+dj)&=&\langle a+bj, c+dj\rangle_{\mc{J}_3}=\Im(a^tc+b^td)
\end{eqnarray*}
  are split signature $(2n, 2n)$ metrics. To check the signature we proceed as follows: If $e_\ell$ is $\ell$-th vector of standard basis of $\C^n$, then $\langle\cdot\,,\cdot\rangle_{\J}$ restricted to $e_\ell,\mc{J}_1e_\ell, \mc{J}_2e_\ell, \mc{J}_3e_\ell$ takes the  form
\[ \left( \begin {array}{cccc} \mu_2&\mu_3&0&-\mu_1
\\ \mu_3&-\mu_2&\mu_1&0
\\ 0&\mu_1&\mu_2&\mu_3
\\ -\mu_1&0&\mu_3&-\mu_2\end {array}
 \right)\,.
\]
 This matrix has four eigenvalues: Two of them are given by  $\sqrt{\mu_1^2+\mu_2^2+\mu_3^2}=1$, and the other by $-\sqrt{\mu_1^2+\mu_2^2+\mu_3^2}=-1$. Since the 4-dimensional subspaces for different basis vectors are clearly orthogonal to each other, it follows that the signature is $(2n, 2n).$ This completes the proof.
\end{proof}
\begin{rem} \label{SOninvariantmetrics} 
\textnormal{\textsf{1)} Let us remark that  although the above pseudo-Euclidean metrics $g_{a}$  have signature $(2n, 2n)$, they are {\it not}  {Norden metrics}, since none of them is $H_0$-Hermitian. In particular, each $g_{a}$ is  $\mc{J}_{a}$-Hermitian, but not Hermitian for $\mc{J}_{b}$ with $b\neq a$. As a consequence of the first observation, there is an embedding \[
\SO^*(2n)\subset\U(n, n)\subset\SO(2n, 2n)\,.
\]
The same  conclusion applies for $g_{\J}$ (and $\langle \cdot , \cdot \rangle_{\J}$).  We mention that $\Aut\big(H_{0}, g_{i}\big)=\SO^*(2n)$,  but 
\[
\Aut\big(Q_{0}=\langle H_{0}\rangle,  g_{i}\big)= \SO^{\ast}(2n)\U(1)\ltimes \mathbb{Z}_2\subset\Gl(n,\Hh)\Sp(1)\,,
\]
for any $i=1, 2, 3$, and the same applies for the stabilizer of $\langle\cdot\,,\cdot\rangle_{\J}$ in $\Gl(n,\Hh)\Sp(1)$.  Here,  one should view $\U(1)$ as the  stabilizer of $\J\in S(Q)$, and  observe that $\mathbb{Z}_2$ acts as $-\id$ on $\J$ and $\omega_{0}$. }\\
\noindent \textnormal{\textsf{2)}  Finally, let us also mention that  the condition    $-g_{J_{a}}(J_{a}x, J_{a}y)+g_{J_{a}}(x, y)=0$, for any  $a=1, 2, 3$,  $x, y\in [\E\Hh]$, and any admissible basis $H=\{J_a : a=1, 2, 3\}$  of $Q_0$,  is equivalent to the final condition  posed in Proposition \ref{usefrel1}.}
\end{rem}

The Killing form of $\sp(1)$ provides a trivial $\SO^{\ast}(2n)\Sp(1)$-invariant submodule in  
\[
([S^2\E]^*\otimes \sp(1))\otimes \sp(1)\subset S^2[\E\Hh]^*\otimes \Gl([\E\Hh])\,.
\]
We show that this allows us to encode  the data   $\{\omega_0,g_1,g_2,g_3\}$ described above, into a single \textsf{quaternionic skew-Hermitian form}  on $[\E\Hh]$, which we may denote by
\[
h\in [\E\Hh]^*\otimes_\R [\E\Hh]^*\otimes_\R\Gl([\E\Hh])\,.
\]
\begin{defi}\label{defnew01}
A $\R$-bilinear form which is valued in endomorphisms  of $[\E\Hh]$,  that is an element $h\in [\E\Hh]^*\otimes_\R [\E\Hh]^*\otimes_\R\Gl([\E\Hh])$, is said to be  \textsf{quaternionic skew-Hermitian} if the following three conditions are satisfied:
\begin{itemize}
\item $\Re(h)(x,y):=\frac12 (h(x,y)-h(y,x))\in \R\cdot \id$;
\item  $\Im(h)(x,y):=\frac12 (h(x,y)+h(y,x))\in  Q_0=\sp(1)$;
 \item $h(\J \cdot \,, \cdot)=\J\circ h(\cdot \,, \cdot)$,
 \end{itemize}
  for all $x,y\in [\E\Hh]$ and  $\J\in \Ss(Q_0)$.   We call $\Re(h),\Im(h)$ the \textsf{real}, respectively \textsf{imaginary part}  of  $h$.   
\end{defi}

\begin{rem}
\textnormal{The standard admissible  basis $H_0$ on $[\E\Hh]$ provides the identification $\Hn\cong \R\oplus Q_0$, and the second condition in the more traditional Definition \ref{basicsright}   is clearly equivalent to the first two conditions in Definition \ref{defnew01}. Finally, the third condition in Definition \ref{defnew01} is clearly equivalent to the first condition in Definition \ref{basicsright} by $\R$-bilinearity of $h$.}
\end{rem}
 \begin{prop}\label{usefrelskewherm}
There is a unique $\SO^{\ast}(2n)\Sp(1)$-invariant trivial submodule in $S^2[\E\Hh]^*\otimes \sp(1)$ which provides the following imaginary part of the \textsf{standard quaternionic skew-Hermitian form} $h_{0}=\Re(h_{0})+\Im(h_{0})=g_{\E}\otimes \id_{[\Hh\Hh^*]}\in [\E\Hh]^*\otimes_\R [\E\Hh]^*\otimes_\R \Gl([\E\Hh])$ on $[\E\Hh]$,  
  \begin{equation}\label{fundskewherm}
 	\Im(h_{0})(\cdot\,,\cdot):=\sum_{a=1}^{3}g_{a}(\cdot\,,\cdot)\mc{J}_{a}\,,  
\end{equation}
where $H_{0}=\{\mc{J}_{a} : a=1, 2, 3 \}$ is the standard admissible basis and $g_{a}, a=1, 2, 3$ are defined in Proposition \ref{signprop}. Moreover, 
\[
\Re(h_{0})(x,y):=\omega_0(x,y)\otimes \id
\]
 is the real part of $h_{0}$ and the stabilizer in $\Gl([\E\Hh])$ of  $h_{0}$ is the Lie group $\SO^{\ast}(2n)\Sp(1)$, in particular
 \[
 \Aut(h_{0})=\SO^{\ast}(2n)\Sp(1)\,.
 \]
Finally, $h_{0}$ is equivalent to the linear quaternionic structure $Q_{0}=\langle H_0\rangle$ and the scalar 2-form $\omega_0$ on $[\E\Hh]$, while any quaternionic skew-Hermitian form $h$ on $[\E\Hh]$ is conjugated to $h_0$ by an element in $\Gl(n,\Hn)$.
\end{prop}
 \begin{proof}
 Observe first  that there is a non-degenerate skew $\C$-Hermitian form $\mathpzc{h}$ on $[\E\Hh]$, defined by 
 \[
 \mathpzc{h}(a+bj, c+dj):=\omega_0(a+bj, c+dj)+\langle a+bj, c+dj\rangle_{\mc{J}_1} \mc{J}_1=a^t\bar{d}-b^t\bar{c}\,.
 \]
Moreover, we see that
\begin{eqnarray*}
\langle a+bj, c+dj\rangle_{\mc{J}_2}\mc{J}_2&=&\omega_0(a+bj, -\bar{d}+\bar{c}j)\mc{J}_2=\Re(a^tc+b^td)\mc{J}_2\,,\\
\langle a+bj, c+dj\rangle_{\mc{J}_3}\mc{J}_3&=&\langle a+bj, -\bar{d}+\bar{c}j\rangle_{\mc{J}_3}\mc{J}_3=\Im(a^tc+b^td)\mc{J}_3\,,
\end{eqnarray*}
and therefore, $\langle x,y\rangle_{\mc{J}_2}\mc{J}_2+\langle x,y\rangle_{\mc{J}_3}\mc{J}_3=g_{\E}(\left( \begin{smallmatrix}
a \\
b
  \end{smallmatrix}  \right),\left( \begin{smallmatrix}
c \\
d
  \end{smallmatrix}  \right))\mc{J}_2.$   Define now
  \[
  h_0:=g_{\E}\otimes \id_{[\Hh\Hh^*]}: [\E\Hh]\otimes [\E\Hh]\to [\Hh\Hh^*]\cong\Hn\,,\]
where we view $[\Hh\Hh^*]\subset \Hh\otimes_\C\Hh$ as a 1-dimensional left quaternionic vector space via the map
\[
p=p_1+p_2j\mapsto p_1\left( 
    \begin{smallmatrix}
1 \\
0
  \end{smallmatrix}
  \right)\otimes \left( 
    \begin{smallmatrix}
0 \\ 
1
  \end{smallmatrix}
  \right)-\bar{p_1}\left( 
    \begin{smallmatrix}
0 \\
1
  \end{smallmatrix}
  \right)\otimes \left( 
    \begin{smallmatrix}
1 \\
0
  \end{smallmatrix}
  \right)+p_2\left( 
    \begin{smallmatrix}
1 \\
0
  \end{smallmatrix}
  \right)\otimes \left( 
    \begin{smallmatrix}
1 \\
0
  \end{smallmatrix}
  \right)+\bar{p_2}\left( 
    \begin{smallmatrix}
0 \\
1
  \end{smallmatrix}
  \right)\otimes \left( 
    \begin{smallmatrix}
0 \\
1
  \end{smallmatrix}
  \right).
\]
Then,    for any $a+bj, c+dj\in [\E\Hh]$,  we obtain that
\begin{align*}
h_0(a+bj, c+dj)&=(
a^t\bar{d}-b^t\bar{c}
)\otimes \left( 
    \begin{smallmatrix}
1 \\
0
  \end{smallmatrix}
  \right)\otimes \left( 
    \begin{smallmatrix}
0 \\
1
  \end{smallmatrix}
  \right)+( 
\bar{b}^tc -
\bar{a}^td
)\otimes \left( 
    \begin{smallmatrix}
0 \\
1
  \end{smallmatrix}
  \right)\otimes \left( 
    \begin{smallmatrix}
1 \\
0
  \end{smallmatrix}
  \right)\\
  &+( a^tc+
b^td
)\otimes \left( 
    \begin{smallmatrix}
1 \\
0
  \end{smallmatrix}
  \right)\otimes \left( 
    \begin{smallmatrix}
1 \\
0
  \end{smallmatrix}
  \right)+(
\bar{b}^t\bar{d}+
\bar{a}^t\bar{c}
)\otimes \left( 
    \begin{smallmatrix}
0 \\
1
  \end{smallmatrix}
  \right)\otimes \left( 
    \begin{smallmatrix}
0 \\
1
  \end{smallmatrix}
  \right)\\
  &=a^t\bar{d}-b^t\bar{c}+(a^tc+b^td)\mc{J}_2\\
  &=\mathpzc{h}(a+bj, c+dj)+g_{\E}(a+bj, c+dj)\mc{J}_{2}\,.
\end{align*}
This shows that the definition of $h_0$ is independent of the choice of an admissible hypercomplex basis. Moreover, by   Definition \ref{basicsright} it remains to check the following:
\begin{align*}
h_0(c+dj,a+bj)&=-\overline{\mathpzc{h}(a+bj, c+dj)}+g_{\E}(a+bj, c+dj)\mc{J}_2=-\overline{h_0(a+bj, c+dj)}\,,\\
h_0(p(a+bj),c+dj)&=p_1h_0(a+bj,c+dj)+p_2h_0(-\bar{b}+\bar{a}j,c+dj)=p_1h_0(a+bj,c+dj)\\
&+p_2j(-j)(-\bar{b}^t\bar{d}-\bar{a}^t\bar{c}+(-\bar{b}^tc+\bar{a}^td)j)=ph_0(a+bj,c+dj)\,,
\end{align*}
for any $p=p_1+p_2j\in \Hn$. Consequently,  $h_0$ is a quaternionic skew-Hermitian form on $[\E\Hh]$.\\
To conclude the proof, note that  $\SO^{\ast}(2n)\Sp(1)$ is clearly contained in the stabilizer of $h_0$.  In addition,  the linear quaternionic structure  $Q_0$ is spanned by the imaginary part $\Im(h_0)(x,y)$, for $x,y\in [\E\Hh]$, and thus the stabilizer of $h_0$ is contained in $\Gl(n,\Hn)\Sp(1)$. On the other hand,  the real part of $h_0(x,y)$ recovers the scalar 2-form $\omega_0$, whose stabilizer  inside $\Gl(n,\Hn)\Sp(1)$ coincides with the Lie group $\SO^{\ast}(2n)\Sp(1)$.  Thus, the last claim follows, since   for any quaternionic skew-Hermitian form $h$ on $[\E\Hh]$ we have 
\[
\omega=\Re(h)\,,\quad  \Im(h)=\sum_{a=1}^{3}\omega(\cdot, \mc{J}_{a}\cdot)\mc{J}_{a}=\sum_{a=1}^{3}g_{\mc{J}_{a}}(\cdot\,, \cdot)\mc{J}_{a}\,,
\]
where $\omega, g_{\mc{J}_a}$ are simultaneously conjugated to $\omega_0,g_a$ by an element in $\Gl(n,\Hn)$.
 \end{proof}

Again the Killing form of $\sp(1)$ provides a trivial submodule in $S^2 (\langle\omega_0 \rangle \otimes_\R \sp(1))$, 
and thus a trivial $\SO^{\ast}(2n)\Sp(1)$-invariant submodule of $S^4[\E\Hh]^*$.  Next we will use  Proposition \ref{usefrelskewherm} to prove  that this tensor provides an analogue of the fundamental 4-form  $\Omega_0=\sum_{a}\omega_{a}\wedge\omega_{a}$, appearing in the theory of hH/qH structures. 
 
 \begin{prop}\label{usefrelfund}
There is a unique $\SO^{\ast}(2n)\Sp(1)$-invariant trivial submodule in $S^4[\E\Hh]^*$, spanned by the totally symmetric  4-tensor
  \begin{equation}
	 \Phi_{0}:=g_{1}\odot g_{1}+g_{2}\odot g_{2}+g_{3}\odot g_{3}\,,
	\label{fundtensorlin}
\end{equation}
where  $\odot$  denotes the symmetrized tensor product and $g_{a}, a:=1, 2, 3$ are defined in Proposition \ref{signprop}. Moreover,  the \textsf{complete symmetrization} $\mathsf{Sym}$ of $\omega_0(\cdot\,, \Im(h_0)\cdot)$, where $\Im(h_0)$ is defined by {\rm (\ref{fundskewherm})}, satisfies the relation
 \[
\Phi_{0}=\mathsf{Sym}\big(\omega_0(\cdot\,, \Im(h_0)\cdot)\big).
\]
Thus, the stabilizer of $\Phi_0$ in $\Gl([\E\Hh])$ is $\SO^{\ast}(2n)\Sp(1)$, i.e., $\Aut(\Phi_{0})=\SO^{\ast}(2n)\Sp(1)$. 
\end{prop}
\begin{proof}
The computation of the  dimension of the space of invariant symmetric 4-tensors requires deeper results from representation theory, which we avoid to  review in detail and refer to \cite{FH}.  Recall that  $\theta$ is the fundamental weight of $\fr{sp}(1)$, and $\Hh=R(\theta)$.   The  following equality is a special case of an equivariant isomorphism which holds for any tensor product of Lie algebra modules (see \cite{FH}):
 \begin{equation}
		S^4(\E\Hh) = \sum_{Y\in \text{Young}(4)} Y(\E) \otimes Y(\Hh),
	\end{equation}
	 where in general $\text{Young}(n)$ denotes the \textsf{set of plethysms} associated to \textsf{Young diagrams} with $n$ boxes.  For $n=4$, there are the following  five Young  diagrams:
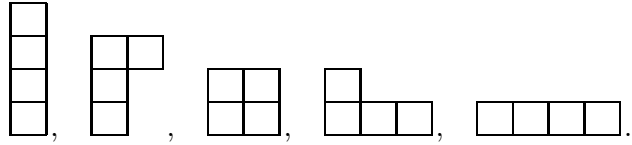
\[
 {\small\begin{Young}
\cr
\cr
\cr
\cr
\end{Young}}\,, \quad  \begin{Young} 
& \cr
\cr
\cr
\end{Young}\,, \quad
\begin{Young} 
& \cr
& \cr
\end{Young}\,, \quad
\begin{Young}
\cr & 
&  \cr
\end{Young}\,,
 \quad
 \begin{Young}
& & & \cr
\end{Young} \,.
\]
Then, with respect to $\Sp(1)$ we obtain the following:
	\begin{itemize}
		\item $(4)R(\theta) = R(4\theta)$;
		\item $(3,1)R(\theta) = R(2\theta)$;
		\item $(2,2)R(\theta) = R(0)$, the trivial representation;
		\item $(2,1,1)R(\theta) = \{0\}$, 0 dimensional;
		\item $(1,1,1,1)R(\theta) = \{0\}$, 0 dimensional.
	\end{itemize}
	This shows that any trivial $\SO^\ast(2n)\Sp(1)$-invariant subspace of $S^4(\E\Hh)$ must be contained in the summand $(2,2)\E\otimes (2,2)\Hh,$
	 and the dimension is equal to the dimension of $\SO(2n,\C)$-invariant subspaces in $(2,2)\E$.   In particular,  the dimension of the space of invariants is one,  which yields the tensor $\Phi_{0}$ (this claim is valid also for the low dimensional cases included in Table \ref{Table1}).
	Indeed, the space $\langle g_1, g_2, g_3 \rangle \subset S^2[\E\Hh]^*$ is $\SO^\ast(2n)$-trivial, but $\Sp(1)$-invariant, and equivariantly isomorphic to the space of imaginary quaternions $\Im (\Hn)$,  equipped with the standard admissible basis $H_{0}$. Since the latter space is self-dual, and has an invariant inner product given by the sum of squares of the admissible basis, the tensor $\Phi_{0}$ given by formula (\ref{fundtensorlin}) is also invariant and thus spans the invariant subspace.\\
	Finally, it is a simple observation that  $\Phi_{0}=\mathsf{Sym}\big(\omega_0(\cdot\,, \Im(h_0)\cdot)\big)$, where $\mathsf{Sym}$ is the operator of complete symmetrization, thus by Proposition \ref{usefrelskewherm} the stabilizer of $\Phi_0$ in $\Gl([\E\Hh])$ must contain the Lie group  $\SO^{\ast}(2n)\Sp(1)$, that is $\SO^*(2n)\Sp(1)\subseteq \Aut(\Phi_0)$.  We will prove also the other inclusion at an infinitesimal level. First, under   the $\SO^*(2n)\Sp(1)$-action we see by Proposition \ref{modules} that    $\Ed([\E\Hh])$ decomposes as follows: 
	\[
	\Ed([\E\Hh]) \cong \R\cdot\Id \oplus  \ \fr{sp}(1) \oplus \so^\ast(2n) \oplus \frac{\fr{sl}(n,\Hn)}{\so^\ast(2n)} \oplus\frac{\sp(\omega_0)}{\so^\ast(2n)\oplus \sp(1)} \oplus [\Lambda^2\E S^2\Hh]^*\,,
\]
where $\R\cdot\id\cong \langle\omega_0\rangle$ and
\[
\fr{sp}(1)\cong [S^2\Hh]^*\,,\quad \fr{so}^*(2n)\cong [\Lambda^{2}\E]^*\,,\quad      \displaystyle\frac{\fr{sl}(n,\Hn)}{\so^\ast(2n)}\cong  [S^2_0\E]^*\,,\quad 
 \displaystyle\frac{\sp(\omega_0)}{\so^\ast(2n)\oplus \sp(1)}\cong [S^2_0\E]^*\otimes\fr{sp}(1)\,.
\]
  Note that for $n>2$ the above decomposition can be read in terms of irreducible submodules.
	The Lie algebra of $\Aut(\Phi_{0})$ is a proper submodule of the above.
If an element of $\fr{sp}(\omega_0)$ or $\gl(n,\Hn)$ preserves $\Phi_{0}$, then this element should belong to $\so^\ast(2n) \oplus \sp(1)$, since none of the algebras  $\fr{sp}(\omega_0)$ or $\gl(n,\Hn)$  preserves a symmetric 4-tensor.   On the other hand, by Remark \ref{zerorem}  we can express the pure tensors $A$ in the final submodule $[\Lambda^2\E S^2\Hh]^*$ as 
 \[
 \omega_0(Ax, y)=\rho_{A}(x,Jy)\,,
 \]
  for some qH pseudo-Euclidean metric $\rho_{A}$ depending on $A$, and some almost complex structure  $J$  belonging to an admissible basis $H$ of $Q_0$. Then, we can  compute the action $A\cdot \Phi_0$. To do so, we need the action of $A$ on $g_{I}, g_{J}, g_{K}$. Based on the    fact that $\omega_0$ is scalar  and by using Proposition \ref{usefrel1}, we deduce that 
\begin{eqnarray*}
(A\cdot g_{I})(x, y)&=&-g_{I}(Ax, y)-g_{I}(x, Ay)=-\omega_{0}(Ax, Iy)-\omega_{0}(Ay, Ix)\\
&=&-\rho_{A}(x, JIy)-\rho_{A}(y, JIx)=0\,,\\
(A\cdot g_{J})(x, y)&=&-g_{J}(Ax, y)-g_{J}(x, Ay)=-\omega_{0}(Ax, Jy)-\omega_{0}(Ay, Jx)\\
&=&-\rho_{A}(x, J^2y)-\rho_{A}(y, J^2x)=2\rho_{A}(x, y)\,,\\
(A\cdot g_{K})(x, y)&=&-g_{K}(Ax, y)-g_{K}(x, Ay)=-\omega_{0}(Ax, Ky)-\omega_{0}(Ay, Kx)\\
&=&-\rho_{A}(x, JKy)-\rho_{A}(y, JKx)=0\,.
 \end{eqnarray*}
  Thus, by the definition of $\Phi_0$ we finally obtain
   \[
  A\cdot \Phi_0=4\rho_{A}\odot g_J\,,
  \]
   which  never vanishes. In particular, for linear independent pure tensors $A$ we see that also the corresponding qH pseudo-Euclidean metrics $\rho_{A}$ are linear independent. Thus, together with the previous inclusion we deduce that  the Lie algebra of $\Aut(\Phi_{0})$ coincides with $\so^\ast(2n)\oplus \sp(1)$. 
\end{proof}


\subsection{Linear $\SO^\ast(2n)$-structures and linear $\SO^{\ast}(2n)\Sp(1)$-structures}\label{mainlinear}

Let $H=\{J_1,J_2,J_3\}$ be a linear hypercomplex structure on $4n$-dimensional vector space $V$, or let $H=\{J_1,J_2,J_3\}$ be an admissible  basis of  a linear quaternionic structure $Q$ on $V$. Next we will show that the  basis defined below provides the identification with the $\E\Hh$-formalism.

\begin{defi}\label{EHbases}
 We say that a basis  $e_1,\dots,e_{2n},f_1,\dots,f_{2n}$ of $V$ is \textsf{adapted} to $H$ if
\[
J_1(e_{c})=e_{c+n}\,,\quad
J_2(e_{c})=f_{c}\,,\quad 
J_3(e_{c})=f_{c+n}
\]
 for $c=1,\dots n$.  Let us  also use the notation 
 \[
 a=(u_1+iu_{n+1},\dots,u_n+iu_{2n})^t\,,\quad bj=(v_1+iv_{n+1},\dots,v_n+iv_{2n})^t
 \]
  for the coordinates $(u_1,\dots,u_{2n},v_1,\dots,v_{2n})^t$ in the basis which provides the isomorphism $V\cong [\E\Hh]$.
\end{defi}
We should mention that such a  basis is not an {\it admissible basis to $H$} in terms of \cite[Def.~1.4]{AM} (see also the appendix, Section \ref{appendix}).
\begin{example}\label{examplebase1}
For $n=2$,  assume that $e_1, e_2$ are non-zero vectors in $\Hn^2$ for which the quaternionic lines $\Hn \cdot e_1$ and $\Hn\cdot e_2$ do not coincide. If $H$ is the linear hypercomplex structure induced by left multiplication via $i, j, k$, then the basis  adapted to $H$ is given by $\{e_1, e_2, ie_1, ie_2, je_1, je_2, ke_1, ke_2\}$.
\end{example}
\begin{prop}\label{hypbasis}
 Let $H=\{J_1,J_2,J_3\}$ be a linear hypercomplex structure on $4n$-dimensional vector space $V$, or let $H=\{J_1,J_2,J_3\}$ be an admissible basis of a linear quaternionic structure $Q$ on $V$. Then, there is a basis adapted to   $H$, such that   
 \[
 V\cong [\E\Hh]\,,
 \]
and under this isomorphism we get the identification $H=H_{0}=\{\mc{J}_{a} : a=1, 2, 3\}$, where $H_{0}$ is the standard admissible basis of $Q_0$ on $[\E\Hh]$.
\end{prop}
\begin{proof}
 Clearly, there is an $n$-tuple of linearly independent vectors $e_1,\dots,e_{n}$ such that 
\[
e_{c+n}:=J_1(e_{c})\,,\quad
f_{c}:=J_2(e_{c})\,,\quad
f_{c+n}:=J_3(e_{c})\,,
\]
are all linearly independent,  which means that  there is a basis adapted to $H$ in the above terms.  It is clear that under the isomorphism $V\cong [\E\Hh]$ provided by this basis, we have  
\[
(a_1+a_2J_1+a_3J_2+a_4J_3)(a+bj)=(a_1+a_2i+a_3j+a_4k)(a+bj)=(a_1+a_2\mc{J}_1+a_3\mc{J}_2+a_4\mc{J}_3)(a+bj)
\]
  and thus $H=H_{0}$. This proves our assertion. 
\end{proof}

 Proposition \ref{usefrelskewherm} in combination with the above construction,  motivates us to proceed with the following definitions.

 \begin{defi}\label{def1} 
 Let   $V$ be a $4n$-dimensional real vector space.   A pair $(h,H)$ consisting of an element  $h\in V^*\otimes V^*\otimes\Gl(V)$ and a  linear hypercomplex structure $H=\{J_a\in\Ed(V) : a=1, 2, 3\}$,  is said to be a \textsf{linear hypercomplex skew-Hermitian structure} on  $V$ (\textsf{linear hs-H structure} for short), if  the following conditions are satisfied:
	\begin{enumerate}
\item[\textsf{(1)}]  The real part $\Re(h)(x,y):=\frac12(h(x,y)-h(y,x))$ of $h$ satisfies $\Re(h)(x,y)=\omega(x,y)\cdot \id$,  for all $x,y\in V$ and   some non-degenerate 2-form $\omega$ on $V$.
\item[\textsf{(2)}] $h$ is a quaternionic skew-Hermitian form with respect to $H$, that is, $\omega$ is a scalar 2-form and 
\[
h(x,y)=\omega(x,y)\id+\sum_{a=1}^3 g_{J_a}(x,y)J_a
\]
 holds for all $x,y\in V$.  
\end{enumerate}
	\end{defi}
	
	\begin{rem}
 \textnormal{Let us also emphasize that the linear quaternionic structure generated by   $H$, can be equivalently obtained by the image of the imaginary part  of $h$, defined by
 \[
 \Im(h):=\frac12(h(x,y)+h(y,x))\,, \quad \forall \ x, y\in V\,.
 \]}  
\end{rem}

 \begin{defi}\label{def2}
 Let   $V$ be a $4n$-dimensional real vector space.     An element $h\in V^*\otimes V^*\otimes\Gl(V)$ is said to be a  \textsf{linear quaternionic skew-Hermitian structure}  (\textsf{linear qs-H structure} for short), if  the following conditions are satisfied:
	\begin{enumerate}
\item[\textsf{(1)}]  The real part $\Re(h)$ of $h$ satisfies $\Re(h)(x,y)=\omega(x,y)\cdot \id$, for all $x,y\in V$ and  some non-degenerate 2-form $\omega$ on $V$.
\item[\textsf{(2)}] The imaginary part   $\Im(h)$ of $h$  induces  a linear quaternionic structure $Q$ on $V$.
\item[\textsf{(3)}] $h$ is a  quaternionic skew-Hermitian form,  that is $\omega$ is a scalar 2-form and 
\[
h(x,y)=\omega(x,y)\id+\sum_{a=1}^3 g_{J_a}(x,y)J_a
\]
 holds for any admissible basis $H=\{J_1,J_2,J_3\}$ of $Q$ and for all $x,y\in V$.   \end{enumerate}
 Often, we shall  call such an admissible basis  $H$ of $Q$ also an \textsf{admissible basis} of the linear quaternionic skew-Hermitian structure $(h, Q=\langle H\rangle)$.
	\end{defi}

 Let us now specify the bases that allow us to  relate linear \textsf{hs-H} and \textsf{qs-H} structures, as defined above,  with the results from the previous subsection.

 \begin{prop}\label{adaptbas}
 Let $(h, H=\{J_1,J_2,J_3\})$ be a linear hypercomplex skew-Hermitian structure on $V$, or let $H=\{J_1,J_2,J_3\}$ be an  admissible basis of a  linear  quaternionic skew-Hermitian structure $(h, Q=\langle H\rangle)$. Set $\omega:=\Re(h)$. Then, there is a symplectic basis of $\omega$ adapted to $H$, that is 
\[
 \omega(e_r,e_s)=0\,,\quad \omega(f_r,f_s)=0\,,\quad \omega(e_r,f_r)=1\,,\quad\omega(e_r,f_s)=0\,, (r\neq s)
 \]
 for $1\leq r\leq 2n$ and $1\leq s\leq 2n$, 
 and 
 \[
J_1(e_{c})=e_{c+n}\,,\quad
J_2(e_{c})=f_{c}\,,\quad 
J_3(e_{c})=f_{c+n}\,,
\] 
for $1\leq c\leq n$, respectively.  Moreover, under the isomorphism $V\cong [\E\Hh]$ provided by the basis adapted to $H$, we have  $\omega=\omega_0$,  $H= H_{0}$, where $\omega_0$ is the standard scalar 2-form  and $H_{0}$ is the standard admissible basis on $[\E\Hh]$, respectively.  In particular, the following claims hold:
 
 \begin{enumerate}
\item[${\sf (1)}$] A linear hypercomplex skew-Hermitian structure on $V$ is equivalent to a pair $(H,\omega)$, consisting of a linear hypercomplex structure $H=\{J_1,J_2,J_3\}$ and a scalar 2-form $\omega$, both defined on $V$.  Equivalently,  a linear hypercomplex skew-Hermitian structure on $V$ is a $\SO^*(2n)$-structure on $V$.
 \item[${\sf (2)}$]  A linear quaternionic skew-Hermitian structure on $V$ is equivalent to a pair $(Q,\omega)$, consisting of a  linear quaternionic structure $Q$ and a scalar 2-form $\omega$, both defined on $V$.  Equivalently,  a linear quaternionic skew-Hermitian structure on $V$ is a $\SO^*(2n)\Sp(1)$-structure on $V$.
   \end{enumerate}
\end{prop}
\begin{proof}
The definition of a \textsf{linear hs-H structure} or a \textsf{linear qs-H structure}, provides pairs $(H,\omega)$ and $(Q,\omega)$, respectively, with the claimed properties. Picking an admissible basis for $Q$, reduces us to the situation of a pair  $(H=\{J_1,J_2,J_3\},\omega)$.
Due to Proposition \ref{hypbasis}, there is a basis $e_1',\dots,e_{2n}',f_1',\dots,f_{2n}'$ adapted to $H$, and $\omega$ is a scalar 2-form on $[\E\Hh]$ under the isomorphism $V\cong [\E\Hh]$. We also know by Proposition \ref{usefrel1} that $\omega$ is conjugated to $\omega_0$ by  an element $B\in \Gl(n,\Hn)$.  This provides a basis $e_1,\dots,e_{2n},f_1,\dots,f_{2n}$ of $V$, such that (after the change of coordinates) $\omega=\omega_0$.  It is a simple observation that $\omega_0$ is the standard symplectic form in these coordinates, and thus $e_1,\dots,e_{2n},f_1,\dots,f_{2n}$ is a symplectic basis adapted to the linear hypercomplex structure $H$ on $V$. This is because the action of $B$ commutes with the action of $H$. In particular, the standard quaternionic skew-Hermitian form $h_0$ introduced in Proposition \ref{usefrelskewherm} defines a \textsf{linear hs-H structure}, and  a \textsf{linear qs-H structure} on $V$. By the last claim in Proposition \ref{usefrelskewherm}, if we start   with a \textsf{linear hs-H} or \textsf{qs-H structure}, then we just obtain its coordinates in the basis $e_1,\dots,e_{2n},f_1,\dots,f_{2n}$ of $V$, and therefore all claims (1) and (2) must hold (see the next section for more details on $G$-structures). 
\end{proof}

Having discussed many alternative ways to define  the particular types of linear structures that we are  interested in, it is convenient to summarize their differences from the well-known linear  hH/qH structures,    which we encode below in Table \ref{Table2}.
 
   \begin{table}[ht]
\centering
\renewcommand\arraystretch{1.3}
{\footnotesize \begin{tabular}{c | l | l | c | l }
 { linear $G$-structure} & initial data & tensors & fundam.  tensor & { stabilizer $G$} \\
 \thickline 
  hH & $\big(( \ , \ )_{0}=\omega_{\E}\otimes\omega_{\Hh}, H_0\big)$ & $\omega_{a}(\cdot , \cdot)=(\cdot, J_{a}\cdot)_{0}$ & $\Omega_0\in(\Lambda^{4}[\E\Hh]^*)^{G}$ & $\Sp(n)$ \\
 hs-H & $\big(\omega_{0}=g_{\E}\otimes\omega_{\Hh}, H_0\big)$ & $g_{a}(\cdot , \cdot)=\omega_{0}(\cdot, J_{a}\cdot)$ & $\Phi_0\in(S^{4}[\E\Hh]^*)^{G}$ & $\SO^{\ast}(2n)$ \\
 qH & $ \big(( \ , \ )_{0}, Q_0=\langle H_0\rangle\big)$ & $\omega_{a}(\cdot , \cdot)=(\cdot, J_{a}\cdot)_{0}$ & $\Omega_0\in(\Lambda^{4}[\E\Hh]^*)^{G}$  & $\Sp(n)\Sp(1)$ \\
 qs-H & $\big(\omega_{0}, Q_0=\langle H_0\rangle\big)$ & $g_{a}(\cdot , \cdot)=\omega_{0}(\cdot, J_{a}\cdot)$ & $\Phi_0\in (S^{4}[\E\Hh]^*)^{G}$ & $\SO^{\ast}(2n)\Sp(1)$
 \end{tabular}}
 \vspace{0.5cm}
\caption{\small hH/qH linear structures Vs hs-H/qs-H  linear structures}\label{Table2}
 \end{table}
Proposition \ref{adaptbas} is a powerful tool which  we will often  apply  when  we examine $\SO^*(2n)$- and $\SO^*(2n)\Sp(1)$-structures  on  manifolds. Moreover, it motivates us to introduce the following 
\begin{defi}\label{basSKEW}
 Let $(h,H=\{J_1,J_2,J_3\})$ be a linear \textsf{hs-H} structure on $V$ or let $H=\{J_1,J_2,J_3\}$ be an  admissible  basis of a   linear quaternionic skew-Hermitian form $h$.  We say that the  symplectic basis adapted to $H$ by Proposition \ref{adaptbas}, is a \textsf{skew-Hermitian basis} of the linear \textsf{hs-H} or linear \textsf{qs-H} structure, respectively.
\end{defi}	
\begin{example}\label{examplebase2}
By Example \ref{examplebase1}, we can consider $\Hn^2$ endowed with $H$ given by the left quaternionic multiplication, and an  adapted basis  to  $H$ given by $\fr{B}:=\{e_1, e_2, ie_1, ie_2, je_1, je_2, ke_1, ke_2\}$.  Let  $\omega$  be  the bilinear form   on $\Hn^2$ defined by 
\[
\omega(x,y)=\frac12(x^tj\bar{y}-y^tj\bar{x})\,,\quad \forall \ x,y\in \Hn^2\,,
\]
 where $\bar{x}$ denotes the quaternionic conjugate. Then, $\omega$ is a scalar 2-form with respect to  $H$ and the basis $\fr{B}$   
  is a skew-Hermitian basis of the \textsf{linear hs-H} structure $(h,H)$. Here,  the linear quaternionic skew-Hermitian form $h$ is induced by $\omega$ and $H$  and so it takes form $h(x,y)=x^tj\bar{y}$, see also Corollary \ref{corolA1} in the appendix.
\end{example}
Note that in this example, we have chosen $\omega$ is such a way  that the  adapted basis to $H$ is the same with the skew-Hermitian basis of the \textsf{linear hs-H structure} $(h, H)$.  However, this is not the generic  case, and we should emphasize that in general an explicit transition between a basis adapted to $H$ and a skew-Hermitian basis, can be  carried out by a generalization of the \textsf{Gram-Schmidt orthogonalization process}. Since the action of $H$ identifies $V$ with a {\it left} quaternionic vector space,   these ``transitions'' between the bases    can {\it not} be realized by left multiplication by quaternionic matrices.  Thus, we will postpone the explicit construction to the appendix,    where bases that provide an identification of $V$ with a right quaternionic vector space are specified.


\subsection{The symplectic viewpoint}\label{mainsymplectic}

 In  Section \ref{mainlinear} we started by fixing a linear hypercomplex structure $H$  on $V$, and by using bases adapted to $H$ we obtained the identification $V\cong [\E\Hh]$. This enabled  a convenient   description of  linear \textsf{hs-H} and \textsf{qs-H} structures in terms of the $\E\Hh$-formalism. In this section, we shall adopt the opposite point of view. This means that  we will fix a \textsf{linear symplectic form $\omega$  on $V$} (i.e., a non-degenerate 2-form on $V$) and a \textsf{symplectic basis}, to get an identification $V\cong\R^{4n}$. This procedure will allow  us to examine linear \textsf{hs-H} and \textsf{qs-H} structures from a symplectic point of view, which can be analyzed  in terms of the  standard symplectic form $\omega_{\st}(x, y)=x^{t}S_{0}y$ on $\R^{4n}$. Here,  as usual, $S_0$ is the matrix defined by 
 \[
 S_{0}:=\begin{pmatrix} 0 & \Id_{2n} \\ -\Id_{2n} & 0 \end{pmatrix}\,.
 \]

With this goal in mind, it is convenient to recall first  the notion of  the so-called \textsf{symplectic twistor space} attached to a symplectic vector space, see also  \cite{CahGR}. 
\begin{defi}
The \textsf{symplectic twistor space} of $(\R^{4n}, \omega_{\st})$ of signature $(p,q)$ is the set of all linear complex structures compatible with $\omega_{\st}$, that is
\begin{itemize}
\item $J^2=-\id_{\R^{4n}}$;
\item $\omega_{\st}(Jx,Jy)=\omega_{\st}(x,y)$, for any $x, y\in\R^{4n}$;
\item the pseudo-Euclidean Hermitian metric $g_J(\cdot\,,\cdot):=\omega_{st}(\cdot\,, J\cdot)$ has signature $(p,q)$;
\end{itemize}
\end{defi}

We can  now describe the twistor space in the following way:
\begin{lem}\label{twistfiber}
The union of all symplectic twistor spaces for all signatures coincides with the  space $\Sp(4n,\R)\cap \sp(4n,\R)$, and in these terms the following claims hold:\\
\textsf{1)} The adjoint orbits of $\Sp(4n,\R)$ in $\Sp(4n,\R)\cap \sp(4n,\R)$ are uniquely characterized by the signature $(4n-2q,2q)$ of the metric
\[
 g_J(x, y)=\omega_{\st}(x, Jy)=x^{t}S_{0}Jy\,,
\]
that is, the stabilizer in $\Sp(4n,\R)$ of a point  in an  orbit with signature $(4n-2q,2q)$ is the Lie group $\U(2n-q,q)$.

\textsf{2)} If $I,J,K\in \Sp(4n,\R)\cap \sp(4n,\R)$ define a linear hypercomplex or a  linear quaternionic structure on $\R^{4n}$, then $I,J,K\in \Sp(4n,\R)/\U(n,n)$, i.e., they are elements of the symplectic twistor space of signature $(n,n)$.
 \end{lem}
 \begin{proof}
 By definition, any  $J\in \Sp(4n,\R)\cap \sp(4n,\R)$ satisfies  $\omega_{\st}(Jx, Jy)=\omega_{\st}(x, y)$ and $\omega_{\st}(Jx, y)+\omega_{\st}(x, Jy)=0$, for any $x, y\in\R^{4n}$. Thus $\omega_{\st}(x, y)=\omega_{\st}(Jx, Jy)=-\omega_{\st}(x, J^{2}y)$. Since $\omega_{\st}$ is non-degenerate, this implies $J^2=-\id$.
Conversely, if  $J$ is such that $\omega_{\st}(Jx, Jy)=\omega_{\st}(x,y)$ and $J^2=-\id$, then 
\[
\omega_{\st}(Jx, y)+\omega_{\st}(x, Jy)=\omega_{\st}(J^{2}x, Jy)+\omega_{\st}(x, Jy)=0\,.
\]
 This proves our  initial claim.
 Now, to prove 1) note that the adjoint orbits of $\Sp(4n,\R)$ in $\sp(4n,\R)$ are well-known and the representatives $J$ of orbits {satisfying} $J^2=-\id$ are given by 
 \[
 J_{2n-q,q}:=\begin{pmatrix} 0 & 0 & \Id_{2n-q} &0 \\ 
0 & 0 & 0 &-\Id_{q} \\
-\Id_{2n-q}& 0 &0 &0 \\
0 & \Id_{q} & 0 &0 \end{pmatrix}\,.
\]
 Thus, the  pseudo-Euclidean metric defined by  $\omega_{\st}(x, J_{2n-q,q}y)=x^{t}S_{0}J_{2n-q,q}y$ has signature  $(4n-2q,2q)$.
Finally, we should mention that the  assertion 2) is a consequence of  Proposition \ref{signprop}.
  \end{proof}
  
The proof of Lemma \ref{twistfiber} suggests defining the following operation:

\begin{defi}
We say  that $A^T\in \Gl(\R^{4n})$ is the \textsf{symplectic transpose} of $A\in \Gl(\R^{4n})$ if
\[
\omega_{\st}(A^{T}x,y)=-\omega_{\st}(x, Ay)
\]
holds for all $x,y\in \R^{4n}.$
\end{defi}

 We shall now prove that the symplectic transpose always exists.  
\begin{lem}\label{symptrans}
Let $\mathscr{B}_{\st}$ be the symplectic basis on $\R^{4n}$ such that $\omega_{\st}$ is given by $\omega_{\st}(x,y)=x^tS_0y$. Then $A^T=S_0A^tS_0$, where $A^t$ is the usual transpose of $A$ in the coordinates of the symplectic basis.  
 \end{lem}
 \begin{proof}
By a direct computation we obtain
\[
\omega_{\st}(Ay, x)=y^tA^tS_0x=y^tS_0(S_0^{-1}A^tS_0)x=\omega_{\st}(y, S_0^{-1}A^tS_0x)=-\omega_{\st}(-S_0A^tS_0x, y)
\] 
and by  the non-degeneracy of $\omega_{\st}$, it follows that $A^T=S_0A^tS_0$.
  \end{proof}
  As a corollary, we deduce  the following:
\begin{corol}
The 2-form $\omega_{\st}$ is Hermitian  with respect to a linear complex structure $J$ on $\R^{4n}$, if and only if $J^T=J$.
 \end{corol}
 \begin{proof}
 Assume that $\omega_{\st}$ is $J$-Hermitian, i.e., $\omega_{\st}(Jx, Jy)=\omega_{\st}(x, y)$ for any $x, y\in\R^{4n}$. Then, by replacing $x$ by $Jx$ and by definition of $J^{T}$ we obtain
 \[
\omega_{\st}(J^2x, Jy)= \omega_{\st}(Jx, y)=-\omega_{\st}(y, Jx)=\omega_{\st}(J^{T}y, x)\,.
 \]
 Since  also $\omega_{\st}(J^2x, Jy)=-\omega_{\st}(x, Jy)=\omega_{\st}(Jy, x)$, we finally obtain $\omega_{\st}(Jy, x)=\omega_{\st}(J^{T}y, x)$, 
 that is $J=J^{T}$. The converse is treated similarly.
 \end{proof}
 
 Let us now link the above description with the structures that we are interested in. So, assume  that  $(\R^{4n}, \omega_{\st})$ is endowed with a  linear quaternionic structure $Q\subset\Ed(\R^{4n})$, for which $\omega_{\st}$ is \textsf{scalar}, that is $\omega_{\st}$ is $Q$-Hermitian in terms of the Definition \ref{basicsright} (see also Proposition \ref{usefrel1}).  Then, by the above description it follows that 
  \begin{corol}
 The 2-sphere $S(Q)=\Sp(1)\cap\sp(1)$ associated to a linear quaternionic structure $Q$ is a subspace of  the $\Sp(4n,\R)/\U(n, n)$-orbit. 
 \end{corol}
  Let us now characterize the space of such linear quaternionic structures $Q$.

\begin{lem}\label{quatchar} Let $(\R^{4n}, \omega_{\textsf{stn}})$ be the standard symplectic vector space.  Then the following hold:\\
\textsf{1)} Let $H$ be  a linear hypercomplex structure on $\R^{4n}$ such that the corresponding symplectic bases are adapted to $H$, in terms of Definition \ref{EHbases} and let $f : \R^{4n}\to[\E\Hh]$ be the induced isomorphism.  This defines a surjective map from the space of symplectic bases onto the space of all linear hypercomplex structures $H$,  and moreover onto the space of all linear quaternionic structures $Q$ on $\R^{4n}$, such that  
\[
\omega_{\st}=f^{*}\omega_{0}\,,\quad H=f^{*}H_{0}\,,\quad Q=f^{*}Q_{0}\,.
\]
In particular, the pairs $(\omega_{\st}, H)$ and $(\omega_{\st}, Q)$ are linear \textsf{hs-H}/\textsf{qs-H} structures, respectively.\\
\textsf{2)} Two  symplectic bases of $\R^{4n}$ define the same  linear \textsf{hs-H} structure,  if and only if the transition matrix between them is an element of $\SO^*(2n)$, and they define the same   linear \textsf{qs-H} structure, if and only if the transition matrix between them is an element of $\SO^*(2n)\Sp(1).$  
 \end{lem}
 \begin{proof}
By the existence result for  bases adapted to the linear hypercomplex structure $H$ (see Proposition \ref{hypbasis}),  the surjectivity follows. The claims about the stabilizers follow by Proposition \ref{usefrelskewherm}.
  \end{proof}


 \section{Almost hypercomplex/quaternionic skew-Hermitian   structures}\label{sec2}
 The  description of the most basic features of linear $\SO^*(2n)$- and $\SO^*(2n)\Sp(1)$-structures  given in the previous section, enables  us to conveniently investigate $\SO^*(2n)$-type structures on smooth manifolds.  For the convenience of the reader,  let us  begin by refreshing a few basic facts from the theory of $G$-structures (for  more details see \cite{Kob2, Salamon89, Schw, BG08}). 
  
Let us  fix, once and for all, a  connected $4n$-dimensional smooth manifold $M$ and some reference $4n$-dimensional  real vector space  $V$ (which we will use as a model of $T_{x}M$).
The  frame bundle    $\mc{F}=\mc{F}(M)$ of $M$ consists of all  linear  isomorphisms  between the tangent space $T_{x}M$ of $M$  at $x\in M$ and $V$,  which we view as  \textsf{co-frames} $u :   T_{x}M\to V$.  The frame bundle $\mc{F}$  is a principal $\Gl(V)$-bundle over $M$.  A \textsf{$G$-structure}   on   $M$ is  defined  to be  a reduction of  the frame bundle  to  a  closed Lie subgroup $G\subset\Gl(V)$, i.e.,  a sub-bundle $\mathcal{P}\subset\mc{F}(M)$ with structure group $G$.

 Let  $\pi : \mc{P}\to M$  be a   $G$-structure on $M$ and let $\vartheta\in\Omega^{1}(\mc{P}, V)$ be the  \textsf{tautological 1-form} on $\mc{P}$, defined by  $\vartheta(X)=u(\pi_{*}(X))$ for any co-frame $u\in  \mc{P}$ and $X\in T_{u} \mc{P}$.  The tautological form is   \textsf{strictly horizontal},  in the sense that the kernel of $\vartheta$ coincides with the vertical sub-bundle $T^{\mf{ver}}\mc{P}$ of the tangent bundle $T\mc{P}$, and $G$-equivariant, i.e., $r_{a}^{*}\vartheta=a^{-1}\vartheta$ for any $a\in G$,  where $r_{a}$ denotes the right translation by an element  $a\in G$. 
Such 1-forms may characterize a $G$-structure, in particular  under our assumptions,  a $G$-structure on $M$ is equivalently defined to be a principal $G$-bundle $\pi : \mc{P}\to M$ over $M$ together with a  1-form $\vartheta\in\Omega^{1}(P, V)$ such that  $\ker\vartheta=\ker\dd\pi=T^{\mf{ver}}\mc{P}$   and  $r_{a}^{*}\vartheta=a^{-1}\vartheta$, for any $a\in G$. This definition enables generalizing the notion of $G$-structures to the case where $\rho : G\to \Gl(V)$ is a covering of a closed subgroup of $\Gl(V)$ (like the case of spin structures), by assuming $r_{a}^{*}\vartheta=\rho(a)^{-1}\vartheta$.

 Let us now recall the following  examples of $G$-structures which we will use frequently below.

\begin{defi}
\textsf{1)} 
An \textsf{almost  hypercomplex structure} on $M$ is a $G$-structure with  $G=\Gl(n, \Hn)$.  This mean that $M$ admits  a triple $H=\{J_{a} : a=1, 2, 3\}$ of smooth endomorphisms  $J_{a}\in\Ed(TM)$ satisfying the quaternionic identity    $J_{1}^2=J_{2}^2=J_{3}^2=-\Id=J_{1}J_{2}J_{3}$. 
Any almost hypercomplex structure $H$ induces a linear hypercomplex structure  $H_{x}$ at each $T_{x}M$, which  establishes a linear isomorphism $(T_{x}M\,, H_x)\cong ([\E\Hh], H_{0})$. Such a pair $(M, H)$  is said to be an \textsf{almost hypercomplex manifold}. Note that $\Aut(H)\cong\Gl(n, \Hn)$,  and in this case the reduction  bundle of the frame bundle of $M$ consists of all bases of $T_{x}M$ adapted to $H_{x}$. Such a basis induces a linear hypercomplex  isomorphism $u :   T_{x}M\to [\E\Hh]$. \\ 
\textsf{2)}  An \textsf{almost quaternionic structure} on $M$  is a $G$-structure with   $G=\Gl(n,\Hn)\Sp(1)$. This means that $M$ admits a rank-3 smooth sub-bundle $Q\subset \Ed(TM)\cong T^*M\otimes TM$ which is locally generated by an almost hypercomplex structure $H=\{J_{a} : a=1, 2,3\}$.  Such a locally defined triple $H$ is called a (local) \textsf{admissible frame} of $Q$, and the pair $(M, Q)$ is called  an 
\textsf{almost quaternionic manifold}.    Note that $\Aut(Q)\cong\Gl(n,\Hn)\Sp(1)$, and  in this case the reduction  bundle of the frame bundle of $M$ consists of  all bases of $T_{x}M$ adapted to some admissible basis $H_x$ of $Q_x$.  Such a basis induces a linear quaternionic  isomorphism $u :   T_{x}M\to [\E\Hh]$.\\
\textsf{3)}  An \textsf{almost symplectic structure} on $M$  is a $G$-structure  with $G=\Sp(4n,\R)$.  This means that $M$ admits a non-degenerate   2-form $\omega$, called an \textsf{almost symplectic form}. Such a pair $(M, \omega)$ is referred to as an \textsf{almost symplectic manifold}. Note that $\Aut(\omega)\cong\Sp(4n, \R)$, and in this case the reduction of $\mc{F}(M)$ to $\Sp(2n,\R)$ consist all  symplectic bases of $(T_{x}M, \omega_x)$. Similarly, such a basis induces a linear symplectomorphism $u :   T_{x}M\to\R^{4n}$, where we consider $\R^{4n}$ as endowed with the standard  linear symplectic form $\omega_{\textsf{stn}}$.
  \end{defi}
 Let us also recall the following operators, which are naturally defined  on any almost symplectic manifold $(M,\omega)$.
 \begin{defi}\label{asymsym}
 Let $L_{X}\in\Ed(TM)$ be an endomorphism on an almost symplectic manifold $(M, \omega)$, induced by  a vector-valued smooth 2-form $L$ on $M$, that is $L_{X}=L(X, \cdot)$ for any $X\in \Gamma(TM)$. Then:\\
 \textsf{1)} The \textsf{symplectic transpose} $L^{T}_{X}$ of $L_{X}$ with respect to  $\omega$ is defined by
\[
\omega(L_{X}^{T}Y, Z)+\omega(Y, L_{X}Z)=0\,,\quad \forall \  X, Y, Z \in \Gamma(TM)\,.
\]
  \textsf{2)} The operator of \textsf{symmetrization}/\textsf{antisymmetrization}  of $L_{X}$ with respect to the symplectic transpose  $L_{X}^{T}$, is respectively defined by
 \[
 Sym(L_{X}):=\frac{1}{2}\big(L_X+L_{X}^{T}\big)\,,\quad   Asym(L_{X}):=\frac{1}{2}\big(L_X-L_{X}^{T}\big)\,.
  \]
 \end{defi}
  
  \noindent
  \textsf{A remark of caution:}  Below we shall use the $\E\Hh$-formalism, where it is appropriate to  emphasize on the role of topology of the smooth manifold $M$ admitting an almost quaternionic structure $Q\subset\Ed(TM)$. This is because not everything from $\E\Hh$-formalism extends globally to a manifold setting, see also \cite{MR, Salamon82, Salamon86}.  Recall first that the quaternionic structure $Q$ is naturally identified with an associated bundle over $M$ with fiber $[S^2\Hh]^*$,   via the canonical section $\omega_{\Hh}$ of the associated bundle with  fiber $[\Lambda^2\Hh]^*$. 
However, since $\Gl(n,\Hn)\Sp(1)$ is a quotient of $\Gl(n,\Hn)\times \Sp(1)$, the  bundle analogies of the modules $\E$ and $\Hh$ are not necessarily globally defined over $M$.  In the second part of this work we describe the analogous  result for  $\SO^*(2n)\Sp(1)$-structures,  see  the appendix in  \cite{CGWPartII}.  Obviously, another obstruction is the \textsf{global trivializability} of $Q$.  Hence, an admissible frame $H=\{I,J,K\}$ of $Q$, or the vector bundles associated to a  $\Gl(n,\Hn)\Sp(1)$-structure via the representations  $\E$ and $\Hh$,  may  not exist globally. However, note that the  projectivization $\mathbb{P}(\Hh)$ of $\Hh$ globally exists and provides us with the \textsf{twistor bundle} (or \textsf{unit sphere bundle})   $Z\to M$ associated to any almost quaternionic manifold $(M, Q)$.  As an example, note that there are manifolds (e.g., the quaternionic projective space), which   can not carry a $\Gl(n,\Hn)$-structure, but admit a  $\Gl(n,\Hn)\Sp(1)$-structure.  We partially examine the topology of $\SO^*(2n)$-  and $\SO^*(2n)\Sp(1)$-structures in \cite{CGWPartII}.

\subsection{Scalar 2-forms}\label{mandef}
Let us now  introduce $\SO^{\ast}(2n)$-type structures on smooth manifolds. This topic will constitute the core of this article. From now on, we may fix   $V=[\E\Hh]$ and assume that $n>1$.  We begin with  the following definition, as the analogue of Definition \ref{scalardef} of scalar 2-forms on manifolds.  

\begin{defi} Let  $M$ be a smooth connected manifold.  \\
\textsf{1)} Assume that $M$ admits an almost hypercomplex structure $H=\{J_{a} : a=1, 2, 3\}$. Then, a  smooth 2-form $\omega\in\Gamma(\Lambda^{2}T^*M)$ is called \textsf{$H$-Hermitian}, if $\omega_{x}$ is  Hermitian with respect to $H_{x}=\{(J_{a})_{x} : a=1, 2, 3\}$ in terms of Definition \ref{basicsright}, for any $x\in M$. An everywhere non-degenerate $H$-Hermitian 2-form $\omega\in\Gamma(\Lambda^{2}T^*M)$ will be called a \textsf{scalar 2-form} (with respect to $H$) on $M$.   \\
\textsf{2)} Assume that $M$ admits an almost quaternionic structure $Q$. Then, a  smooth  2-form $\omega\in\Gamma(\Lambda^{2}T^*M)$ is called \textsf{$Q$-Hermitian}, if $\omega_{x}$ is  Hermitian with respect to $Q_{x}$ in terms of Definition \ref{basicsright}, for any $x\in M$. An everywhere non-degenerate $Q$-Hermitian 2-form $\omega\in\Gamma(\Lambda^{2}T^*M)$ will be called a \textsf{scalar 2-form} (with respect to $Q$) on $M$.
\end{defi}
 
As a consequence of Proposition    \ref{adaptbas} we may now pose the following characterization of smooth scalar  2-forms in the manifold setting.
 \begin{corol}\label{usefscalar}
 \textsf{1)}  Let  $(M, H=\{J_{a} : a=1, 2, 3\})$ be an almost hypercomplex manifold. Then, a real-valued  smooth 2-form $\omega\in\Gamma(\Lambda^{2}T^*M)$ is a smooth scalar 2-form, if and only if  there is a linear hypercomplex skew-Hermitian structure $(H_x,h_x)$ on $T_xM$ for any $x\in M$ such that $\omega_x=\Re(h_x)$, i.e., $\omega_x$ is a scalar 2-form on $T_xM$ (with respect to $H_x$), for any $x\in M$.\\
\textsf{2)}  Let  $(M, Q)$ be an almost quaternionic manifold. Then a real-valued  smooth 2-form $\omega\in\Gamma(\Lambda^{2}T^*M)$ is a smooth scalar 2-form, if and only if  there is a linear quaternionic skew-Hermitian structure $h_x$ on $T_xM$ for any $x\in M$ such that $\omega_x=\Re(h_x)$ and $Q_x$ is the quaternionic structure induced by $\Im(h_x)$,  i.e., $\omega_x$ is  a scalar 2-form on $T_xM$ (with respect to $Q_x$), for any $x\in M$.\\
\textsf{3)} Let $(M, H=\{J_{a} : a=1, 2, 3\})$ or  $(M, Q)$ be as above. Then, the set of smooth scalar 2-forms on $M$ defines a  sub-bundle of $\Lambda^{2}T^{*}M$, which we denote by $\Lambda^{2}_{\sc}T^{*}M$.
 \end{corol}
 \begin{rem}
 \textnormal{Note that $\Lambda^{2}_{\sc}T^*M\subset\Lambda^{2}_{\Re(\Hn)}T^*M$ is {\it not} a vector sub-bundle of $\Lambda^{2}T^*M$, due to the requirement that smooth scalar 2-forms are non-degenerate.  The reader may consult Proposition \ref{usefrel1} to derive further equivalent characterizations of scalar 2-forms on smooth manifolds.}
 \end{rem}
We can now proceed by introducing the geometric structures that we are mainly interested in.  
\begin{defi}
\textsf{1a)}  An \textsf{almost hypercomplex skew-Hermitian structure} $(H, \omega)$ on a  $4n$-dimensional  manifold $M$ (\textsf{almost hs-H structure} for short)  consists of an almost hypercomplex structure $H=\{J_{a} : a=1, 2, 3\}$ and  a  smooth scalar 2-form $\omega \in \Gamma(\Lambda^{2}_{\sc}T^*M)$  (with respect to $H$).  A manifold $M$ endowed with an almost hypercomplex skew-Hermitian structure will be referred to as an \textsf{almost hypercomplex skew-Hermitian manifold} (\textsf{almost hs-H manifold} for short), and denoted by $(M, H, \omega)$.\\
\textsf{1b)} A \textsf{hypercomplex symplectomorphism} $f :  (M, H, \omega)\to (\hat{M}, \hat{H}, \hat{\omega})$ between two  almost hypercomplex skew-Hermitian manifolds is a diffeomorphism $f : M\to \hat{M}$ satisfying $H=f^*\hat{H}$ and $\omega=f^*\hat{\omega}$. \\
\textsf{2a)}  An \textsf{almost quaternionic skew-Hermitian structure} $(Q, \omega)$ on a  $4n$-dimensional  manifold $M$  (\textsf{almost qs-H structure} for short)   consists of an almost quaternionic structure $Q\subset\Ed(TM)$ and a smooth scalar 2-form $\omega \in \Gamma(\Lambda^{2}_{\sc}T^*M)$  (with respect to $Q$).   A manifold $M$ endowed with an almost  quaternionic skew-Hermitian structure will be referred to as an \textsf{almost quaternionic skew-Hermitian manifold} (\textsf{almost qs-H manifold} for short), and denoted by $(M, Q, \omega)$.\\
\textsf{2b)} A \textsf{quaternionic symplectomorphism} $f :  (M, Q, \omega)\to (\hat{M}, \hat{Q}, \hat{\omega})$ between two  almost quaternionic skew-Hermitian manifolds is a diffeomorphism $f : M\to \hat{M}$ satisfying $Q=f^*\hat{Q}$ and $\omega=f^*\hat{\omega}$. 
\end{defi}

 From this definition it's obvious that 
\begin{itemize}
\item An \textsf{almost hs-H manifold} $(M, H, \omega)$ is already an almost hypercomplex manifold $(M, H)$.
\item  An \textsf{almost qs-H manifold} $(M, Q, \omega)$ is  already an almost quaternionic manifold $(M, Q)$.
\end{itemize}
Therefore, the structures that we treat are special subclasses of almost hypercomplex/quaternionic structures, endowed with a bit more structure provided by the smooth scalar 2-form $\omega$. In particular, they form the symplectic counterparts of almost  hH structures and almost qH structures, respectively, which are almost hypercomplex/quaternionic structures endowed with a (pseudo)-Riemannian metric which is Hermitian with respect to $H$ and $Q$, respectively. 
  On the other side, 
\begin{itemize}
\item Any \textsf{almost hs-H} manifold $(M, H, \omega)$ or  any  \textsf{almost qs-H} manifold $(M, Q, \omega)$ is already an almost symplectic manifold $(M, \omega)$.
\end{itemize}
Hence,  one may start with an almost symplectic structure $\omega$ and look for ``compatible''    almost hypercomplex/quaternionic structures, in the sense that we require $\omega$ to be a scalar 2-form with respect to such an almost hypercomplex/quaternionic structure.

\smallskip
Next we  describe the bundle reductions corresponding to such $G$-structures.
By Proposition \ref{adaptbas} we  obtain the following characterization.
 \begin{prop}\label{hsH}
\textsf{1)} An \textsf{almost hs-H manifold} is  a  $4n$-dimensional connected manifold $M$, whose frame bundle $\mc{F}=\mc{F}(M)$  admits a reduction to $\SO^{\ast}(2n)\subset\Gl([\E\Hh])$, namely $\mc{P}=\sqcup_{x\in M}\mc{P}_{x},$ where
\[
\mc{P}_{x}:=\big\{u:=(e_1, \ldots, e_{2n}, f_1, \ldots, f_{2n}) : u \ \text{skew-Hermitian basis of} \ (T_{x}M, H_{x}, \omega_{x}) \big\}\,.
\]
Thus, $\mc{P}$ is a principal $\SO^{\ast}(2n)$-bundle over $M$, and we can identify
  \[
\mathcal{F}=\mathcal{P}\times_{\SO^{*}(2n)}\Gl([\E\Hh])\,,\quad\text{and}\quad TM=\mc{F}\times_{\Gl(V)}V\cong \mathcal{P}\times_{\SO^*(2n)}[\E\Hh]\,.
\] 
\textsf{2)} The set of $\SO^{\ast}(2n)$-structures on $M$ coincides with the space of sections of the quotient bundle  
\[
\mathcal{F}/\SO^{\ast}(2n)=\mathcal{F}\times_{\Gl(V)}\big(\Gl(V)/\SO^{\ast}(2n)\big)
\] 
with  typical fiber isomorphic to the  space $\Gl([\E\Hh])/\SO^{\ast}(2n)$.
 \end{prop}
Observe that the existence of  global sections of $\mathcal{F}/\SO^{\ast}(2n)$ is purely topological in nature.  
\begin{example}
\textnormal{Consider $M=[\E\Hh]$. Then,  $M$ is an \textsf{almost hs-H manifold}, and the space of almost  \textsf{hs-H}-structures   coincides with the space of functions from $[\E\Hh]$ to $\Gl([\E\Hh])/\SO^{\ast}(2n)$.   At  $x\in M$ such a function describes a linear transformation from $(H_x, \omega_x)$ to $(H_{0}, \omega_{0})$}.  
\end{example}

Similarly, by the discussion in Section   \ref{mainlinear} and Proposition   \ref{adaptbas} we obtain an analogous statement for  \textsf{almost qs-H structures}.
\begin{prop}\label{qsH}
\textsf{1)} An \textsf{almost qs-H manifold} is  a  $4n$-dimensional connected manifold $M$, whose frame bundle $\mc{F}=\mc{F}(M)$  admits  a reduction to $\SO^{\ast}(2n)\Sp(1)\subset\Gl([\E\Hh])$, namely $\mc{Q}=\sqcup_{x\in M}\mc{Q}_{x}\subset \mc{F}$, where 
\[
\mc{Q}_{x}:=
\left\{
\begin{tabular}{cc}
$ u:=(e_1, \ldots, e_{2n}, f_1, \ldots, f_{2n})$  & $ : u \ \text{skew-Hermitian basis of} \ (T_{x}M, Q_{x}=\langle H_{x}\rangle, \omega_{x})$ \\
& $ \text{for all admissible bases} \ H_x \ \text{of} \ Q_{x}$
\end{tabular}
\right\}.
\]
Thus, $\mc{Q}$ is a principal $\SO^{\ast}(2n)\Sp(1)$-bundle over $M$, and we can identify
  \[
\mathcal{F}=\mathcal{Q}\times_{\SO^{*}(2n)\Sp(1)}\Gl([\E\Hh])\,,\quad\text{and}\quad TM=\mc{F}\times_{\Gl(V)}V\cong \mathcal{Q}\times_{\SO^*(2n)\Sp(1)}[\E\Hh]\,.
\] 
\textsf{2)} The set of $\SO^{\ast}(2n)\Sp(1)$-structures on $M$ coincides with the space of sections of the quotient bundle  
\[
\mathcal{F}/\SO^{\ast}(2n)\Sp(1)=\mathcal{F}\times_{\Gl(V)}\big(\Gl(V)/\SO^{\ast}(2n)\Sp(1)\big)
\] 
with  typical fiber isomorphic to the  space $\Gl([\E\Hh])/\SO^{\ast}(2n)\Sp(1)$.
 \end{prop}

\begin{defi}\label{skewHframe}
\textsf{1)}  Let $(M, H, \omega)$ be an  \textsf{almost hs-H manifold}. A (local) section of the  principal $\SO^*(2n)$-bundle $\mc{P}\to M$ given in Proposition \ref{hsH} is said to be a (local) \textsf{skew-Hermitian frame}.\\
\textsf{2)}  Let $(M, Q, \omega)$ be an  \textsf{almost qs-H manifold}. A (local) section of the  principal $\SO^*(2n)\Sp(1)$-bundle $\mc{Q}\to M$ given in Proposition \ref{qsH} is said to be a (local) \textsf{skew-Hermitian frame} with respect to some local admissible frame $H$ of $Q$. Note that   the local admissible frame $H$ is uniquely determined by the corresponding  skew-Hermitian frame.
\end{defi}
 
\smallskip
 According to Proposition \ref{usefrel1},  the $\SO^*(2n)\Sp(1)$-module of linear scalar 2-forms on $[\E\Hh]$,  denoted by $\Lambda^{2}_{\sc}[\E\Hh]^*$,  is  the set of non-degenerate elements inside $[S^2\E]^*=[S^2_0\E]^*\oplus\langle\omega_0\rangle$.
Hence, given an almost quaternionic skew-Hermitian manifold $(M, Q, \omega)$  with  reduction $\mc{Q}\subset\mc{F}$ described in Proposition \ref{qsH}, the scalar 2-form $\omega$ or any other smooth scalar 2-form can be viewed as a smooth section of the following  associated  bundle: 
 \begin{eqnarray*}
 \Lambda^2_{\sc}T^{*}M&=&\mc{Q}\times_{\SO^*(2n)\Sp(1)}\Lambda^{2}_{\sc}[\E\Hh]^*\\
 &=&\mc{Q}\times_{\SO^*(2n)\Sp(1)}\Lambda^{2}_{\sc, 0}[\E\Hh]^*\oplus \mc{Q}\times_{\SO^*(2n)\Sp(1)}\R^{\times} \omega_{0}
  \\
 &=&\mc{Q}\times_{\SO^*(2n)\Sp(1)}\Lambda^{2}_{\sc, 0}[\E\Hh]^*\oplus\mc{L}_{\omega_{0}}\,.
 \end{eqnarray*}
Here the notation $\Lambda^{2}_{\sc, 0}[\E\Hh]^*$ encodes the non-degenerate elements in $[S^2_{0}\E]^*$,  and 
 \[
 \mc{L}_{\omega_{0}}= \mc{Q}\times_{\SO^*(2n)\Sp(1)}\R^{\times} \omega_{0}
 \]
  is a line bundle without zero sections. 
 Due to Propositions  \ref{usefrel1} and   \ref{hsH},  the reader can describe a similar decomposition of the space of scalar 2-forms $ \Lambda^2_{\sc}T^{*}M$ associated to  an \textsf{almost hs-H manifold} $(M, H, \omega)$.
 
\begin{rem}
	\textnormal{The line bundle $\mc{L}_{\omega_{0}}$ defined above defines the (almost) conformal symplectic version of $\SO^*(2n)\Sp(1)$-structures, which was thoroughly discussed by Sala\v{c} and \v{C}ap \cite{Cap}. For a \textsf{qs-H} manifold $(M, Q, \omega)$, the scalar 2-form $\omega$ defines a global section of $\mc{L}_{\omega_{0}}$, which does not have to exist globally in the conformal symplectic setting (in this case the corresponding structure group is  the Lie group $\SO^*(2n)\Gl(1, \Hn)$).  This is responsible for the differences between our results and the results in \cite{Cap}.}
\end{rem}
Finally, note that  locally an \textsf{almost qs-H structure} on a manifold $M$ can be understood in terms of an \textsf{almost hs-H structure}, although globally the situation differs. Of course, this establishes an analogue  with  the local relation of almost quaternionic structures and almost hypercomplex structures.   Let us summarize this phenomenon as follows:
 \begin{prop}
  Let $(M,  Q, \omega)$ be an \textsf{almost qs-H manifold} and let $x\in M$ be some point of $M$. Then, there exists an open neighbourhood $U\subset M$ of $x$ and an \textsf{almost hs-H structure} $(H, \omega')$ defined on $U$, such that
\[
\omega|_{U}=\omega', \quad Q|_{U}=\langle H\rangle.
\]
\end{prop}
 Explicit  constructions providing  examples of the structures introduced  above,  are analyzed in the second part of this work, in particular see Sections 3, 4 and 5 in \cite{CGWPartII}.  In the final section of this first part, we have collected details related to   \textsf{torsion-free} (or \textsf{integrable}) examples (see below Section \ref{secintrinsic} for adapted connections and intrinsic torsion).

\subsection{The quaternionic skew-Hermitian form and the fundamental 4-tensor}
Since for $n>1$ the Lie group $\SO^{*}(2n)$ is non-compact,  in principle there is {\it no}  underlying Riemannian metric structure on a manifold $M$ with a $\SO^{\ast}(2n)$-structure, and similarly for $G$-structures with $G=\SO^{\ast}(2n)\Sp(1)$.  However,    given any \textsf{almost hs-H manifold} $(M, H=\{I, J, K\}, \omega)$ by Propositions \ref{signprop} and  \ref{adaptbas} we may introduce three pseudo-Riemannian metrics of signature $(2n, 2n)$, defined by
\[
g_{I}(X, Y):=\omega(X, IY)\,,\quad g_{J}(X, Y):=\omega(X, JY)\,,\quad g_{K}(X, Y):=\omega(X, KY)\,,
\]
for any $X, Y\in\Gamma(TM)$.   Note that these metrics are $\SO^{*}(2n)$-invariant, as it is claimed in  Remark \ref{SOninvariantmetrics}.   
The above tensors are also obtained in the case of an  \textsf{almost qs-H manifold} $(M, Q, \omega)$, where  $H=\{I, J, K\}$ is a local admissible frame of $Q$, but they are only locally defined.  In particular, for any local section $\J\in \Gamma(Z)$ the tensor $g_{\J}(X, Y):=\omega(X, \J Y)$ is a locally defined tensor.  Nevertheless, via the assignment $\J\mapsto g_{\J}$  we obtain a global embedding of the twistor bundle   $Z\to M$ into $S^2T^*M$.
 
Let us now construct some globally defined tensors, which allow us to provide alternative definitions of the structures that we  are interested in. Indeed,  by Proposition \ref{usefrelskewherm}  on  an \textsf{almost qs-H manifold} $(M, Q, \omega)$ we obtain a globally defined tensor $h$, given by   
\begin{equation}
 	h:=\omega\id_{TM}+ g_{I}I+g_{J}J+g_{K}K\in  \Gamma\big(T^{*}M\otimes T^{*}M\otimes \Ed(TM)\big)\,,
	\label{qhermtensor}
\end{equation}
 where  $\{I,J,K\}$ is an arbitrary local admissible frame of $Q$. In particular, \[
h(X,Y)Z=\omega(X,Y)Z++ g_{I}(X,Y)IZ+g_{J}(X,Y)JZ+g_{K}(X,Y)KZ
\] for any $X,Y,Z\in\Gamma(TM).$
Since each $(T_xM,h_x)$ is a vector space with a \textsf{linear qs-H structure} as defined in Definition \ref{def2}, we shall refer to $h$ by the term  \textsf{quaternionic skew-Hermitian form} associated to the \textsf{almost qs-H structure} $(Q, \omega)$.  Observe that $h$ is defined even for  the case where  $(M, H, \omega)$ is an \textsf{almost hs-H manifold}. However, note that  $h$   is actually stabilized   by the larger  group $\SO^*(2n)\Sp(1)$, and we may pose the following characterization:
\begin{corol}\label{fund2tensorQ}
A $4n$-dimensional connected smooth manifold $M$ admits a  $\SO^{\ast}(2n)\Sp(1)$-structure, if and only if  admits a smooth (1, 3)-tensor $h$ which in a local frame  of $TM$ is given by the tensor $h_{0}$ of Proposition \ref{usefrelskewherm}.
\end{corol}
Of course, this corollary may serve as an alternative way to define $\SO^{\ast}(2n)\Sp(1)$-structures. 
 Similarly,  by Proposition \ref{usefrelfund}  on  $(M, Q, \omega)$, we obtain a globally defined 4-tensor $\Phi$, given by   
\begin{equation}
	 \Phi:=g_{I}\odot g_{I}+g_{J}\odot g_{J}+g_{K}\odot g_{K}={\sf Sym}(g_{I}\otimes g_{I}+g_{J}\otimes g_{J}+g_{K}\otimes g_{K})\in \Gamma\big(S^{4}T^{*}M\big)\,,
	\label{fundtensor}
\end{equation}
where    ${\sf Sym} : {\mc{T}}^{4}T^{*}M\to S^{4}T^{*}M$  denotes the operator of complete symmetrization at the bundle level, and $\{I,J,K\}$ is an arbitrary local admissible frame of $Q$.
We  call $\Phi$ the \textsf{fundamental 4-tensor field} associated to the \textsf{almost  qs-H structure} $(Q, \omega)$.  Again, $\Phi$ is defined even for  the case where  $(M, H, \omega)$ is an \textsf{almost hs-H manifold}. However,    similarly to  $h$ above, note that   $\Phi$ is actually stabilized   by the larger group $\SO^*(2n)\Sp(1)$, so similarly we get the following
\begin{corol}\label{fund4tensorQ}
A $4n$-dimensional connected smooth manifold $M$ admits a  $\SO^{\ast}(2n)\Sp(1)$-structure, if and only if  admits a symmetric 4-tensor  $\Phi$ which in a local frame  of $TM$ is given by the tensor $\Phi_{0}$ of Proposition \ref{usefrelfund}.
\end{corol}
 Corollary \ref{fund4tensorQ}, just like Corollary \ref{fund2tensorQ},  can be used to   define $\SO^{\ast}(2n)\Sp(1)$-structures in an alternative way, via a global symmetric 4-tensor.
Moreover, since $\Phi$ satisfies
\begin{eqnarray*}
\Phi(X, Y, Z, W)&=&\frac{1}{24}\sum_{\sigma\in S^4}\Big(g_{I}(X_{\sigma(1)}, Y_{\sigma(2)})g_{I}(Z_{\sigma(3)}, W_{\sigma(4)})+g_{J}(X_{\sigma(1)}, Y_{\sigma(2)})g_{I}(Z_{\sigma(3)}, W_{\sigma(4)})\\
&&+g_{K}(X_{\sigma(1)}, Y_{\sigma(2)})g_{K}(Z_{\sigma(3)}, W_{\sigma(4)})\Big)
\end{eqnarray*}
for any $X, Y, Z, W\in\Gamma(TM)$, 
we deduce that
\begin{lem}
The fundamental tensor field $\Phi$ satisfies the identities:
\begin{eqnarray*}
\Phi(X, Y, Z, W)&=&\frac{1}{3}\sum_{I\in H}\Big(\fr{S}_{Y, Z, W}g_{I}(X, Y)g_{I}(Z, W)\Big)\\
&=&\frac{1}{3}\Big(g_{I}(X, Y)g_{I}(Z, W)+g_{I}(X, Z)g_{I}(Y, W)+g_{I}(X, W)g_{I}(Y, Z) \\
&&+g_{J}(X, Y)g_{J}(Z, W)+g_{J}(X, Z)g_{J}(Y, W)+g_{J}(X, W)g_{J}(Y, Z) \\
&&+g_{K}(X, Y)g_{K}(Z, W)+g_{K}(X, Z)g_{K}(Y, W)+g_{K}(X, W)g_{K}(Y, Z) \Big)\,,
\end{eqnarray*}
for any $X, Y, Z, W\in\Gamma(TM)$.
\end{lem}


\section{Intrinsic torsion of $\SO^*(2n)$-structures and $\SO^*(2n)\Sp(1)$-structures}\label{secintrinsic}
\subsection{Generalities on adapted connections to $G$-structures}
Let us maintain the notation from the introduction in Section \ref{sec2}, and assume that $M$ is a connected manifold, but not necessarily $4n$-dimensional, and  that $G\subset\Gl(V)$ is a closed subgroup.  Recall that the \textsf{torsion} of a linear connection $\nabla$ on $M$ is a smooth section of the \textsf{torsion bundle} $\Tor(M):=\Lambda^{2}T^*M\otimes TM$ defined by
 \[
 T^{\nabla}(X, Y):=\nabla_{X}Y-\nabla_{Y}X-[X, Y]\,,\quad X, Y\ \in\Gamma(TM)\,.
 \]
 If $T^{\nabla}$ vanishes identically, then $\nabla$ is said to be \textsf{torsion-free}. 
 
Let  $\pi : \mc{P}\to M$  be a   $G$-structure on $M$ with tautological 1-form $\vartheta$.   A linear connection $\nabla$  is called \textsf{adapted} to $\mc{P}\subset\mc{F}$, or simply a \textsf{$G$-connection}, when the corresponding connection on the frame bundle $\mc{F}$ of $M$ reduces to $\mc{P}$. Since the Lie algebra $\fr{g}$ of $G$ can be identified with a subalgebra of $\fr{gl}(V)=\Ed(V)$, one can show that the space of $G$-principal connections on $\mc{P}$ is a space modeled on the space
of smooth sections of the associated bundle $\mc{P}\times_{G}(V^{*}\otimes\fr{g})=T^*M\otimes\fr{g}_{\mc{P}},$ where $\fr{g}_{\mc{P}}$ is the adjoint bundle.     Each  $G$-principal connections on $\mc{P}$ is defined by a connection 1-form $\gamma : T\mc{P}\to\fr{g}$ and induces a $G$-connection $\nabla^{\gamma}$ on $TM$. In particular, there is bijective correspondence between $G$-connections and  $G$-principal connections on $\mc{P}$.
For any  adapted connection  $\nabla=\nabla^{\gamma}$  corresponding to a connection 1-form $\gamma : T\mc{P}\to\fr{g}$
we define its  \textsf{torsion form} $\Uptheta^{\gamma}$, which is the vector-valued 2-form on $\mc{P}$  defined by the structure equation
$\Uptheta^{\gamma}=\dd{\vartheta}+\gamma\wedge{\vartheta}$.
  When the torsion form $\Uptheta^{\gamma}$  vanishes, the $G$-connection corresponding to the connection 1-form $\gamma$ is said to be \textsf{torsion-free}, which is in a line with the  definition above.  
 Indeed, $\Uptheta^{\gamma}$ is   $G$-equivariant and strictly horizontal, so it induces a  smooth $G$-equivariant \textsf{torsion function} $t^{\gamma}:\mc{P}\to \Lambda^{2}V^*\otimes V$, which assigns to a co-frame $u\in \mc{P}$ the coordinates  $t^{\gamma}(u)=u(T^{\nabla})$ of the torsion $T^{\nabla}$ in $\Lambda^{2}V^*\otimes V$.

 \begin{defi}
 A $G$-structure $\mc{P}\subset\mc{F}$ is called a \textsf{torsion-free $G$-structure}, or \textsf{1-integrable}, when it admits a torsion-free adapted connection.
 \end{defi}

 Let us fix a  $G$-structure $\pi : \mc{P}\to M$ on $M$ as above.  Recall that the \textsf{first prolongation} of the Lie algebra $\fr{g}$  of $G$ is defined by
  \[
\fr{g}^{(1)}:= (V^*\otimes\fr{g})\cap(S^2V^*\otimes V)=\{\alpha\in V^*\otimes\fr{g} : \alpha(x)y=\alpha(y)x, \ \forall \ x, y\in V\}\subset\Hom(V, \fr{g})\,.
\]
Note that for any Lie subalgebra $\fr{g}\subset\Ed(V)$ we may consider  the $G$-equivariant map
\[
\delta :  V^{*} \otimes\fr{g} \to \Lambda^{2}V^{*}\otimes V\,, \quad \delta(\alpha)(x, y):=\alpha(x)y-\alpha(y)x\,,
\]
 with $\alpha\in V^{*} \otimes\fr{g}$ and $x, y\in V$.  This is     the {\textsf{Spencer operator of alternation}}, which is actually one of the boundary maps of the Spencer complex of $\fr{g}\subset\Ed(V)$, also called \textsf{Spencer differential}.  It fits into the following exact sequence,
\[
0\longrightarrow \ke\delta \cong\fr{g}^{(1)}\longrightarrow V^{*} \otimes\fr{g}\cong\Hom(V, \fr{g})\overset{\delta}{\longrightarrow} \Lambda^{2}V^{*}\otimes V\cong\Hom(\Lambda^{2}V, V)\longrightarrow  \mf{Coker}(\delta)\cong \mc{H}(\fr{g})\longrightarrow 0
\]
where  we denote by $\mc{H}(\fr{g})\equiv \mc{H}^{0, 2}(\fr{g})$ the following Spencer cohomology of $\fr{g}$:
 \[
\mc{H}(\fr{g}):=\Hom(\Lambda^{2}V, V)/\imm(\delta)=\Lambda^{2}V^*\otimes V/\imm(\delta)\,.
 \] 
 Let us consider the vector bundle
$
\mathscr{H}(\fr{g}):=\mc{P}\times_{G}\mc{H}(\fr{g})\,,
$
called the \textsf{intrinsic torsion bundle} over $M$,  and maintain the same notation for  the bundle map
 \[
 \delta : T^*M\otimes \fr{g}_{\mc{P}} \to \Tor(M)\,,
 \] 
induced by  the Spencer operator $\delta :  V^{*} \otimes\fr{g} \to \Lambda^{2}V^{*}\otimes V$.  There is a natural projection from $\Tor(M)$ to $ \mathscr{H}(\fr{g})$ which we shall denote by  $p :\Tor(M)\to \mathscr{H}(\fr{g})$. In these terms, we have an isomorphism
\[
\mathscr{H}(\fr{g})\cong \textsf{Tor}(M)/\imm(\delta)=\Lambda^{2}T^{*}M\otimes TM/\delta(T^*M\otimes\fr{g}_{\mc{P}})\,,
\] 
where  similarly we maintain the same notation for $\imm(\delta)\subset \Lambda^{2}V^*\otimes V$ and the corresponding sub-bundle   induced in $\textsf{Tor}(M)$. It is not hard to see that the projection of the torsion   $T^{\nabla}$   via $p$  to this quotient bundle is the same for all $G$-connections $\nabla$,  and it  only depends  on the specific $G$-structure.  Hence, one obtains a well-defined section  $\tau:=p(T^{\nabla})\in\Gamma(\mathscr{H}(\fr{g}))$ of $\mathscr{H}(\fr{g})$,  which is an invariant of 1st-order structures with structure group $G$, called  the \textsf{intrinsic torsion} of $\mc{P}$.
\begin{rem}
\textnormal{For a given $G$-structure $\mc{P}$ on $M$, the intrinsic torsion measures the obstruction to the existence of adapted torsion-free connections. In particular, it vanishes if and only if $M$ admits a torsion-free adapted connection, i.e.,  $\mc{P}$ is a  torsion-free $G$-structure. }
\end{rem}

Suppose now that there exists a $G$-invariant complementary space $\mc{D}=\mc{D}(\fr{g})$ of $\imm(\delta)$ inside $\Lambda^{2}V^*\otimes V$ which gives rise to the direct sum decomposition
\begin{equation}\label{Dcompl}
\Lambda^{2}V^*\otimes V=\imm(\delta)\oplus \mc{D}(\fr{g})\,.  
\end{equation}
Usually, one refers to $\mc{D}(\fr{g})$ as a \textsf{normalization condition}, and it is easy to prove that such a normalization condition always exists for reductive $G\subset \Gl(V)$.  However, in general there is neither any natural way of fixing such a complement, nor is it necessarily unique.

A normalization condition $\mc{D}(\fr{g})$ for a given $G$-structure on $M$  determines  a smooth sub-bundle of   $\Tor(M)$, which we denote by $\mathscr{D}(\fr{g})$. This is  isomorphic with the associated vector  bundle $\mc{P}\times_{G}\mc{D}(\fr{g})$. Then, the splitting (\ref{Dcompl}) induces the following bundle decomposition
\[
\Tor(M)=\imm(\delta)\oplus \mathscr{D}(\fr{g})\,. 
\]
 In this case, a $G$-connection $\nabla=\nabla^{\gamma}$ is called  \textsf{minimal with respect to the normalization condition $\mc{D}(\fr{g})$} (or simply a   \textsf{$\mc{D}$-connection} if there is no matter of confusion),   if  $T^{\nabla}$ is smooth section of $\mathscr{D}(\fr{g})$, i.e., its torsion function $t^{\gamma}$ takes values in $\mc{D}(\fr{g})$. Moreover, one can show that 
\begin{corol}
 Let $\pi : \mc{P}\to M$ be a $G$-structure on a smooth manifold $M$, where $G\subset\Gl(V)$ is a closed subgroup,  and let $\mc{D}(\fr{g})$ be a normalization condition. Then the space of all  $\mc{D}$-connections  is an  affine space modeled on smooth sections of the associated bundle $\mc{P}\times_{G}\fr{g}^{(1)}$. Hence, whenever the first prolongation $\fr{g}^{(1)}$ is trivial,  then  the principal $G$-bundle $\mc{P}\subset\mc{F}$ admits  a unique (up to choice of a normalization condition)  minimal connection. 
 \end{corol}


  \subsection{Adapted connections to $\SO^*(2n)$- and $\SO^{*}(2n)\Sp(1)$-structures}\label{adapatcon}
    For the reminder of this section we shall discuss adapted connections to $\SO^*(2n)$- and $\SO^*(2n)\Sp(1)$-structures.
It is well-known  which linear connections preserve an almost symplectic, an almost hypercomplex, or an almost quaternionic structure, separately, see \cite{Vais, AM} and the references therein, and see also below.  Given an \textsf{almost hs-H manifold} $(M,  H,  \omega)$ we want to specify  a linear connection $\nabla^{H,\omega}$  on $M$ which preserves the pair $(H, \omega)$,  that is
\[
\nabla^{H,\omega}\omega=0, \quad\quad \nabla^{H,\omega} H=0\,.
\] 
It is easy to prove these conditions  are equivalent to the following relations 
\begin{eqnarray*}
\nabla^{H, \omega}_{X}\omega(Y, Z)&=&\omega(\nabla^{H, \omega}_{X}Y, Z)+\omega(Y, \nabla^{H, \omega}_{X}Z)\,,\\
\nabla^{H, \omega}_{X}J_{a}(Y)&=&J_{a}(\nabla^{H, \omega}_{X}Y)\,,\quad a=1, 2, 3\,,
\end{eqnarray*}
 for any  $X, Y, Z\in\Gamma(TM)$, where $H=\{J_a : a=1, 2, 3\}$.
Such a connection is  a $\SO^*(2n)$-connection in the above terms, and  one may refer to  $\nabla^{H,\omega}$  by the term \textsf{almost hypercomplex skew-Hermitian connection}.  When $\nabla^{H, \omega}$ is torsion-free, then it is said to be a  \textsf{hypercomplex skew-Hermitian connection}. 

Similarly,  given an \textsf{almost qs-H manifold} $(M, Q, \omega)$, we want to specify a linear connection $\nabla^{Q,\omega}$  on $M$ which preserves the pair $(Q, \omega)$,  that is
\[
	\nabla^{Q,\omega}\omega=0, \quad\text{and}\quad \nabla^{Q,\omega}_X\sigma\in \Gamma(Q),
	\]
for any smooth vector field $X\in\Gamma(TM)$ and smooth section $\sigma\in\Gamma(Q)$. 
Here the second condition   is equivalent to say that (see for example \cite{AM})
\[
\nabla^{Q, \omega}_{X}J_{a}=\varphi_{c}(X)J_{b}-\varphi_{b}(X)J_{c}\,,\quad \forall \ X\in\Gamma(TM), \ a=1, 2, 3\,,
\]
for any cyclic permutation $(a,b,c)$ of $(1,2,3)$, where  $\{J_a :  a=1, 2, 3\}$ is a local admissible basis of $Q$ and  $\varphi_{a}$ are local 1-forms for any $a=1, 2, 3$.
Such a connection is a $\SO^*(2n)\Sp(1)$-connection and one may call $\nabla^{Q,\omega}$ an \textsf{almost quaternionic skew-Hermitian connection}.  When $\nabla^{Q, \omega}$ is torsion-free, then it is said to be a  \textsf{quaternionic skew-Hermitian connection}.

For the above goal, it is convenient to start with a {\it unique} connection that preserves part of the structure and modify it, to preserve all of it. 
The other connections differ from this connection by a endomorphism valued 1-form with values in $\so^*(2n)$, and $\so^*(2n)\oplus\sp(1)$, respectively.    We should mention that on any     \textsf{almost hs-H manifold} $(M, H, \omega)$ or any \textsf{almost qs-H manifold} $(M^{4n}, Q, \omega)$   we may define an  
  orientation induced by the scalar 2-form $\omega$. This is given by the globally defined volume form
\[
\vol:=\omega^{2n}:=\underbrace{\omega\wedge\cdots\wedge\omega}_{2n-\text{times}}\,.
\]
Hence,  we get the following as an immediate corollary.
\begin{corol}
Any \textsf{almost hs-H/qs-H structure} is a \textsf{unimodular structure}  in terms of \cite{AM}.
\end{corol}
We shall make use of this corollary especially for $\SO^*(2n)\Sp(1)$-structures, because among all the \textsf{Oproiu connections} $\nabla^Q$ there is a {\it unique} \textsf{unimodular Oproiu connection} $\nabla^{Q, \vol}$ associated to the pair $(Q, \vol)$, see  below for more details.

 A further goal  is to determine explicitly   normalization conditions which establish $\nabla^{H,\omega},\nabla^{Q,\omega}$ as minimal connections. 
For this task it is useful to  proceed with a detailed description of the torsion corresponding to such structures, and in particular of their intrinsic torsion, which we present in Section \ref{intrisictorsion}, while   minimal connections are presented in Section \ref{minconnections}. 

  \smallskip
 We  begin with details about the first prolongation of $\SO^*(2n)$ and $\SO^*(2n)\Sp(1)$. It is known  by results of Cartan and others (see for example  \cite[p.~113]{Bryant} and \cite{AM}),  that  for $\fr{g}=\fr{so}^{\ast}(2n)$, $\fr{g}=\fr{so}^{\ast}(2n)\oplus\sp(1)$,   $\fr{g}=\fr{gl}(n,\Hn)$ and $\fr{g}=\fr{sl}(n,\Hn)\oplus\sp(1)$ the first prolongation  $\fr{g}^{(1)}$ is trivial, $\fr{g}^{(1)}=\{0\}$.  
 However, next we demonstrate this result for $\fr{so}^*(2n)$.. In particular, we will provide a proof of the statement $\ker(\delta)=\fr{so}^{\ast}(2n)^{(1)}=\{0\}$,  based explicitly on the geometric properties of the tensor fields defining a $\SO^*(2n)$-structure, proving in this way also compatibility of our new definitions with previous considerations by other authors.

	\begin{lem}\label{zeroprog1}
	Let  $(H, \omega)$ be an \textsf{almost hs-H  structure}, or let $(Q, \omega)$ be an \textsf{almost qs-H structure} on a $4n$-dimensional   manifold $M$. Let $\fr{g}$ be one of the Lie algebras $\fr{so}^*(2n)$ or $\fr{so}^{*}(2n)\oplus\fr{sp}(1)$. 
Then, the Spencer  operator of alternation
 \[
 \delta : \Hom([\E\Hh], \fr{g})= [\E\Hh]^{*}\otimes \fr{g}\to \Hom(\Lambda^{2}[\E\Hh], [\E\Hh])= \Lambda^{2}[\E\Hh]^{*}\otimes [\E\Hh]
\]
is injective, 
 $\fr{g}^{(1)}=\ker(\delta)=\{0\}$.
 \end{lem}
\begin{proof}  As it is mentioned above, we shall prove the statement for $\fr{g}=\fr{so}^{*}(2n)$ only.  Of course, the  vanishing of $\fr{g}^{(1)}$ is a direct consequence of the  embedding  of  $\fr{g}=\fr{so}^{\ast}(2n)$  in $\fr{so}(p, q)$ in  combination with the relation $\fr{so}(p, q)^{(1)}=\{0\}$, see    \cite[p.~113]{Bryant}.  To provide an alternative  proof,
let $\alpha : [\E\Hh]\to\fr{so}^{*}(2n)$ belonging to  the kernel  of $\delta$, that is $\alpha_{X}Y=\alpha_{Y}X$.  Since $\alpha_{X}\in\fr{so}^{*}(2n)$, we have 
\begin{eqnarray}
\alpha_{J_{a}X}Y&=&J_{a}(\alpha_{X}Y), \quad \text{for any} \ \  J_{a}\in H, \ (a=1, 2, 3)\,, \label{so2nalg}\\
\omega(\alpha_{X}Y, Z)&=&\omega(Y, \alpha_{X}Z)\,, \label{so2nalg2}
\end{eqnarray}
for any three vectors $X, Y, Z\in[\E\Hh]$. Since $\alpha_{X}Y=\alpha_{Y}X$  for any  $X, Y$, by    (\ref{so2nalg}) we obtain
\[
J_{a}(\alpha_{Y}X)=J_{a}(\alpha_{X}Y)=\alpha_{J_{a}X}Y=\alpha_{Y}(J_{a}X), \quad  \forall \  a=1, 2, 3,
\]
or in other words $(\ref{so2nalg})$  is equivalent to $\alpha_{X}(J_{a}Y)=J_{a}(\alpha_{X}Y)$, for any  $a=1, 2, 3$, and $X, Y\in [\E\Hh]$.
Moreover, for any triple $I, J, K=IJ$, where $I, J\in H$  are two anticommuting  almost complex structures, we  infer that  $I(\alpha_{X}JY)=\alpha_{IX}(JY)=\alpha_{JY}(IX)$. To see this, in the relation $\alpha_{IX}Y=I\alpha_{X}Y$, replace $X$ by $IX$ and $Y$ by $JY$, and next multiply both sides of the  relation  by $I$.   By  combining these relations, it is now easy  to prove that $K(\alpha_{X}Y)=0$, which obviously implies the vanishing of $\alpha$.   On the other hand, for some $\alpha\in\ker(\delta)$,  by  Proposition \ref{usefrel1}  and the non-degeneracy of $\omega$, we also see  that  (\ref{so2nalg2}) is   equivalent to the condition $\alpha_{Z}(J_{a}Y)=\alpha_{Y}(J_{a}Z)$,  or equivalent to $\alpha_{J_{a}Y}Z=\alpha_{Y}(J_{a}Z)$ for any $a=1, 2, 3$, and hence it coincides with the first condition (\ref{so2nalg}).    
\end{proof}
Lemma \ref{zeroprog1}  has several direct but important consequences, which we summarize in a corollary.
\begin{corol}\label{uniqueminimal} Let $M$ be a $4n$-dimensional connected smooth manifold. Then,\\
 \textsf{1)} An adapted connection $\nabla$ to a   $\SO^{*}(2n)$-structure,  or to a $\SO^{*}(2n)\Sp(1)$-structure  on $M$, or respectively a $\SO^{*}(2n)$- or  a $\SO^{*}(2n)\Sp(1)$-connection, is entirely determined  by its torsion $T^{\nabla}$.\\
 \textsf{2)}  A torsion-free $\SO^{*}(2n)$-connection, or $\SO^{*}(2n)\Sp(1)$-connection, if any, is unique.\\
 \textsf{3)} Let $\fr{g}$ be one of the Lie algebras $\fr{so}^{*}(2n)$ or  $\fr{so}^{*}(2n)\oplus\fr{sp}(1)$, and let $\mc{D}=\mc{D}(\fr{g})$ be a normalization condition related to such a $G$-structure on $M$, where $G$ denotes the connected Lie group corresponding to $\fr{g}$. Then, a $\mc{D}$-connection, or in other words a minimal connection of such a $G$-structure  on $M$  with respect to the normalization condition $\mc{D}(\fr{g})$, is unique.
\end{corol}
We should emphasize on the fact   that similarly to the case of  almost hypercomplex structures (\cite{Gau}), such normalization conditions $\mc{D}(\fr{g})$ are not unique (due to multiplicities of the involved representations). This provides a certain difficulty to the description of the (local) geometry associated to  $\SO^{*}(2n)$- and $\SO^{*}(2n)\Sp(1)$-structures on $4n$-dimensional smooth manifolds.

Let $(M,  H=\{J_{a}\}, \omega)$ be an \textsf{almost hs-H manifold}. By Obata  \cite{Obata} it is known that there is a {\it unique minimal} affine connection  $\nabla^{H}$ with respect to a certain normalization condition, preserving the almost hypercomplex structure  $H=\{J_a : a=1, 2, 3\}$, that is $\nabla^{H}J_{a}=0$ for any $a=1, 2, 3$.   We will refer to this connection as the \textsf{Obata connection}. 

\begin{prop}\label{charcan}
Let $\nabla_{X}Y=\nabla^{H}_{X}Y+A_{X}Y=\nabla^{H}_{X}Y+A(X, Y)$ be a linear connection on an \textsf{almost hs-H manifold} $(M, H, \omega)$, where $\nabla^{H}$ is the Obata connection and $A$ is a smooth tensor field on $M$ of type $(1, 2)$.  Then, $\nabla$ satisfies the conditions
\[
\nabla\omega=0\,, \quad\text{and}\quad \ \nabla J_{a}=0\,, \ \ \   \forall \ a=1, 2, 3\,,
\] 
if and only if  the following two relations hold for any $X, Y, Z\in\Gamma(TM)$:
\begin{eqnarray}
(\nabla^{H}_{X}\omega)(Y, Z)&=&\omega\cc A(X, Y), Z\rr+\omega\cc Y, A(X, Z)\rr\,, \label{con1} \\
 A(X, J_{a}Y)&=&J_{a}\cc A(X, Y)\rr\,, \ \forall \ a=1, 2, 3\,. \label{con2}
 \end{eqnarray}
In particular, if $\nabla^{H,\omega}$ is an almost hypercomplex skew-Hermitian connection and $A$ is a tensor field on $M$ of type $(1, 2)$, then $\nabla^{H,\omega,A}_XY=\nabla^{H,\omega}_{X}Y+A(X, Y)$ is an almost hypercomplex skew-Hermitian connection, if and only if $A$ has values in $[\E\Hh]^*\otimes \so^*(2n)$.
\end{prop}
\begin{proof} The proof is easy and left to the reader.
\end{proof}

Let us now  find some particular $A$ depending only on the Obata connection $\nabla^{H}$ and the almost symplectic form $\omega$, to define the connection $\nabla^{H,\omega}$.

 \begin{theorem}\label{starconnection}
Let   $(M, H=\{J_{a} : a=1, 2, 3\}, \omega)$ be  an \textsf{almost hs-H manifold} endowed with the Obata connection  $\nabla^{H}$. Then, the connection $\nabla^{H,\omega}:=\nabla^{H}+A$, where the tensor field $A$   of type $(1, 2)$ is defined by
\[
\omega\cc A(X, Y), Z\rr=\frac{1}{2}(\nabla^{H}_{X}\omega)(Y, Z)
\]
for any $X, Y, Z\in\Gamma(TM)$,  is  an almost hypercomplex skew-Hermitian connection. In particular, the tensor $\omega\cc A(\cdot\,, \cdot), \cdot\rr$ of type $(0,3)$ takes values in $[\E\Hh]^*\otimes [S^2\E]^*$. Moreover, the torsion of $\nabla^{H,\omega}$ takes the form $T^H+\delta(A)$, where $T^H$ is the torsion of $\nabla^{H}$.
 \end{theorem}
 \begin{proof}
 When $A$ is defined as above, then the condition (\ref{con1}) is satisfied, hence $\nabla^{H,\omega}\omega=0$. We  will show  that also $\nabla^{H, \omega}J_{a}=0$ for any $a=1, 2, 3$.	By Proposition \ref{charcan}  this is equivalent to  the relation (\ref{con2}). Notice that    after  applying 
   $\omega$ on (\ref{con2}) it yields the relation   
\[
 \omega\cc A(X, J_{a}Y), Z\rr=\omega\cc J_{a}\cc A(X, Y)\rr, Z\rr=-\omega\cc A(X, Y), J_{a}Z\rr\,, \quad \forall X, Y, Z\in\Gamma(TM)\,,
\]
 where the second equality occurs  via the identity
\[
\omega(J_{a}X, Y)+ \omega(X, J_{a}Y)=0, \quad \forall  X, Y\in\Gamma(TM)\,,
\]
 see  Proposition \ref{usefrel1}. In particular,    it turns out that the relation 
 \[
 \omega\cc A(X, J_{a}Y), Z\rr+\omega\cc A(X, Y), J_{a}Z\rr=0\,, \quad \forall X, Y, Z\in\Gamma(TM)\,,
 \]
 is equivalent to say 
 that the tensor $\omega\cc A(\cdot\,, \cdot), \cdot\rr$ of type $(0,3)$ is $H$-Hermitian with respect to the last two indices.    
 However, the Obata connection  $\nabla^{H}$ is  a $\Gl(n,\Hn)$-connection, and   the space $[S^2\E]^*$ of scalar 2-forms is a $\Gl(n,\Hn)$-submodule of the 2-tensors that are $H$-Hermitian.  Hence,   $\omega\cc A(\cdot\,, \cdot), \cdot\rr$ takes values in $[\E\Hh]^*\otimes [S^2\E]^*$, and we conclude that the claim holds.
 \end{proof}

 \begin{prop}\label{parallel1}
An   almost hypercomplex skew-Hermitian connection $\nabla$ on an \textsf{almost hs-H manifold} $(M, H=\{I, J, K\}, \omega)$ satisfies,
\[
\nabla g_{I}=\nabla g_{J}=\nabla g_{K}=\nabla h=\nabla \Phi=0\,.
\]
Hence, it is a metric connection with respect to any of the three pseudo-Riemannian metrics $g_{I}, g_{J}, g_{K}$ and preserves the quaternionic skew-Hermitian form $h$ and the fundamental 4-tensor $\Phi$.
\end{prop}
\begin{proof}
	Consider for example $g_{I}$. It coincides with a contraction of the  composition of two $\nabla$-parallel tensor fields, namely $\omega$ and $I$.  Hence, $g_I$ must be parallel, which also occurs by a straightforward computation. The other claims are treated similarly.
\end{proof}

\smallskip
We now proceed with adapted connections to $\SO^*(2n)\Sp(1)$-structures. 
 Given an \textsf{almost qs-H manifold} $(M^{4n}, Q, \omega)$    with $n>1$,  there exists a class of Oproiu connections  $\nabla^{Q}$ preserving $Q$, that is 
$
 \nabla^{Q}_X\sigma\in \Gamma(Q)$ for all $X\in\Gamma(TM)$ and all smooth sections $\sigma \in \Gamma(Q)$.
Moreover, there is a unique Oproiu connection  $\nabla^{Q, \vol}$ preserving $Q$ and the volume form $\vol=\omega^{2n}$ induced by $\omega$, that is 
\[
\nabla^{Q,\vol}\vol=0\,,\quad  \nabla^{Q, \vol}_X\sigma\in \Gamma(Q), \quad \forall \ X\in\Gamma(TM), \quad \forall \ \sigma\in\Gamma(Q)\,.
\]
  This connection is the so-called \textsf{unimodular Oproiu connection} for the pair $(Q, \vol)$. 
  Recall that an Oproiu connection for an almost quaternionic structure $Q$ is a  minimal   connection for $Q$, see the seminal works of Oproiu \cite{Oproiu1, Oproiu2} and see also  \cite{AM} for more details on  $ \nabla^{Q, \vol}$. 

\begin{prop}\label{charcan}
Let $\nabla_{X}Y=\nabla^{Q}_{X}Y+A_{X}Y=\nabla^{Q}_{X}Y+A(X, Y)$ be a linear connection on  an \textsf{almost qs-H manifold} $(M, Q, \omega)$, where $\nabla^{Q}$ is any Oproiu connection, and $A$ is a smooth tensor field on $M$ of type $(1, 2)$.  Then, $\nabla$ satisfies the conditions
	\[
	\nabla \omega=0, \quad\text{and}\quad \nabla_X\sigma\in \Gamma(Q),
	\]
for any vector field $X\in\Gamma(TM)$ and section $\sigma\in\Gamma(Q)$,  if and only if  any $X, Y, Z\in\Gamma(TM)$ satisfy the following two relations
 \begin{eqnarray}
(\nabla^{Q}_{X}\omega)(Y, Z)&=&\omega\cc A(X, Y), Z\rr+\omega\cc Y, A(X, Z)\rr\,, \label{con3} \\
 A(X, J_{a}Y)-J_{a}\cc A(X, Y)\rr&=&\varphi^A_{c}(X)J_{b}(Y)-\varphi^A_{b}(X)J_{c}(Y)\,, \label{con4}
 \end{eqnarray}
for any cyclic permutation $(a,b,c)$ of $(1,2,3)$, any admissible basis $H=\{J_a\}$  and some (local) 1-forms $\varphi^A_{a}$ for any $a=1, 2, 3$. \\
 In particular, if $\nabla^{Q,\omega}$ is  an almost quaternionic skew-Hermitian connection and $A$ is a smooth tensor field on $M$ of type $(1, 2)$, then $\nabla^{Q,\omega,A}_XY=\nabla^{Q,\omega}_{X}Y+A(X, Y)$ is an almost quaternionic skew-Hermitian connection, if and only if $A$ takes values in $[\E\Hh]^*\otimes(\so^*(2n)\oplus\sp(1))$.
\end{prop}
\begin{proof}
	The claim follows since  only the part of $A$ belonging to $[\E\Hh]^*\otimes \sp(1)$ acts non-trivially  on the local admissible basis $\{J_a : a=1, 2, 3\}$, and at the same time  preserves $\{J_a : a=1, 2, 3\}$  in $\Gamma(Q)$. This provides the stated formula \eqref{con4}, and we leave the further details to the reader.
\end{proof}

  Next  we will construct an almost quaternionic skew-Hermitian connection, as in Theorem \ref{starconnection},  by  using a  tensor field $A$ of type $(1, 2)$   defined  via the relation
\begin{equation}\label{aformula}
\omega\cc A(X, Y), Z\rr=\frac{1}{2}(\nabla^{Q}_{X}\omega)(Y, Z)\,,\quad X,Y,Z\in \Gamma(TM)\,.
\end{equation} 
To do so, we benefit from the fact that  the space $[S^2\E]^*$ of scalar 2-forms is a $\Gl(n,\Hn)\Sp(1)$-module and $\nabla^{Q}$ is an $\Gl(n,\Hn)\Sp(1)$-connection. Thus, again $\omega\cc A(\cdot\,, \cdot), \cdot\rr$ has values in $[\E\Hh]^*\otimes [S^2\E]^*$, and consequently the  conditions \eqref{con3} and \eqref{con4} must  be satisfied by similar arguments as in the proof of Theorem \ref{starconnection}. However, to complete our construction, we need to overcome the following  

\smallskip
  \noindent{\bf\textsf{Challenge:}} {\it Although $A$ is determined uniquely by \eqref{aformula}, it depends on the choice of Oproiu connection and thus the space spanned by $\delta A$ is {\it not} complementary to $\delta([\E\Hh]^*\otimes (\so^*(2n)\oplus \sp(1))$.}
 
\smallskip
For a better visualization of this problem,  we need to consider  the following four  components isomorphic to $[\E\Hh]^*$ (see also \cite[p.~215]{AM} but be aware of slightly different conventions):
\begin{enumerate}
\item[$\mf{(A)}$] The component spanned by $\zeta\otimes \id$ for $\zeta\in \Gamma(T^*M)$, with values in $[\E\Hh]^*\otimes [S^2\E]^*\subset [\E\Hh]^*\otimes  \gl(n,\Hn)$.
\item[$\mf{(B)}$] The component spanned by the projection 
\begin{eqnarray*}
\pi_A(\omega\otimes Z)(X,Y)&:=&Asym\Big(\pi_{1, 1}(\omega(X,.) \otimes Z)\Big)Y=\frac{1}{8}\Big(\omega(X,Y) Z-\omega(X,Z) Y\\
&&-\sum_a g_{J_a}(X,Y) J_aZ+\sum_a g_{J_a}(X,Z) J_aY\Big)
\end{eqnarray*}
of $\omega \otimes Z$ for $X,Y,Z\in \Gamma(TM)$,
  with values in  $[\E\Hh]^*\otimes [S^2\E]^*$,  where    $\pi_{1, 1}$ is the  usual projection  (see for example \cite[p.~214]{AM} or \cite[p.~395]{CS})
 \[
\pi_{1, 1} : \gl([\E\Hh])\to \gl(n,\Hn)\,,\quad \pi_{1, 1}(\omega(X,.) \otimes Z)Y:=\frac{1}{4}\Big(\omega(X,Y) Z-\sum_a g_{J_a}(X,Y) J_aZ\Big)\,.
\]
Here, $Asym$  denotes the antisymmetrization with respect to the symplectic transpose (see Definition \ref{asymsym}).
\item[$\mf{(C)}$] The component spanned by the projection 
\begin{eqnarray*}
\pi_S(\omega\otimes Z)(X,Y)&:=&Sym\Big(\pi_{1, 1}(\omega(X,.) \otimes Z)\Big)Y=\frac{1}{8}\Big(\omega(X,Y) Z+\omega(X,Z) Y\\
&&-\sum_a g_{J_a}(X,Y) J_aZ-\sum_a g_{J_a}(X,Z) J_aY\Big)
\end{eqnarray*}
 of $\omega \otimes Z$ for $X,Y,Z\in \Gamma(TM)$, with values in  $[\E\Hh]^*\otimes\fr{so}^*(2n)$,  where   $Sym$  denotes the symmetrization with respect to the symplectic transpose (see Definition \ref{asymsym}).
\item[$\mf{(D)}$] The component locally spanned by $\sum_{a=1}^3 \zeta\circ J_a\otimes J_a$ with values in  $[\E\Hh]^*\otimes \sp(1)$ for some $\zeta\in \Gamma(T^*M)$ and some local admissible basis $H=\{J_1,J_2,J_3\}$.
\end{enumerate}
We also need to consider the following  traces $\Tr_{i}:\Lambda^2[\E\Hh]^*\otimes [\E\Hh]\to [\E\Hh]^*$ for $i=1,\dots,4$:
\begin{enumerate}
\item[$\mf{(i)}$] $\Tr_1(A)(X):=\Tr(A(\cdot \,,X))$;
\item[$\mf{(ii)}$] $\Tr_2(A)(X):=\Tr(A(X\,,\cdot))$;
\item[$\mf{(iii)}$] $\Tr_3(A)(X):=\Tr (A^T_X)$, where $A^T_X$ is the symplectic transpose of   $A_X:=A(X,\cdot)$;
\item[$\mf{(iv)}$]   $\Tr_4(A)(X):=\Tr( \J  A(\J X\,, \cdot))$, for $\J\in \Ss(Q).$ 
\end{enumerate}
Clearly, the components $\mf{(A)}$ and $\mf{(B)}$ are parts of the tensor $A$, and the components $\mf{(C)}$ and $\mf{(D)}$ are in $[\E\Hh]^*\otimes (\so^*(2n)\oplus \sp(1))$.  Hence, to assert our claim, it suffices to show that $\delta$ is {\it not} injective on these four components. 
 \begin{lem}\label{kernel}
Let us set $X^T(Y):=\omega(X,Y)$ for any $X,Y\in \Gamma(TM)$ and consider the tensor field 
\[
A:=\zeta_1\otimes \id+ \pi_A(\omega \otimes Z_2)+\pi_S(\omega \otimes Z_3)+\sum_{a=1}^3 \zeta_4\circ J_a\otimes J_a
\]
 for $\zeta_1, \zeta_4\in \Gamma(T^*M)$ and $Z_2,Z_3\in \Gamma(TM)$, given by the above components $\mf{(A)}$, $\mf{(B)}$, $\mf{(C)}$ and $\mf{(D)}$. Then, the traces $\Tr_{i}(A)\in \Gamma(T^*M)$ for  $i=1,\dots,4$  do not depend on the choice of  $\J\in \Gamma(Z)$ and moreover the following holds:
\begin{enumerate}
\item[$\al)$] $\delta(A)=0$,
if and only if $\zeta_1=-\zeta_4=-\frac14Z_2^T=\frac14Z_3^T$.
\item[$\be)$] $\omega(\delta(A)\,,\cdot)\in \Gamma(\Lambda^3T^*M)$, if and only if 
\[
\Tr_1(A)+\Tr_3(A)=0\,,\quad\text{and}\quad \Tr_4(A)=0\,,
\]
 which is equivalent  to say that  $Z_3^T= -\frac13Z_2^T+\frac83\zeta_1$ and $\zeta_4 =\frac16Z_2^T-\frac13\zeta_1$\,.
\item[$\gamma)$] $A\in \Gamma( [\E\Hh]^*\otimes (\sl(n,\Hn)\oplus \sp(1)))$, if and only if $\Tr_2(A)$ vanishes,  which is equivalent  to say that  $Z_2^T=4n\zeta_1.$
\end{enumerate}
\end{lem}
\begin{proof}
We can directly compute the traces of $A$ and obtain the following matrix equality,  which is independent of the choice of $\J\in \Gamma(Z)$:
\[
\left(\begin{smallmatrix}
\Tr_1(A)\\
\Tr_2(A)\\
\Tr_3(A)\\
\Tr_4(A)
\end{smallmatrix}\right)=\left(\begin{smallmatrix}
1& -\frac{2n+1}{4}  &  \frac{2n-1}{4} & -3\\
4n& -1& 0 &0 \\
-1& \frac{2n+1}{4} &  \frac{2n-1}{4} & -3 \\
0& 0& 0& 4n\\
\end{smallmatrix}\right)\left(\begin{smallmatrix}
\zeta_1\\
Z_2^T\\
Z_3^T\\
\zeta_4
\end{smallmatrix}\right)\,.
\]
Note that the determinant of the matrix is $2n(n+1)(2n-1)^2$, hence the matrix is invertible.
Moreover, $\Tr_1(\delta(A))=-\Tr_2(\delta(A))$ and we obtain
\begin{equation}\label{transition}
\left(\begin{smallmatrix}
\Tr_1(\delta(A))\\
\Tr_3(\delta(A))\\
\Tr_4(\delta(A))
\end{smallmatrix}\right)=\left(\begin{smallmatrix}
1-4n & -\frac{2n-3}{4}  &  \frac{2n-1}{4} & -3\\
-2&  \frac{2n+1}{2}& \frac{2n-1}{2} &-6 \\
1& - \frac{2n+1}{4}&  \frac{2n-1}{4}& 4n+1\\
\end{smallmatrix}\right)\left(\begin{smallmatrix}
\zeta_1\\
Z_2^T\\
Z_3^T\\
\zeta_4
\end{smallmatrix}\right),
\end{equation}
which provides the claimed kernel of $\delta$. On the other hand,   on an element $\phi$ of $\Gamma(\Lambda^3T^*M)$ given by the complete antisymmetrization of $2\omega\otimes \zeta$ we see that
\[
\Tr_1(\phi)=-\frac43(2n-1)\zeta\,,\quad \Tr_3(\phi)=\frac43(2n-1)\zeta\,,\quad \Tr_4(\phi)=0\,.
\]
In this way we obtain   the claimed condition for $\omega(\delta(A)\,,\cdot)\in  \Gamma(\Lambda^3T^*M)$. Since $\pi_S(\omega \otimes Z_3)+\sum_{a=1}^3 \zeta_4\circ J_a\otimes J_a$ has values in $[\E\Hh]^*\otimes (\so^*(2n)\oplus \sp(1))$, we need to characterize when $\zeta_1\otimes \id+ \pi_A(\omega \otimes Z_2)$ has values in $[\E\Hh]^*\otimes (\sl(n,\Hn)\oplus \sp(1))$, which is encoded by the vanishing of $\Tr_2$. This is because $A\cdot \vol=\Tr_2(A)\vol.$ This completes the proof.
\end{proof}
 
Now, we are able to define the connection $\nabla^{Q,\omega}$ explicitly.

 \begin{theorem}\label{starconnectionsp1}
Let  $(M, \omega, Q)$ be an \textsf{almost qs-H manifold}   endowed with any Oproiu connection $\nabla^{Q}$, and let us denote its  torsion by $T^Q$.  Let $A$ be the $(1, 2)$-tensor field on $M$ defined by 
\[
\omega\cc A(X, Y), Z\rr=\frac{1}{2}(\nabla^{Q}_{X}\omega)(Y, Z)\,,\quad  \forall \  X,Y,Z\in \Gamma(TM)\,,
\]
and set $Z_3^T:=\frac{\Tr_2(A)}{n+1}$, $\zeta_4:=-\frac{\Tr_2(A)}{4(n+1)}$. 
 Then, the connection 
\[
\nabla^{Q,\omega}:=\nabla^{Q}+A+\pi_S(\omega\otimes Z_3)+\sum_{a=1}^3 \zeta_4\circ J_a\otimes J_a
\]
is an almost quaternionic skew-Hermitian connection with the following property: The only component of its torsion $T^{Q,\omega}=T^Q+\delta(A+\pi_S(\omega\otimes Z_3)+\sum_{a=1}^3 \zeta_4\circ J_a\otimes J_a)$ isomorphic to $[\E\Hh]^*$, is contained in $\ke(2\Tr_1+\Tr_3)\cap \ke(\Tr_1-\Tr_4)\subset\Gamma(\Tor(M))$. In particular, for the unimodular Oproiu connection $\nabla^{Q,\vol}$  and a tensor $A^{\vol}$ defined by $\omega\cc A^{\vol}(X, Y), Z\rr=\frac{1}{2}(\nabla^{Q,\vol}_{X}\omega)(Y, Z)$, we obtain
\[
\nabla^{Q,\omega}=\nabla^{Q,\vol}+A^{\vol}.
\]
    \end{theorem}
 \begin{proof}
 We know  that $\omega\cc A(\cdot\,, \cdot), \cdot\rr$ has values in $[\E\Hh]^*\otimes [S^2\E]^*$, while $ \pi_S(\omega\otimes Z_3)+\sum_{a=1}^3 \zeta_4\circ J_a\otimes J_a$ has values in $[\E\Hh]^*\otimes (\so^*(2n)\oplus \sp(1))\,.$   Thus,  since $A$ satisfies the condition \eqref{con3} by definition, it follows that $A+\pi_S(\omega\otimes Z_3)+\sum_{a=1}^3 \zeta_4\circ J_a\otimes J_a$ will satisfy  the same relation.  Moreover,  for the same reasons as in Theorem \ref{starconnection} we conclude that the relation \eqref{con4} is valid for the tensor field  $A+\pi_S(\omega\otimes Z_3)+\sum_{a=1}^3 \zeta_4\circ J_a\otimes J_a$.  Hence,  Proposition \ref{charcan} is satisfied, and consequently $\nabla^{Q,\omega}$ is an almost quaternionic skew-Hermitian connection. Next, we need to show that $\nabla^{Q,\omega}$ does not depend on the choice of $\nabla^{Q}$.  By Lemma \ref{zeroprog1} and Corollary \ref{uniqueminimal}, it suffices to show that $T^{Q,\omega}$ does not depend on the choice of $\nabla^{Q}$. However, it is well-known that all Oproiu connections have the same torsion $T^Q$ and hence $T^{Q,\omega}=T^Q+\delta\cc A+\pi_S(\omega\otimes Z_3)+\sum_{a=1}^3 \zeta_4\circ J_a\otimes J_a\rr$. 
Since the difference of two  Oproiu connections belongs to the kernel $\ke(\delta)$ described in Lemma \ref{kernel}, we  conclude that such an element in the kernel of $\delta$ will change  $A$ by $-\zeta_1\otimes \id+ \pi_A(\omega \otimes 4\zeta_1^T)$ for some 1-form $\zeta_1$.  We compute
\[
\Tr_2(-\zeta_1\otimes \id+ \pi_A(\omega \otimes 4\zeta_1^T))=-4(n+1)\zeta_1\,,
\]
and hence we conclude that $Z_3$ will change by $-4\zeta_1^T$ and $\zeta_4$ will change by $\zeta_1$.  All together, this change takes the form
$
-\zeta_1\otimes \id+ \pi_A(\omega \otimes 4\zeta_1^T)-4\pi_S(\omega\otimes \zeta_1^T)+\sum_{a=1}^3 \zeta_1\circ J_a\otimes J_a,
$
and hence by Lemma \ref{kernel} we deduce that it belongs to  $\ke(\delta)$. This shows that  $T^{Q,\omega}$ is independent of the Oproiu connection. 
Now,  by using the formula \eqref{transition} we  deduce that the component of the torsion $T^{Q,\omega}$ isomorphic to $[\E\Hh]^*$ belongs to $\ke(2\Tr_1+\Tr_3)\cap \ke(\Tr_1-\Tr_4)$. Finally, since the unimodular Oproiu connection is a $\Sl(n,\Hn)\Sp(1)$-connection, by Lemma \ref{kernel} the corresponding tensor field $A^{\vol}$ satisfies $\Tr_2(A^{\vol})=0$ and the last claim follows, because then we have $Z_3^T=\zeta_4=0$.
 \end{proof}
 \begin{prop}\label{parallel2}
An   almost quaternionic skew-Hermitian connection $\nabla$ on an \textsf{almost qs-H manifold} $(M, Q, \omega)$ satisfies,
\begin{eqnarray*}
\nabla g_{J_a}&=&\varphi_{c}\otimes g_{J_{b}}-\varphi_{b}\otimes g_{J_{c}}\,,\\
\nabla h&=&\nabla \Phi=0\,,
\end{eqnarray*}
for any cyclic permutation $(a,b,c)$ of $(1,2,3)$  and some local 1-forms $\varphi_{a}$ $(a=1, 2, 3)$ on $M$, depending on a local admissible frame $\{J_1, J_2, J_3\}$ of $Q$.
Hence, $\nabla$ is not (necessarily) a metric connection with respect to any of the three pseudo-Riemannian metrics $g_{J_a}$ for $a=1, 2, 3$, but it preserves the quaternionic skew-Hermitian form $h$ and the fundamental 4-tensor $\Phi$.
\end{prop}
\begin{proof}
Since $\omega$ is $\nabla$-parallel, the covariant derivatives $\nabla g_{J_a}$ are induced by the covariant derivatives of the local admissible frame, which  are given by the claimed action of elements of $[\E\Hh]^*\otimes \sp(1)$. Since the Lie algebra $\sp(1)$ acts trivially on $h$ and $\Phi$, we get the other claims directly.  
\end{proof}

Let us consider the adapted connections $\nabla^{H, \omega}$ and $\nabla^{Q, \omega}$ constructed in Theorems \ref{starconnection} and \ref{starconnectionsp1}, respectively.  Due to  Propositions \ref{parallel1} and \ref{parallel2}   it makes sense to consider  the  covariant derivatives $\nabla^{H}\Phi$ and $\nabla^{Q}\Phi$, respectively.  It follows that 
 these  covariant derivatives take values in  modules which are naturally isomorphic to submodules of the intrinsic torsion corresponding to $\SO^*(2n)$- and $\SO^*(2n)\Sp(1)$-structures on a $4n$-dimensional manifold $M$, respectively. These modules  are related to 1st-order integrability conditions which we analyze   in detail in  the second part of this work, together with $T^{Q, \omega}$ (see  Sections 1 and 2 in \cite{CGWPartII}).   
  The covariant derivative $\nabla^{\omega}\Phi$, where $\nabla^{\omega}$ is any almost symplectic connection on $M$, can be used similarly,  while obviously a similar idea can be carried out  via the  quaternionic skew-Hermitian form $h$ and the covariant derivatives $\nabla^{H}h$, $\nabla^{Q, \vol}h$,  and $\nabla^{\omega}h$, respectively. 


\subsection{Decomposition of the space of torsion tensors and intrinsic torsion} \label{intrisictorsion}  

 Next we  present the decomposition  of  the module $\Lambda^2[\E\Hh]^*\otimes [\E\Hh]$ into  submodules   with respect to $\SO^*(2n)\Sp(1)$- and $\SO^*(2n)$-action, respectively.
 
 \begin{prop}\label{genCar1}  Let $(M, Q, \omega)$ be an \textsf{almost qs-H manifold}. 
	Then  the following  $\SO^*(2n)\Sp(1)$-equivariant decompositions hold (and should be read according to the conventions given in Section \ref{ehformsec}):
\begin{eqnarray*}
	\Lambda^2[\E\Hh]^*\otimes [\E\Hh]&\cong& [(\Lambda^3 \E\oplus\K\oplus\E)\otimes S^3\Hh]^*
	\oplus [(\Lambda^3 \E\oplus2\K\oplus3\E\oplus S^3_0 \E)\otimes \Hh]^*\,,\\
	\delta([\E\Hh]^*\otimes \so^\ast(2n))&\cong& [(\Lambda^3 \E\oplus\K\oplus\E)\otimes \Hh]^*\,,\\
	\delta([\E\Hh]^*\otimes \fr{sp}(1))&\cong& [\E\otimes (S^3\Hh\oplus \Hh)]^*\,.
	\end{eqnarray*}
	All the components in the decompositions are irreducible as $\SO^*(2n)\Sp(1)$-representations, with the exception of $\K$ for $n=2$ and of $\Lambda^3 \E$ for $n=3$. Moreover, all  these  $\SO^*(2n)\Sp(1)$-modules are non-equivalent, apart from the stated multiplicities, and the isomorphism $\Lambda^3 \E\cong \E$ for $n=2$.
	\end{prop}
\begin{proof}
	In terms of the $\E\Hh$-formalism, $\E\Hh$ is the complex tensor product $\E\otimes_{\C}R(\theta)$, and it is irreducible as a complex $\so^\ast(2n)\oplus \fr{sp}(1)$-representation. 	We now compute the decomposition of $\Lambda^2 [\E\Hh]^*$ via complexification (note that all $\so^\ast(2n)\oplus \fr{sp}(1)$-modules that appear  below  have real type, so this is identical to the real decomposition). We obtain
	\begin{eqnarray*}
		\Lambda^2 \E\Hh &=& \big(\Lambda^2\E \otimes S^2R(\theta)\big) \oplus \big(S^2\E \otimes \Lambda^2R(\theta)\big) \\
		&=&\big(\Lambda^2\E \otimes S^2\Hh\big)\oplus \big(S^2_0\E \oplus R(0)\big)\,.
	\end{eqnarray*}
As a consequence, we deduce that
	\begin{eqnarray*}
		\Lambda^2[\E\Hh]^*\otimes [\E\Hh] &=&[ \Lambda^2\E \otimes S^2\Hh \oplus S^2_0\E \oplus R(0) ]^* \otimes [\E\Hh] \\
		&=&[\big(\Lambda^2\E\otimes \E\big) \otimes \big(S^2R(\theta)\otimes R(\theta) \big)]^* \\
		&&\oplus [\big(S^2_0\E\otimes \E)\big) \otimes  R(\theta)]^* \oplus [\E\Hh]^*\\
		&=& [(\Lambda^3 \E\oplus\K\oplus\E)\otimes S^3\Hh]^*
	\oplus [(\Lambda^3 \E\oplus\K\oplus\E)\otimes \Hh]^*\\
&& \oplus [(S^3_0 \E\oplus \K\oplus\E)\otimes \Hh]^* \oplus [\E\Hh]^*\, .
	\end{eqnarray*}
	This gives rise to the first result. We also compute
	\begin{align*}
		[\E\Hh]^*\otimes \so^\ast(2n)&= [\E\otimes R(\theta)\otimes \Lambda^2\E]^*=[(\Lambda^3 \E\oplus\K\oplus\E)\otimes \Hh]^*,\\ 
      [\E\Hh]^*\otimes \fr{sp}(1)&= [\E\otimes R(\theta)\otimes  R(2\theta)]^*= [\E\otimes (R(3\theta)\oplus R(\theta))]^*=[\E\otimes (S^3\Hh\oplus \Hh)]^*.
	\end{align*}
Since by Lemma \ref{zeroprog1} for $\fr{g}=\so^\ast(2n)\oplus \fr{sp}(1)$ or $\fr{g}=\so^\ast(2n)$, the Spencer differential $\delta$ is injective, our final statement easily   follows by the above decomposition. 
\end{proof}

 As a  corollary (in combination with the results in Table \ref{Table1}) we obtain the  number of algebraic types depending on the irreducible intrinsic torsion modules of $\SO^*(2n)\Sp(1)$-structures on a $4n$-dimensional  smooth manifold $M$ $(n>1)$.
\begin{corol}\label{algebtypes}
For $n>3$,  the  intrinsic torsion module corresponding to $\fr{so}^*(2n)\oplus \fr{sp}(1)$ admits the following $\SO^*(2n)\Sp(1)$-equivariant decomposition into irreducible submodules:
\[
\begin{tabular}{l c c c c c c c c c c}
$\mc{H}(\so^\ast(2n) \oplus \fr{sp}(1))$  &  $\cong$ &   $\mc{X}_1$ &  $\oplus$ &    $\mc{X}_2$  &  $\oplus$  &   $\mc{X}_3$  & $\oplus$ &     $\mc{X}_4$   & $\oplus$ &   $\mc{X}_5$\\
& $=$ &  $[\K S^3\Hh]^*$ &  $\oplus$ &    $[\Lambda^3\E S^3\Hh]^*$ & $\oplus$ &   $[\K\Hh]^*$ & $\oplus$ &   $[\E\Hh]^*$  & $\oplus$ & $ [S^3_0 \E \Hh]^*$,
\end{tabular}
\]
where $\mc{X}_3\oplus\mc{X}_4=[\K\Hh]^*\oplus[\E\Hh]^*\subset\Lambda^{3}[\E\Hh]^*$. 
 Consequently, for $n>3$ there exist  five main types of $\SO^*(2n)\Sp(1)$-structures, $\mc{X}_1, \ldots, \mc{X}_5$, defined as above, and up to $2^5=32$ algebraic types of $\SO^*(2n)\Sp(1)$-geometries. 
\begin{itemize}
\item For $n=3$, $\mc{X}_2$ decomposes into two irreducible $\SO^*(6)\Sp(1)$-modules, namely 
\begin{equation}\label{mcX2}
\mc{X}_2=\mc{X}_2^{+} \oplus   \mc{X}_2^-\cong [R(2\pi_2)S^3\Hh]^*\oplus [R(2\pi_3)S^3\Hh]^*\,.
\end{equation}
\item For $n=2$,  both $\mc{X}_1$ and $\mc{X}_3$ decompose into two irreducible $\SO^*(4)\Sp(1)$-modules, namely
\renewcommand\arraystretch{1.5}
\begin{equation}\label{mcX13}
\left\{
\begin{tabular}{ll}
$\mc{X}_1=$ & $\mc{X}_1^{+}\oplus \mc{X}_1^-\cong [R(\pi_1+3\pi_2)S^3\Hh]^*\oplus[R(3\pi_1+\pi_2)S^3\Hh]^*\,,$\\
$\mc{X}_3=$ &$\mc{X}_3^{+}\oplus \mc{X}_3^-\cong [R(\pi_1+3\pi_2)\Hh]^*\oplus[R(3\pi_1+\pi_2)\Hh]^*\,.$
\end{tabular}\right.
\end{equation}
\end{itemize}
Consequently,  there exist up to $2^6$ algebraic types of $\SO^*(6)\Sp(1)$-geometries,  and up to $2^7$ algebraic types of $\SO^*(4)\Sp(1)$-geometries.
\end{corol}
As it is customary to the theory of $G$-structures,  we split the geometries into classes with respect  to the algebraic types. 
 \begin{defi}\label{quatypes}
	 For $n>3$, a $\SO^*(2n)\Sp(1)$-structure  is said to be of \textsf{pure type $\mc{X}_i$} for $i=1,\ldots, 5$, if the intrinsic torsion takes values in the corresponding irreducible submodule. Similarly, we say that a $\SO^*(2n)\Sp(1)$-structure is of   \textsf{mixed type $\mc{X}_{i_1\ldots i_j}$} for $1\leq i_1< \ldots< i_j\leq 5$, if the intrinsic torsion takes values in the module $\mc{X}_{i_1}\oplus\cdots\oplus\mc{X}_{i_j}$. For instance, mixed type $\mc{X}_{135}$ means that the intrinsic torsion takes values in $\mc{X}_1\oplus\mc{X}_3\oplus\mc{X}_5$.  Similar notations are adapted for $n=2, 3$.
 \end{defi}

  Let  us discuss now the case of $\SO^*(2n)$.   Assume that $W$ is  some   $\SO^*(2n)$-module. Then we obtain the following branching from  $\SO^*(2n)\Sp(1)$-modules  to $\SO^*(2n)$-modules:
\[
[W\otimes \Hh]^*\cong W^*\,,\quad [W\otimes S^3\Hh]^*\cong 2W^*.
\]
Therefore,  the following occurs as a simple corollary of Proposition \ref{genCar1}, where for the low-dimensional cases we again rely on  Table \ref{Table1}.
\begin{corol}\label{genCar2}
\textsf{1)} As  $\SO^\ast(2n)$-modules, $\Lambda^2[\E\Hh]^*\otimes [\E\Hh]$ and the image $\delta([\E\Hh]^*\otimes \so^\ast(2n))$ admit the following $\SO^*(2n)$-equivariant (irreducible for $n>3$) decompositions:
\begin{align*}
	\Lambda^2[\E\Hh]^*\otimes [\E\Hh]&\cong 3\Lambda^3 \E^*\oplus4\K^*\oplus5\E^*\oplus S^3_0 \E^*\,,\\
	\delta([\E\Hh]^*\otimes \so^\ast(2n))&\cong \Lambda^3 \E^*\oplus\K^*\oplus\E^*\,.
	\end{align*}
All the components in the decompositions are irreducible as $\SO^*(2n)$-representations, with the exception of $\K$ for $n=2$ and of $\Lambda^3 \E$ for $n=3$. As a consequence, in terms of $\SO^*(2n)\Sp(1)$-modules we have the decomposition
{\small \[
\begin{tabular}{l c c c c c c c c c c c c c c}
$\mc{H}(\so^\ast(2n))$  &  $\cong$ &   $\mc{X}_1$ &  $\oplus$ &    $\mc{X}_2$  &  $\oplus$  &   $\mc{X}_3$  & $\oplus$ &     $\mc{X}_4$   & $\oplus$ &   $\mc{X}_5$ & $\oplus$ &    $\mc{X}_6$   & $\oplus$ &   $\mc{X}_7$\\
& $=$ &  $[\K S^3\Hh]^*$ &  $\oplus$ &    $[\Lambda^3\E S^3\Hh]^*$ & $\oplus$ &   $[\K\Hh]^*$ & $\oplus$ &   $[\E\Hh]^*$  & $\oplus$ & $ [S^3_0 \E \Hh]^*$ & $\oplus$ & $[\E S^3\Hh]^*$ & $\oplus$ &  $[\E\Hh]^*$ 
\end{tabular}
\]}
where  $\mc{X}_3\oplus\mc{X}_4=[\K\Hh]^*\oplus[\E\Hh]^*\subset\Lambda^{3}[\E\Hh]^*$,  $\mc{X}_6= [\E S^3\Hh]^*\subset \delta\big([\E\Hh]^*\otimes \sp(1)\big)$,  
\[
\mc{X}_7= [\E\Hh]^*= \Big\{\delta\big(\sum_{a=1}^3 \zeta\circ J_a\otimes J_a\big)\in \delta\big([\E\Hh]^*\otimes \sp(1)\big): \zeta\in [\E\Hh]^*\Big\}\,,
\]
and $H=\{J_a : a=1, 2, 3\}$ denotes the corresponding almost hypercomplex structure. 
On the other hand, in terms of 
  $\SO^*(2n)$-modules we get the decomposition 
\begin{equation}\label{spen2}
\mc{H}(\so^\ast(2n))\cong  2\Lambda^3\E^*\oplus3\K^*\oplus4\E^*\oplus S^3_0 \E^*\,.
\end{equation}
 \textsf{2)}  Consequently, for $n>3$ there exist up to  $2^{7}$ $\Sp(1)$-invariant algebraic types of $\SO^*(2n)$-geometries, and totally up to $2^{10}$ algebraic types of $\SO^*(2n)$-geometries. \\
Moreover, for $n=2, 3$ the following hold:
\begin{itemize}
\item For $n=2$  there exist  up to $2^9$  $\Sp(1)$-invariant  algebraic types of $\SO^*(4)$-geometries, since in this case the modules $\mc{X}_1, \mc{X}_3$ decomposes as in {\rm (\ref{mcX13})}.
\item For  $n=3$ there exists up to $2^8$ $\Sp(1)$-invariant  algebraic types of  $\SO^*(6)$-geometries, since in this case  the module  $\mc{X}_2$ decomposes as in {\rm (\ref{mcX2})}.
\end{itemize}
\end{corol}
\begin{rem}
\textnormal{
\textsf{1)} Due to the multiplicities appearing in the decomposition (\ref{spen2}) of $ \mc{H}(\so^\ast(2n))$,  a definition similar with the Definition \ref{quatypes} requires precise projections to the individual irreducible factors. In general, such projections are not unique and their construction is a very  complicated task, which requires a further study of  $\SO^*(2n)$-structures. However, we can successfully use the  $\Sp(1)$-invariant algebraic types to describe some distinguished classes of  $\SO^*(2n)$-structures, characterized by $\Sp(1)$-invariant conditions. This procedure is analyzed in Sections 1 and 2 of \cite{CGWPartII}.  Note also that a precise geometric characterization of the pure modules $\mc{X}_i$ $(i=1, \ldots, 7)$  requires deeper investigation of the corresponding Bianchi identities, which we do not present in this first part. \\ 
\textsf{2)}  The decompositions of $\Lambda^2[\E\Hh]^*\otimes [\E\Hh]$ under $\SO^*(2n)\Sp(1)$ and $\SO^*(2n)$, given respectively in Proposition \ref{genCar1} and Corollary \ref{genCar2}, can be viewed as the counterpart to Cartan's decomposition of the space of torsion tensors corresponding to \textsf{metric connections}, see e.g.  \cite{CGW} and the references therein. 
 For $G=\Sp(2n, \R)$ and the torsion of almost symplectic connections, analogous decompositions have been recently  presented   in \cite{APick}.}
\end{rem}
 
\section{Minimality of adapted connections to $\SO^{*}(2n)$- and $\SO^*(2n)\Sp(1)$-structures}\label{minconnections}
 \subsection{Minimal connections}
  Corollaries \ref{algebtypes} and \ref{genCar2}  need to be studied in a greater detail, since neither the posed isomorphisms, nor  the way that the image of the Spencer differential $\delta$ sits inside $\Lambda^2[\E\Hh]^*\otimes [\E\Hh]$, are obvious.  Both these tasks occur due to the  involved multiplicities appearing  in the decompositions presented in Corollary \ref{algebtypes} and Corollary \ref{genCar2}, respectively. As a consequence
\begin{lem}\label{manyp}
For $\SO^*(2n)$- or $\SO^*(2n)\Sp(1)$-structures there are many possible 
invariant normalization conditions, and thus different classes of minimal adapted connections.
\end{lem}
\begin{example}
For example, our choice of the modules $\mc{X}_1, \ldots, \mc{X}_7$ provides a particular normalization condition for $\SO^*(2n)$-structures. However, $\nabla^{H, \omega}$ is not a minimal connection with respect to this normalization condition. 
\end{example}
Our goal below is to provide the normalization conditions which establish both of our connections  $\nabla^{H,\omega}$ and $\nabla^{Q,\omega}$ defined in Section \ref{adapatcon}, as minimal. In order to do this, we will combine  results of Sections \ref{adapatcon} and \ref{intrisictorsion}   with certain results from \cite{Bon, AM, Gau}, which we   recall.

Let $(H=\{J_a : a=1, 2, 3\}, \omega)$ be  an \textsf{almost hs-H structure}     on a smooth manifold $M$, or let $(Q, \omega)$ be an \textsf{almost qs-H structure} on $M$ for which $H$ provides a local admissible frame.   Next  it is again convenient to work  at an algebraic level, in terms of the $\E\Hh$-formalism. Since $\omega$ is a scalar 2-form with respect to  $H$, the space $\Lambda^{2}[\E\Hh]^*\otimes [\E\Hh]$ 
 admits several distinguished equivariant projections, which induce projections on the  space of sections of the induced vector bundles. 
 First, we may associate to  any $J_{a}\in H$    the subspace   
\[
\mathscr{C}_{J_{a}}:=\big\{\phi\in \Lambda^{2}[\E\Hh]^*\otimes [\E\Hh] : \phi(J_{a}X,  Y)=\phi(X,  J_{a}Y)=-J_{a}\phi(X,  Y)\big\}\,,
\]  
which is isomorphic to $\delta([\E\Hh]^*\otimes\fr{gl}(2n, \C))$. We have a $\Gl(2n, \C)$-equivariant 
  projection   $\pi_{J_{a}} : \Lambda^{2}[\E\Hh]^*\otimes [\E\Hh]  \longrightarrow \mathscr{C}_{J_{a}}$ 
defined by
\[
\pi_{J_{a}}(\phi)(X,  Y):=\frac{1}{4}\CC\phi(X,  Y)+J_{a}\big(\phi(J_{a}X, Y)+\phi(X,  J_{a}Y)\big)-\phi(J_{a}X, J_{a}Y)\RR\,,\quad a=1, 2, 3,
\]
 for any $X,  Y, Z\in [\E\Hh]$. By setting 
\[
	\pi_{H}:=\frac{2}{3}(\pi_{J_{1}}+\pi_{J_{2}}+\pi_{J_{3}})=\frac{2}{3}\sum_{a=1}^{3}\pi_{J_{a}},
\]
we  obtain a $\Gl(n, \Hn)$-equivariant projection, i.e., $\pi_{H}^{2}=\pi_{H}$, see \cite[p.~420]{Bon}. In full terms 
\begin{eqnarray*}
\pi_H(\phi)(X,Y)&=&\frac{1}{6}\Big(3\phi(X,Y)-\phi(IX,IY)-\phi(JX,JY)-\phi(KX,KY) +I\phi(X,IY)\\
 &&+I\phi(IX,Y)+J\phi(X,JY)+J\phi(JX,Y)+K\phi(X,KY)+K\phi(KX,Y)\Big)\,,
\end{eqnarray*}
for any $\phi\in  \Lambda^{2}[\E\Hh]^*\otimes [\E\Hh]$, where we have assumed that $H=\{I, J, K\}$. In fact, $\pi_H$ is $\Gl(n, \Hn)\Sp(1)$-equivariant.   Then, we get the following $\Gl(n, \Hn)$-equivariant isomorphisms
\[
\ke(\pi_{H})\cong\cap_{J_{a}\in\Ss^2}\ke(\pi_{J_{a}}) \cong \delta([\E\Hh]^*\otimes\fr{gl}(n, \Hn))\,,
\]
which are also $\Gl(n, \Hn)\Sp(1)$-equivariant. It follows that  for a  $\Gl(n, \Hn)$-structure $H=\{I, J, K\}$ the    intrinsic torsion module $\mc{H}(\fr{gl}(n, \Hn))$ coincides with $ \sf{Im}(\pi_{H})$, and moreover the intrinsic torsion itself   is expressed by an appropriate linear combination of the  Nijenhuis tensors corresponding to two anticommuting elements $I, J\in H$.   In particular, the torsion $T^{H}$ of the  Obata connection $\nabla^{H}$ associated to $H$ satisfies $T^H=\pi_H(T^H)$. The torsion $T^{Q}$ of $\nabla^{Q}$ is also in the image of $\pi_H$ and in addition it satisfies the condition (see \cite{AM})
\[
\Tr_4(T^{Q}_X)=\Tr( \J \circ T^{Q}_{\J X})=0\,,
\]
 for all $\J\in S(Q)\cong\Ss^2$.  
 We are  now ready  to pose our main theorem related to minimal connections.

\begin{theorem}\label{zeroprog}
\textsf{1)} The connection $\nabla^{H,\omega}$ defined in Theorem \ref{starconnection}  is the unique almost hypercomplex skew-Hermitian connection with torsion in the module
\begin{eqnarray}
\imm(\pi_H)\oplus \delta([\E\Hh]^*\otimes [S^2\E]^*)&\cong&(\mc{X}_1\oplus \mc{X}_2\oplus\mc{X}_6)\oplus (\mc{X}_3\oplus\mc{X}_4\oplus\mc{X}_5\oplus\mc{X}_7)\label{decso21n}\\
&\cong& (2\Lambda^3 \E^*\oplus 2\K^*\oplus2\E^*) \oplus (\K^*\oplus2\E^*\oplus S^3_0 \E^*)\label{decso2n}\,,
\end{eqnarray}
which is complementary to $\delta\big([\E\Hh]^*\otimes \so^*(2n)\big)$, that is  $\nabla^{H,\omega}$ is  the unique  minimal ($\mc{D}$-connection) for the normalization condition
\[
\mc{D}(\so^*(2n)):=\imm(\pi_H)\oplus \delta([\E\Hh]^*\otimes [S^2\E]^*)\,.
\]
Note that the first decomposition {\rm (\ref{decso21n})}  is given in terms of $\SO^*(2n)\Sp(1)$-modules, while   {\rm (\ref{decso2n})} should be read in terms of $\SO^*(2n)$-modules. \\ 
\textsf{2)}  Let us define $\mc{D}(\so^*(2n)\oplus \sp(1))$ as the submodule of $\imm(\pi_H)\oplus \delta\big([\E\Hh]^*\otimes [S^2\E]^*\big)$  which is in the kernel of $(2\Tr_1+\Tr_3)$ and $(\Tr_1-\Tr_4)$, for all $\J\in \Ss(Q)$. Then, the connection $\nabla^{Q,\omega}$ defined in Theorem   \ref{starconnectionsp1}   is the unique almost quaternionic skew-Hermitian connection with torsion in the module
\begin{gather*}
\mc{D}(\so^*(2n)\oplus \sp(1))\cong    [\K S^3\Hh]^* \oplus [\Lambda^3\E S^3\Hh]^* \oplus [\K\Hh]^* \oplus [\E\Hh]^* \oplus  [S^3_0 \E \Hh]^* \,,
\end{gather*}
which is complementary to $\delta\big([\E\Hh]^*\otimes (\so^*(2n)\oplus \sp(1))\big)$, that is $\nabla^{Q,\omega}$ is  the unique minimal ($\mc{D}$-connection) for the normalization condition $\mc{D}(\so^*(2n)\oplus \sp(1)).$
\end{theorem}
\begin{proof}
A part of the uniqueness claim is related to the vanishing of $\fr{g}^{(1)}$ for 
\[
\fr{g}\in\{\fr{so}^{\ast}(2n), \fr{so}^{\ast}(2n)\oplus\sp(1), \fr{gl}(n,\Hn), \fr{sl}(n,\Hn)\oplus\sp(1)\}\,,
\]
as we have mentioned before.  Also, by the previous discussion,  we know that $\imm(\pi_H)$ is complementary to the image  $\delta \big([\E\Hh]^*\otimes \fr{gl}(n,\Hn)\big)$.
Now, $[\E\Hh]^*\otimes [S^2\E]^*$ is clearly a complementary subspace to $[\E\Hh]^*\otimes \fr{so}^{\ast}(2n)$ in $[\E\Hh]^*\otimes \fr{gl}(n,\Hn)$. Since $\fr{g}^{(1)}$ vanishes,   $\delta$ is injective and so the first claim follows in combination with Corollary \ref{genCar2}.
For the second assertion,  by Lemma \ref{kernel} we know that
\[
\delta([\E\Hh]^*\otimes [S^2\E]^*)
\]
 is {\it not} a complementary subspace of $\delta\big([\E\Hh]^*\otimes (\so^*(2n)\oplus \sp(1))\big)$. On the other hand,  Theorem \ref{starconnectionsp1} implies that the connection $\nabla^{Q,\omega}$ is the unique almost quaternionic skew-Hermitian connection with torsion in the module $\mc{D}(\so^*(2n)\oplus \sp(1))$. Since its torsion component isomorphic to $[\E\Hh]^*$ is by Lemma \ref{kernel} complementary to $\delta\big([\E\Hh]^*\otimes (\so^*(2n)\oplus \sp(1))\big)$, we get the second claim  in combination with Corollary \ref{algebtypes}.
\end{proof}

  Let us now recall that the torsion of almost symplectic connections is normalized to be an element of $\Lambda^3[\E\Hh]^*$, which is complementary to $\delta([\E\Hh]^*\otimes \fr{sp}(4n,\R))$ in the space of torsion tensors $\Lambda^2[\E\Hh]^*\otimes [\E\Hh]$ (after raising the last index using $\omega$). 
  In particular,   $\Lambda^{3}[\E\Hh]^*$    sits  inside $\Lambda^{2}[\E\Hh]^*\otimes [\E\Hh]$  and one can  introduce the  following $\Sp(4n, \R)$-equivariant projection 
\[
\pi_{\omega} :  \Lambda^{2}[\E\Hh]^*\otimes [\E\Hh] \longrightarrow \Lambda^{3}[\E\Hh]^*\,,\quad  \pi_{\omega}(\phi)(X, Y, Z):=\frac{1}{3}\fr{S}_{X, Y, Z}\omega\big(\phi(X,  Y), Z\big)\,,\\
\]
where $\fr{S}_{X, Y, Z}$ denotes the cyclic sum over $X, Y, Z\in [\E\Hh]$.  Next we shall prove that the decomposition $\mc{X}_1\oplus\cdots\oplus\mc{X}_7$ described in Corollary  \ref{genCar2}  for $\SO^*(2n)$-structures,  is compatible with the decomposition  $\Lambda^3[\E\Hh]^*\oplus\delta([\E\Hh]^*\otimes \fr{sp}(4n,\R))$, and we shall analyze some further properties of these special $\Sp(1)$-invariant intrinsic torsion modules.  

To do so,  initially we need to do some preparatory work and prove a preliminary lemma.   Next we shall denote the ``\textsf{lowering operator}'' by 
\[
\ll :  \Lambda^{2}[\E\Hh]^*\otimes [\E\Hh] \longrightarrow  \Lambda^{2}[\E\Hh]^*\otimes [\E\Hh]^*\subset \otimes^{3}[\E\Hh]^*\,,\quad \ll(\phi)(X, Y, Z):=\omega\big(\phi(X, Y), Z\big)\,,
\]
for any $\phi\in  \Lambda^{2}[\E\Hh]^*\otimes [\E\Hh]$.   The   left inverse of $\ll$ is the    
   ``\textsf{raising operator}''
\[
\ll^{-1} : \otimes^{3}[\E\Hh]^*\longrightarrow  \Lambda^{2}[\E\Hh]^*\otimes [\E\Hh]\,,  
\]
 defined as follows:  
\[
\otimes^{3}[\E\Hh]^*\ni \Uptheta \longmapsto  \ll^{-1}(\Uptheta)\in \Lambda^{2}[\E\Hh]^*\otimes [\E\Hh]\,,\quad \ll^{-1}(\Uptheta)(X, Y):=\phi(X, Y)\,,
\]
where $X, Y\in[\E\Hh]$ and  $\phi\in  \Lambda^{2}[\E\Hh]^*\otimes [\E\Hh]$ is a vector-valued 2-form with $\Uptheta =\ll(\phi)$.  Passing to the level of bundles, we can prove that
\begin{lem} \label{altoper} 
Let $(M, \omega)$ be an \textsf{almost symplectic manifold}. Then,
 for any $\phi\in\Gamma(\Tor(M))$   we have  $\ll^{-1}\big(\pi_{\omega}(\phi)\big)=\sf{Alt}_{\phi}$, where  $\sf{Alt}_{\phi}\in \Gamma(\Tor(M))$ is  the operator  given by 
 \begin{equation}\label{Alt}
{\sf{Alt}}_{\phi}(X, Y)=\frac{1}{3}(\phi_{X}Y-\phi_{X}^{T}Y+\phi^{T}_{Y}X)\,.
\end{equation}
Here, for the vector-valued  2-form  $\phi$, and for any $X\in\Gamma(TM)$, we denote by $\phi_{X}\in \Ed(TM)$ the induced  endomorphism  with $\phi_{X}Y=\phi(X, Y)$, and by $\phi_{X}^{T}$ its symplectic transpose with respect to $\omega$.
\end{lem}
\begin{proof}
 It is sufficient to prove that 
\begin{equation}\label{equT1}
\pi_{\omega}(\phi)(X, Y, Z)=\omega\cc {\sf Alt}_{\phi}(X, Y), Z\rr\,,
\end{equation}
 where ${\sf Alt}_{\phi}$ is given by $(\ref{Alt})$.   Recall that the symplectic transpose $\phi_{X}^{T}$ is defined via the relation $\omega(\phi_{X}^{T}Y, Z)=-\omega(Y, \phi_{X}Z)$, for any $X, Y, Z\in\Gamma(TM)$.   
Since any    $\phi\in \Lambda^{2}T^*M\otimes TM$   satisfies $\phi_{X}Y=-\phi_{Y}X$, by the definition of $\pi_\omega$ we see that
\begin{eqnarray*}
3\pi_{\omega}(\phi)(X, Y, Z)&=&\omega\cc\phi(X, Y), Z\rr+\omega\cc\phi(Y, Z), X\rr+\omega\cc\phi(Z, X), Y\rr\\
&=&\omega(\phi_{X}Y, Z)-\omega(X, \phi_{Y}Z)+\omega(Y, \phi_{X}Z)\\
&=&\omega(\phi_{X}Y, Z)+\omega(\phi_{Y}^{T}X, Z)-\omega(\phi_{X}^{T}Y, Z)\\
&=&\omega(\phi_{X}Y-\phi_{X}^TY+\phi_{Y}^{T}X, Z)\,,
\end{eqnarray*}
for any $X, Y, Z\in\Gamma(TM)$, which proves (\ref{equT1}). Then, by the definition of $\ll^{-1}$ it follows that
\[
\ll^{-1}\cc\pi_{\omega}(\phi)\rr(X, Y)={\sf Alt}_{\phi}(X, Y)=\frac{1}{3}(\phi_{X}Y-\phi_{X}^TY+\phi_{Y}^{T}X)\,.
\]
\end{proof}
Now we are able to proceed with a proof of the  claims pronounced above.

\begin{prop}\label{symtorscomp}
\textsf{1)}  The decomposition $\mc{X}_1\oplus\cdots\oplus\mc{X}_7$ is compatible with the decomposition  $\Lambda^3[\E\Hh]^*\oplus\delta([\E\Hh]^*\otimes \fr{sp}(4n,\R))$. In particular, $\mc{X}_{234}=\Lambda^{3}[\E\Hh]^*$ and $\mc{X}_{1567}\subset \delta([\E\Hh]^*\otimes \fr{sp}(4n,\R))$.\\
\textsf{2)} The torsion components $\mc{X}_1, \mc{X}_2, \mc{X}_5$ and $\mc{X}_6$ are independent of the normalization condition and coincide with the following torsion components of the connection $\nabla^{H, \omega}$ introduced in Theorem \ref{starconnection}: $[\K S^3\Hh]^*$,    $[\Lambda^3 \E  S^3\Hh]^*$,  $[ S^3_0 \E \Hh]^*$ and  $[\E S^3\Hh]^*$, respectively.  Moreover, there are the following $\SO^*(2n)\Sp(1)$-equivariant   maps
\[
\pi_3 : [\K\Hh]^*\to \mc{X}_3\,,\quad \pi_4 : 2[\E\Hh]^*\to \mc{X}_4\,,\quad \pi_7 : 2[\E\Hh]^*\to \mc{X}_7\,,
\]
explicitly defined as 
\begin{eqnarray*}
\pi_3(\phi)&:=&{\sf Alt}_{\phi}\in \mc{X}_3\subset \Lambda^3[\E\Hh]^*\,,\\
\pi_4\Big(\delta\big(\zeta_1\otimes \id+\pi_A(\omega \otimes Z_2)\big)\Big) &:=& {\sf Alt}_{(2\zeta_1+\frac12Z_2^T)\otimes \omega}\in \mc{X}_4\subset \Lambda^3[\E\Hh]^*\,,\\
\pi_7\Big(\delta\big(\zeta_1\otimes \id+\pi_A(\omega \otimes Z_2)\big)\Big) &:=& \delta\big(\sum_{a=1}^3 (\frac13\zeta_1-\frac16Z_2^T)\circ J_a\otimes J_a\big)\in \mc{X}_7\subset \delta([\E\Hh]^*\otimes \fr{sp}(4n,\R))\,,
\end{eqnarray*}
where $\zeta_1, Z_2^T\in [\E\Hh]^*$, $[\K\Hh]^*$ and $2[\E\Hh]^*$ are torsion components of $\nabla^{H, \omega}$, and $\sf{Alt}$ is the operator   introduced in Lemma \ref{altoper}.   In particular,  $\pi_3$ is an isomorphism,   $\pi_4, \pi_7$ are surjections, and the sum $\pi_4\oplus\pi_7$ is also an isomorphism.\\ 
\textsf{3)} The torsion components $\mc{X}_1, \mc{X}_2$ and  $\mc{X}_5$  are independent of the normalization condition and coincide with the following torsion components of the connection $\nabla^{Q, \omega}$ introduced in Theorem \ref{starconnectionsp1}: $[\K S^3\Hh]^*$,    $[\Lambda^3 \E  S^3\Hh]^*$, and  $[ S^3_0 \E \Hh]^*$, respectively.  Moreover, there are the following $\SO^*(2n)\Sp(1)$-equivariant isomorphisms 
\[
\pi_3 : [\K\Hh]^*\to \mc{X}_3\,,\quad \pi_4 :  [\E\Hh]^*\to \mc{X}_4\,,
\]
where $\pi_3$, $\pi_4$ are given as above, and  $[\K\Hh]^*$ and $[\E\Hh]^*$ are torsion components of $\nabla^{Q, \omega}$. 
\end{prop}
\begin{proof}
First we need to compute $\Lambda^3[\E\Hh]^\ast$. This is of real type, and can be computed via the following $\big(\gl([\E])\oplus \gl([\Hh])\big)$-invariant decomposition
\[
	\Lambda^3\E\Hh = \bigoplus_{Y\in \text{Young(3)}}Y(\E)Y^t(\Hh)\,,
\]
where  as  before $\text{Young}(3)$ denotes  the set of plethysms associated to Young diagrams with $3$ boxes, and $Y^t$ is the diagram $Y$ transposed. 
This is straightforward to evaluate and yields the following $\SO^*(2n)\Sp(1)$-equivariant decomposition
\begin{equation}\label{L3}
	\Lambda^3[\E\Hh]^\ast \cong [\Lambda^3\E S^3\Hh]^\ast \oplus [\K\Hh]^\ast \oplus [\E\Hh]^\ast\,.
\end{equation}
  Hence,  by Corollaries \ref{algebtypes} and \ref{genCar2}  we obtain the assertion $\Lambda^3[\E\Hh]^\ast \cong\mc{X}_{234}$.
 Next,  identifying $[\E\Hh]$ with the standard $\sp(4n,\R)$-module, we see that the branching of the $\sp(4n,\R)$-module $\delta([\E\Hh]^\ast \otimes \sp(4n,\R))$ to $\sp(1)\oplus \so^\ast(2n)$ is  isomorphic  to the quotient
\[
	\Lambda^2[\E\Hh]^\ast \otimes [\E\Hh]^\ast / \Lambda^3[\E\Hh]^\ast\,,
\]
as abstract modules.
 Then, by using the decomposition from Proposition \ref{genCar1}, as well as the computation above, we obtain the following equivariant isomorphism
\[
	\delta([\E\Hh]^\ast\otimes \sp(4n,\R)) \cong
	[(\K\oplus\E)\otimes S^3\Hh]^*\oplus [(\Lambda^3 \E\oplus\K\oplus2\E\oplus S^3_0 \E)\otimes \Hh]^*\,.
\]
This still leaves the question of embedding into $\Lambda^2 [\E\Hh]^\ast \otimes [\E\Hh]^\ast.$ We  can uniquely recognize the modules with multiplicity one in this embedding, independently of any $\SO^*(2n)\Sp(1)$-equivariant choice of embedding. Thus, the relation between the submodules of torsion tensors in $2[\K\Hh]^*\oplus 3[\E\Hh]^*$, $\mc{X}_{347}$ and $\Lambda^3[\E\Hh]^*\oplus\delta([\E\Hh]^*\otimes \fr{sp}(4n,\R))$ remains to be clarified. However, by definition of $\mc{X}_{347}$ in Corollaries \ref{algebtypes} and \ref{genCar2},  these modules are compatible with the decomposition $\Lambda^3[\E\Hh]^*\oplus\delta([\E\Hh]^*\otimes \fr{sp}(4n,\R))$. This completes the proof of the first assertion.\\
 Now, having the explicit formula of ${\sf Alt}_{\phi}$ by Lemma \ref{altoper} we see that 
\[
\ke({\sf Alt}_{\phi}\big|_{2[\K\Hh]^*})=[\K\Hh]^*\subset \delta([\E\Hh]^*\otimes \so^*(2n))\subset \delta([\E\Hh]^*\otimes \fr{sp}(4n,\R))\,,
\]
which follows  by counting multiplicities. Hence, $\pi_3$ is a well-defined isomorphism between the spaces  presented in the  second and third claim, respectively.  Be aware however  that  $\pi_3$  is not the identity map.\\
Finally, we have three parametrizations of  $3[\E\Hh]^*$ by Lemma \ref{kernel}. The first one is given  by the torsion components $2[\E\Hh]^*$ of $\nabla^{H,\omega}$ and by $[\E\Hh]^*\subset \delta([\E\Hh]^*\otimes \so^*(2n))$. The second one consists of the torsion component $\mc{X}_4$ of $\nabla^{Q,\omega}$, and of $\mc{X}_7\oplus [\E\Hh]^*\subset  \delta([\E\Hh]^*\otimes \sp(1))\oplus \delta([\E\Hh]^*\otimes \so^*(2n))$. The third one is provided by the traces $\Tr_1,\Tr_3,\Tr_4$. The transition matrix  from the first parametrization to the third one can be immediately deduced from the formula \eqref{transition}. The transition matrix from the second parametrization to the third one can be also deduced from the formula \eqref{transition}, and the explicit computation of the traces presented at the end of the proof of  Lemma \ref{kernel}. Finally, the composition of the endomorphisms corresponding to these two matrices provides the claimed formulas for $\pi_4$ and $\pi_7$. This finishes the proof of the second and third claim.
\end{proof}

\begin{rem}\label{remjan}
 \textnormal{The projections $\pi_3,\pi_4,\pi_7$ from Proposition \ref{symtorscomp} provide the difference of the torsions of  the two minimal connections  with respect to our normalization conditions from Theorem \ref{zeroprog},  and  the normalization condition given by the modules $\mc{X}_1,\ldots,\mc{X}_7$, respectively. Therefore, by the inverse of $\delta$ one obtains the difference of the corresponding minimal connections.  However, the formula for inverse of $\delta$ is too complicated to be presented here. }  \end{rem}

 Finally, as a conclusion of the above results we obtain the following.
\begin{corol} 
\textsf{1)} Let $(H, \omega)$ be a $\SO^*(2n)$-structure.
Then,   the corresponding intrinsic torsion is a 3-form if and only if  $(H, \omega)$  is of type $\mc{X}_{234}$, and it is of vectorial type (i.e., it is defined by a non-trivial vector field on $M$) if and only if $(H, \omega)$ is of type $\mc{X}_{47}$.\\
\textsf{2)} Let $(Q, \omega)$ be a $\SO^*(2n)\Sp(1)$-structure.
Then,   the corresponding intrinsic torsion is a 3-form if and only if  $(Q, \omega)$  is of type $\mc{X}_{234}$, and it is of vectorial type if and only if $(Q, \omega)$ is of type $\mc{X}_{4}$.
\end{corol}
Note that for $n=2, 3$ the module $\mc{X}_{234}$  decomposes into further irreducible submodules.

\subsection{Symplectomorphisms that are affine maps of minimal connections}

Since the first prolongation of our $G$-structures $G\in \{\SO^*(2n),\SO^*(2n)\Sp(1)\}$ vanishes,    the general theory of $G$-structures (see \cite{Kob2}) provides   several important assertions about the hypercomplex/quaternionic symplectomorphisms.   Such conclusions occur due to the uniqueness of the  minimal almost hypercomplex skew-Hermitian connection $\nabla^{H,\omega}$ (respectively, minimal almost quaternionic skew-Hermitian connection $\nabla^{Q,\omega}$), with respect to certain normalization conditions described in Theorem \ref{zeroprog}. In particular:

\begin{prop}\label{firstjet}
\textsf{1)} The hypercomplex symplectomorphisms between two \textsf{almost hs-H manifolds} $(M,H,\omega)$ and $(\hat{M},\hat{H},\hat{\omega})$ are those  affine transformations between $(M,\nabla^{H,\omega})$ and $(\hat{M},\nabla^{\hat{H},\hat{\omega}})$,   satisfying the relations
\[
f^*\hat{H}_{f(x)}=H_x\,, \quad f^*\hat{\omega}_{f(x)}=\omega_x\,,\quad x\in M\,.
\]
\textsf{2)} The quaternionic symplectomorphisms between two \textsf{almost qs-H manifolds} $(M,Q,\omega)$ and $(\hat{M},\hat{Q},\hat{\omega})$ are those affine transformations between $(M,\nabla^{Q,\omega})$ and $(\hat{M},\nabla^{\hat{Q},\hat{\omega}})$,
satisfying   the relations
\[
 f^*\hat{Q}_{f(x)}=Q_x\,,\quad f^*\hat{\omega}_{f(x)}=\omega_x\,,\quad x\in M\,.
\]
 \textsf{3)} If  two hypercomplex/quaternionic symplectomorphisms $f_1, f_2 : M\to \hat{M}$  satisfy $j^{1}_xf_1=j^{1}_xf_2$ for some $x\in M$, then $f_1=f_2$, where in general $j^{1}_{x}f$ denotes the first jet at $x\in M$ of a smooth function $f : M\to\hat{M}$.
\end{prop}

\begin{proof}
Since the hypercomplex/quaternionic symplectomorphisms map minimal connections to minimal connections, the first two claims are  consequences of  the uniqueness of such connections, see Corollary \ref{uniqueminimal}.  We leave the details of the remaining assertion to the reader.
\end{proof}

This result has the following classical consequences.

\begin{corol}
\textsf{1)} The group of hypercomplex symplectomorphisms of a $4n$-dimensional \textsf{almost hs-H manifold} $(M,H,\omega)$ is a Lie group of dimension less than or equal to $2n^2+3n.$\\
\textsf{2)} The group of quaternionic symplectomorphisms of a $4n$-dimensional \textsf{almost qs-H manifold} $(M,Q,\omega)$ is a Lie group of dimension less than or equal to $2n^2+3n+3.$
\end{corol}
\begin{proof}
	By Lemma \ref{zeroprog1}, a  Lie algebra $\fr{g}\in \{\fr{so}^*(2n),\fr{so}^*(2n)\sp(1)\}$  has vanishing first prolongation. Since the dimension of the automorphism group of a $G$-structure with $\fr{g}^{(1)}$ trivial is bounded by  $\dim_{\R}M+\dim_{\R}\fr{g}$, see \cite{Kob2}, the claim follows. 
\end{proof}

\section{Torsion-free examples}\label{tfexamples}
 In this final section of this article,  we will focus on torsion-free examples.  In particular, based on certain conclusions presented in the articles  \cite{ACort} and \cite{G13}, we  shall present the  classification  of symmetric spaces $K/L$ admitting \textsf{invariant torsion-free $\SO^*(2n)\Sp(1)$-structures}, under the assumption that $K$ is semisimple. Moreover, we recall a construction from the theory of special symplectic connections.  
   Note that many examples with torsion which realize   some of the types $\mc{X}_{i_1\ldots i_j}$ introduced in this article, are  described  in \cite{CGWPartII}.  There we illustrate non-integrable $\SO^*(2n)$- and $\SO^*(2n)\Sp(1)$-structures  in terms of   both homogeneous  and non-homogeneous geometries, and other constructions arising for example by using the bundle of Weyl structures, and more.  
\subsection{Semisimple symmetric spaces}
 For $\SO^*(2n)$- and $\SO^*(2n)\Sp(1)$-structures  a natural source where one may initially look for \textsf{integrable examples}  is the category of symmetric spaces (see \cite{Hel} for the theory of symmetric spaces).  By the results in  \cite{G13} it follows that  there are no semisimple symmetric spaces with an  invariant  almost hypercomplex skew-Hermitian structure. However, let us consider the symmetric space 
\[
M=K/L=\SO^*(2n+2)/\SO^*(2n)\U(1)\,.
\]
By \cite{G13} it is known that $M$ carries a $\SO^*(2n+2)$-invariant quaternionic structure, although  it is {\it not} a pseudo-Wolf space. Let us denote by $\fr{k}=\fr{l}\oplus\fr{m}$ the corresponding Cartan decomposition  and by $\chi : K\to\Aut(\fr{m})$ the isotropy representation, where as usual we identify $\fr{m}\cong T_{eL}K/L$. Since 
\[
\chi(\SO^*(2n)\U(1))\subset \SO^*(2n)\Sp(1)
\]
 clearly holds, there is also an invariant \textsf{qs-H} structure on $K/L$.  Let us consider   the $\Ad(L)$-invariant complex structure   $I_{o} : \fr{m}\to\fr{m}$  induced by $\chi|_{\U(1)}$, that is
 $I_{o}=\chi_{*}(U)=\ad(U)$ for some $U\in\fr{u}(1)$. Let us also denote by $\langle \ , \ \rangle_{\fr{m}}$ the $\Ad(L)$-invariant  symmetric pseudo-Hermitian metric (with respect to $I_{o}$) on $\fr{m}$. Note that $I_{o}$ corresponds to a  $K$-invariant complex structure $I$ on $M$ which we may use to build a local admissible base of the  invariant quaternionic structure $Q$ on the origin $o=eL\in K/L$,   induced  by  the $\Sp(1)$ action on $\fr{m}$.    Then,  since $M=K/L$ is isotropy irreducible, by Schur's lemma we deduce that  the 2-form $\omega$ defined by $\omega_o(\cdot\,, \cdot):=\langle I_{o}\cdot\,, \cdot\rangle_{\fr{m}}$  is an $\Ad(L)$-invariant scalar 2-form with respect to $Q$. In terms of  Proposition \ref{signprop}  this  means that $\langle \ , \ \rangle_{\fr{m}}=g_{I_{o}}$.
 \begin{rem}
 \textnormal{Note that $\langle \ , \ \rangle_{\fr{m}}$ induces the unique (up to scale) $K$-invariant Einstein metric on $M=K/L$ of signature $(2n, 2n)$, which is actually a multiple of the Killing form of $\SO^*(2n+2)$ restricted to $\fr{m}$.}
 \end{rem}

Next  by using the  classification of the pseudo-Wolf spaces, i.e., quaternionic pseudo-K\"ahler symmetric spaces, given by   Alekseevsky-Cortes in \cite{ACort}, we prove that in addition to $M=K/L$ there are a few more  symmetric spaces with the same property.  
\begin{theorem}\label{homogthem}
 The symmetric space $\SO^*(2n+2)/\SO^*(2n)\U(1)$ and the pseudo-Wolf spaces 
\[
\SU(2+p,q)/(\SU(2)\SU(p,q)\U(1))\,,\quad\ \Sl(n+1,\mathbb{H})/(\Gl(1,\mathbb{H})\Sl(n,\mathbb{H}))
\]
 are the only (up to covering) symmetric spaces $K/L$ with $K$ semisimple, admitting an invariant torsion-free quaternionic skew-Hermitian structure $(Q, \omega)$.  In particular, the corresponding canonical connections on  these symmetric spaces coincides with  the  associated minimal  quaternionic skew-Hermitian connection $\nabla^{Q,\omega}$.
\end{theorem}
\begin{proof} 
As it was shown in \cite{G13}, the classification of invariant quaternionic structures on semisimple symmetric spaces $K/L$ is divided according to the dimension of intersection of $\chi(L)$ with $\Sp(1)$, where  $\chi : L\to\Aut(\fr{m})$ is the isotropy representation and $\fr{m}$ is the symmetric reductive complement. This dimension can not be zero. Up to covering, the above symmetric space $\SO^*(2n+2)/\SO^*(2n)\U(1)$ is the only one in the classification with  one dimensional intersection, i.e., $\chi(L)\cap\Sp(1)=\U(1)$. This coincides with the center of $L$.

The remaining cases in the classification have the intersection $\chi(L)\cap\Sp(1)=\Sp(1)$ and this induces an invariant quaternionic structure $Q$ on $K/L$ induced by the isotropy representation $\chi$. This means that in this case, the invariant pseudo-Riemmanian metric $g$ induced by restriction of the Killing form to $\fr{m}$ is a quaternionic pseudo-K\"ahler metric, and so $K/L$ is a pseudo-Wolf space. The classification of  pseudo-Wolf spaces was obtained in \cite[Theorem 2]{ACort}. On the other hand, it is well-known that invariant symplectic structures on a simple symmetric space correspond to the center of the isotropy algebra. Then, the isotropy action of the center of the stabilizer  provides a invariant complex or paracomplex structure $I\notin \Gamma(Q)$ on $K/L$ and
 \[
 \omega(\cdot\,, \cdot):=g(I\cdot\,, \cdot)\,,
 \] 
is an invariant scalar 2-form with respect to $Q$. In the classification of the pseudo-Wolf spaces obtained in \cite[Theorem 2]{ACort}, we see that 
\[
\SU(2+p,q)/(\SU(2)\SU(p,q)\U(1))\,,\quad \Sl(n+1,\mathbb{H})/(\Gl(1,\mathbb{H})\Sl(n,\mathbb{H}))
\]
 are the only  pseudo-Wolf spaces for which the  isotropy algebra contains a   non-trivial center. This proves our first assertion. Recall finally that by the Ambrose-Singer theorem the
canonical connection  $\nabla^0$ on $K/L$ must preserve $(Q, \omega)$ and is in particular a minimal connection because it is torsion free, see \cite{KoNo}.  Then,  we obtain the identification $\nabla^{0}=\nabla^{Q, \omega}$ by uniqueness of the minimal connection, see also Proposition \ref{uniqueminimal}.
\end{proof}

\subsection{Examples with special symplectic holonomy}

 Let us recall that
on a symplectic manifold $(M, \omega)$ a symplectic 
connection $\nabla$ is said to be of \textsf{special symplectic holonomy} if ${\sf Hol}(\nabla)$ is a proper subgroup of $\Sp(2n,\R)$ that acts absolutely irreducibly on the tangent space, i.e., it acts irreducibly and does not preserve a complex structure. A  \textsf{special symplectic connection} is a symplectic connection with special   symplectic  holonomy, and it is known that such connections may exist only in dimensions $\geq 4$.
The first special symplectic holonomies were constructed by Bryant \cite{Bryant91}, and by  Chi, Merkulov and Schwachh\"ofer  \cite{ChiMS1, ChiMS2}.  Finally, these exotic holonomies were classified by
Merkulov and   Schwachh\"ofer   and include the Lie group $\SO^*(2n)\Sp(1)$, see for example  \cite[Table 3]{MS1} (note that in contrast to $\SO^*(2n)$, the Lie group $\SO^{\ast}(2n)\Sp(1)$ is a real non-symmetric Berger subgroup, see \cite[Tab.~II]{Schw}).

The construction providing such special symplectic holonomies has been  described in \cite{CahS}. Let us recall how this procedure works for  a  $\SO^*(2n)\Sp(1)$-structure $(Q, \omega)$. 
Let $P$ be the connected subgroup of $\SO^\ast(2n+2)$ which stabilizes an isotropic (with respect to $\omega$) quaternionic line in $\Hn^{n+1}$.   Then, the homogeneous space $N=\SO^\ast(2n+2)/P$ admits an invariant contact structure (see \cite[p.~298]{CS}), and we denote by $\mathscr{D}$  the corresponding contact distribution. In such terms we obtain the following local construction:

\begin{prop}
	Let $(Q,\omega,\nabla^{Q,\omega})$ be a smooth torsion-free $\SO^*(2n)\Sp(1)$-structure with  special symplectic holonomy, i.e.,  $T^{Q,\omega}=0$ and  \ ${\sf Hol}(\nabla^{Q, \omega})=\SO^*(2n)\Sp(1)$. 
	Then $(Q,\omega,\nabla^{Q,\omega})$ is analytic, and locally equivalent to a symplectic reduction $\mathbb{T}\backslash U$ by a one-parameter subgroup $\mathbb{T}\subset \SO^\ast(2n+2)$ with Lie algebra $\mathfrak{t}$, such that the corresponding right-invariant vector fields are transversal to $\mathscr{D}$ everywhere on $U$. 
	Here $U\subset N$ is a sufficiently small open subset of $N$.
	In particular, the moduli space of such structures is $n$-dimensional, where $n$ represents the quaternionic dimension of the symplectic reduction.
	\end{prop}
\begin{proof}
	This result occurs  as the restriction of  \cite[Corollary C]{CahS} to the particular case of torsion-free almost quaternionic skew-Hermitian structures.
\end{proof}
 
  Note that the manifold $N=\SO^\ast(2n+2)/P$ happens to be a flat parabolic geometry. The interplay  between the parabolic geometry on $N$  and the  almost conformal symplectic geometry on the symplectic reduction was explored in detail by \v{Cap} and Sala\v{c} in a series of   papers, see  \cite{Cap, CapII} for example.
  Indeed,  they described a generalization of the above construction in the presence of torsion. This construction   requires $\mathbb{T}$ to be  a  flow of a transversal infinitesimal automorphism of the parabolic contact structure,  and yields structures which are almost conformally symplectic, rather than almost symplectic, and hence  less relevant to our current situation.

\appendix

\section{Adapted bases with coordinates in $\Hn^n$}\label{appendix}
  
    \subsection{Left quaternionic vector space $\Hn^n$}
    
   Observe that after a choice of an admissible hypercomplex basis $H$ as in Lemma \ref{IJKstandard} we can identify $a+bj\in [\E\Hh]$  with $a+bj\in\Hn^n$, where the latter is viewed as  a left quaternionic vector space.  Consequently:
    
     \begin{corol}\label{corolA1}
 Let $(h, H=\{J_1,J_2,J_3\})$ be a linear hypercomplex skew-Hermitian structure on a $4n$-dimensional real vector space $V$, or let $H=\{J_1,J_2,J_3\}$ be an admissible basis of the linear quaternionic skew-Hermitian structure $h$ on $V$, and set $\omega:=\Re(h)$. Then, a skew-Hermitian basis of $(h, H)$, in terms of Definition \ref{basSKEW},  provides an isomorphism $V\cong \Hn^n$, such that
 \begin{enumerate}
 \item[$\mf{(1)}$] $(a_1+a_2J_1+a_3J_2+a_4J_3)x=(a_1+a_2i+a_3j+a_4k)x$, for any $x\in  \Hn^n.$
 \item[$\mf{(2)}$]  $\omega(x,y)=\frac12(x^tj\bar{y}-y^tj\bar{x})$ for all $x,y\in  \Hn^n$, where $\bar{x}$ is the quaternionic conjugate.
 \item[$\mf{(3)}$]  $g_{\J}(x,y)=\frac12(-x^tj\bar{y}(\mu_1 i+\mu_2 j+\mu_3 k)-(\mu_1 i+\mu_2 j+\mu_3 k)y^tj\bar{x})$, for $\J=\mu_1 J_1+\mu_2 J_2+\mu_3 J_3\in\Ss^2$.
  \item[$\mf{(4)}$]  $h(x,y)=x^tj\bar{y}$.
 \end{enumerate}
\end{corol}
  \begin{rem}
\textnormal{The formula $h(x,y)=x^tj\bar{y}$ is the usual formula for a skew-Hermitian form on the left quaternionic space vector space $\Hn^n$, but be aware that some authors replace equivalently $j$ by $i$.}
\end{rem}
    \subsection{Right quaternionic vector space  $\Hn^n$}
    
    Left and right quaternionic vector spaces are related by conjugation. Therefore, after the choice of an admissible hypercomplex basis $H$, the element $\bar{a}-bj\in [\E\Hh]$ can be identified with $a+bj$ in the right quaternionic  vector space $\Hn^n$ (by Lemma \ref{IJKstandard}). Thus, we get the following 
    
         \begin{corol}
 Let $(h, H=\{J_1,J_2,J_3\})$ be a linear hypercomplex skew-Hermitian structure on $V$, or let $H=\{J_1,J_2,J_3\}$ be an admissible basis of the linear quaternionic skew-Hermitian structure $h$, and let $\omega=\Re(h)$. Then a skew-Hermitian basis of $(h, H)$,   in terms of Definition \ref{basSKEW},  provides an  isomorphism $V\cong \Hn^n$,  such that
 \begin{enumerate}
 \item[$\mf{(1)}$]  $(a_1+a_2J_1+a_3J_2+a_4J_3)x=x(a_1-a_2i-a_3j-a_4k)$, for any $x\in  \Hn^n.$
 \item[$\mf{(2)}$]  $\omega(x,y)=\frac12(x^*jy-y^*jx)$ for all $x,y\in  \Hn^n$, where $x^*$ is conjugate transpose.
 \item[$\mf{(3)}$]  $g_{\J}(x,y)=\frac12(x^*(-\mu_2+\mu_3 i-\mu_1 k)y+y^*(-\mu_2-i\mu_3+\mu_1 k)x)$, for  $\J=\mu_1 J_1+\mu_2 J_2+\mu_3 J_3\in\Ss^2$.
 \item[$\mf{(4)}$]  $h(x,y)=x^*jy$.
 \end{enumerate}
\end{corol}
  \begin{rem}
\textnormal{The formula $h(x,y)=x^*jy$ is the usual formula for a skew-Hermitian form on the right quaternionic space vector space $\Hn^n$. However, observe that  some authors replace equivalently $j$ by $i$, see \cite[p.~8]{Harvey}.}
\end{rem}
     
On the right quaternionic vector space $\Hn^n$, we can find a skew-Hermitian  basis of the linear \textsf{hs-H} structure $(\Hn^n, h)$, as follows:
\begin{enumerate}
\item[$\mf{(I)}$] We start with a quaternionic skew-Hermitian form $h$ on  $\Hn^n$.
\item[$\mf{(IIa)}$] If $n=1$, then by definition $h(x,y)=\bar{x}(h_1i+h_2j+h_3k)y$ for some $h_1,h_2,h_3\in \R$, and by non-degeneracy exists  some $q\in \Hn$ such that $\bar{q}(h_1i+h_2j+h_3k)q=j$. Thus we are done.
\item[$\mf{(IIb)}$] If $n>1$, then we start by finding $e_1\in \Hn^n$, such that $h(e_1x,e_1y)=\bar{x}jy,\ x,y\in \Hn$. Since $h$ is non-degenerate, there are $f_1,f_2\in \Hn^n$ such that $h(f_1,f_2)\neq 0$. Thus,  for either $e=f_1$, $e=f_2$, or $e=kf_1+jf_2$ it holds $h(e,e)\neq 0$ and  the step \textsf{(IIa)} can be applied, i.e., $e_1=qe$.
\item[$\mf{(III)}$] If $n>1$, then we complete $e_1$ to a quaternionic basis of $\Hn^n$. In this basis, $h$ can be expressed by the following block matrix
\[\left( 
    \begin{smallmatrix}
j & -X^*\\
 X&\square
  \end{smallmatrix}
  \right)
  \]
  for $X\in \Hn^{n-1}$, and $\square$ is a  quantity not important for us.    Thus, changing the quaternionic basis by left multiplication by $\left( \begin{smallmatrix}
 1 & -jX^\ast\\
 0&\id_{\Hn^{n-1}}
  \end{smallmatrix}  \right)$, we compute
  \[\left( \begin{smallmatrix}
 1 & -jX^\ast\\
 0&\id_{\Hn^{n-1}}
  \end{smallmatrix}  \right)^*\left( 
    \begin{smallmatrix}
 j & -X^*\\
 X&\square
  \end{smallmatrix}
  \right)\left( \begin{smallmatrix}
 1 & -jX^*\\
 0&\id_{\Hn^{n-1}}
  \end{smallmatrix}  \right)=\left( \begin{smallmatrix}
 1 &0\\
 Xj&\id_{\Hn^{n-1}}
  \end{smallmatrix}  \right)\left( 
    \begin{smallmatrix}
 j &0\\
 X&\square
  \end{smallmatrix} \right)=\left( 
    \begin{smallmatrix}
 j &0\\
 0&\square
  \end{smallmatrix}
  \right)\,.
  \]
  Consequently, we have constructed an orthogonal complement of $e_1$, and we may repeat the algorithm for restriction of $h$ to the orthogonal complement.  Then we conclude by induction with respect to the dimension.
\end{enumerate}

At this point the right basis looks very useful, however there is still the following task.

\begin{lem}
 Let  $(h, H=\{J_1,J_2,J_3\})$ be a  linear hypercomplex skew-Hermitian structure on $V$.  Then, there is no basis $e_1,\dots,e_{2n},f_1,\dots,f_{2n}$ of $V$ satisfying  
\[
J_1(e_{c})=-e_{c+n}\,,\quad
J_2(e_{c})=-f_{c}\,,\quad
J_3(e_{c})=-f_{c+n}\,,
\] 
which is at the same time a symplectic basis with respect to  the induced scalar 2-form $\omega=\Re(h)$.
 \end{lem}
 \begin{proof}
In a skew-Hermitian basis of $(H, \omega)$ (in terms of Definition \ref{basSKEW}), and  after identifying  $\bar{a}-bj\in [\E\Hh]$ with $a+bj\in \Hn^n$, we compute 
\[
\omega(e_c,f_c)=1\,,\quad \omega(e_{c+n},f_{c+n})=-1
\]
 for $c=1, \ldots, n$, and   all the other  combinations (up to antisymmetry) are zero.  The action of $\Gl(n,\Hn)$ on $\omega$ preserves the property $\omega(e_c,f_c)=-\omega(e_{c+n},f_{c+n})$, and thus there is no basis satisfying the given condition.
  \end{proof}
     
  However,  as we show below, this problem can be resolved exactly when    the  quaternionic dimension is {\it even}. This is because 
of the existence of another natural ordering of a basis of a right quaternionic vector space, which can be adapted to our purpose.  

\begin{lem}\label{darb}
Let  $H=\{J_1,J_2,J_3\}$ be  a linear hypercomplex structure on $V$. Then, $V$ admits  a basis $e_1,\dots,e_{4n}$  satisfying 
  \[
  J_1(e_{4a-3})=-e_{4a-2}\,,\quad J_2(e_{4a-3})=-e_{4a-1}\,, \quad J_3(e_{4a-3})=-e_{4a}\,,
  \] 
and   such that  the set of vectors  $e_1,\dots,e_{2n},f_1=e_{2n+1},\dots,f_{2n}=e_{4n}$  is  a symplectic basis for a scalar 2-form $\omega$ on $V$, if and only if $n$ is even, $n=2m$.
\end{lem}
\begin{proof}
Let us consider $e_1,\dots,e_{4n}$ as the reordering of a symplectic basis adapted to a hypercomplex structure $H$, after identifying  $\bar{a}-bj\in [\E\Hh]$ with $a+bj\in \Hn^n$, so that the conditions on the action of $H$ are satisfied. Then, the formula $\omega(x,y)=\frac12(x^*jy-y^*jx)$ is still {\it not} the standard symplectic form in this basis.  For $n=1$, it is immediate to show that the transformation $q^*jq$ never induces the standard symplectic form. But if $n=2$, then we can directly compute that the following matrix 
\[\mathscr{C}:=\left( 
    \begin{smallmatrix}
 -\frac12 k & i\\
 -\frac12 j &-1
  \end{smallmatrix}
  \right)\] 
  satisfies 
\[h(\mathscr{C}x, \mathscr{C}y)=x^*
\left(\begin{smallmatrix}
 \frac12 k & \frac12 j \\
 -i&-1
  \end{smallmatrix}
  \right)\left(\begin{smallmatrix}
 -\frac12 i  & -k\\
 \frac12  &j
  \end{smallmatrix}
  \right)y=x^*\left(\begin{smallmatrix}
 0 & 1\\
 -1 &0
  \end{smallmatrix}
  \right)y\,.
\]
This  provides the   standard expression for $\omega$ in our new basis, by setting $\omega=\Re(h)$.  This procedure can be successfully generalized to  the general case $n=2m$,  via the matrix
   \[\mathscr{C}_{2m}:=\left( 
    \begin{smallmatrix}
 -\frac12 k\id_{\Hn^m} & i\id_{\Hn^m}\\
 -\frac12 j\id_{\Hn^m} &-\id_{\Hn^m}
  \end{smallmatrix}
  \right)\,.\] 
In particular, note that $\mathscr{C}_{2}=\mathscr{C}$. However,  for $n=2m+1$ this is not possible. Indeed,   we can restrict $\mathscr{C}_{2m}$  to the first $4m$ vectors of the above basis, and for the remaining vectors in the basis, we get $j$ in the diagonal of the matrix corresponding to $\omega$.  Hence, for $n=2m+1$ this task can not be resolved, in line with the case $n=1$.  
  \end{proof}
     
It is reasonable to provide a special name for the basis  constructed  in Lemma \ref{darb}, as we do below.
     
    \begin{defi}\label{darbas}
Let $e_1,\dots,e_{4m},f_1,\dots,f_{4m}$ be  a symplectic basis on a symplectic vector space $(V,\omega)$ of real dimension $8m$.  Then, we say that it forms a  \textsf{quaternionic Darboux basis} for the linear hypercomplex structure $H=\{J_1,J_2,J_3\}$  on $V$,   if the following holds:
\[
  \begin{tabular}{l l}
  $J_1(e_{4a-3})=-e_{4a-2}$\,,  & $J_1(f_{4a-3})=-f_{4a-2}$\,,\\
  $J_2(e_{4a-3})=-e_{4a-1}$\,, &  $J_2(f_{4a-3})=-f_{4a-1}$\,,\\
  $J_3(e_{4a-3})=-e_{4a}$\,, &   $J_3(f_{4a-3})=-f_{4a}$\,, 
  \end{tabular}
  \]
for any $a=1,\ldots, m$.
 \end{defi}
 Note that this definition coincides with the definition of  an admissible basis to a linear hypercomplex structure, given in \cite{AM}, but be aware that their $V$ is a left quaternionic vector space.
 
 \begin{example}
For $n=2$,  assume that $e_1, e_2$ are non-zero vectors in $\Hn^2$ for which the quaternionic lines $e_1\cdot \Hn$ and $e_2\cdot \Hn$ do not coincide. If $H$ is the linear hypercomplex structure induced by right multiplication via $-i,-j,-k$, then a quaternionic Darboux basis is given by 
\[
\{e_1,e_1i,e_1j,e_1k, e_2, e_2i,e_2j,e_2k\}\,.
\]
 Let us  now consider the skew-Hermitian basis $\fr{B}=\{e_1, e_2, ie_1, ie_2, je_1, je_2, ke_1, ke_2\}$ of $\Hn^2$ (viewed as a left quaternionic vector space) described in Example \ref{examplebase2}.  Then  we can multiply the transition matrix $\mathscr{C}$ in the proof of Lemma \ref{darb} with the matrix corresponding to the linear quaternionic conjugation. This composition  maps the coordinates $(a_1,\dots,a_8)$ in the quaternionic Darboux basis to 
\[
(-a_6+\frac12a_4, -a_5+\frac12a_3, -a_5-\frac12a_3, a_6+\frac12a_4, a_8+\frac12a_2, a_7+\frac12a_1, -a_7+\frac12a_1, a_8-\frac12a_2)
\]
 in the skew-Hermitian basis $\fr{B}=\{e_1, e_2, ie_1, ie_2, je_1, je_2, ke_1, ke_2\}$.  Note that the indicated composition is a linear isomorphism between the linear  hs-H structures provided by these bases.
 \end{example}
 Let us finally emphasize the following nice application of quaternionic Darboux bases in the theory of parabolic geometries.
\begin{prop}\label{gradbas}
\textsf{1)} In a quaternionic Darboux basis, $\so^{*}(4n)$ carries a $|1|$-grading 
\[
 \so^{*}(4n)_{-1}\oplus  \so^{*}(4n)_{0}\oplus  \so^{*}(4n)_{1}
 \]
  represented by the following matrix
\[
\left(\begin{smallmatrix}
A & C\\
B &-A^*
\end{smallmatrix}\right),
\]
for $A,B,C\in \gl(n,\Hn)$, $A\in \so^{*}(4n)_{0}$ $B^*=B\in \so^{*}(4n)_{-1},C^*=C\in \so^{*}(4n)_{1}.$ This depth 1-gradation corresponds to the last node in the corresponding Satake diagram.\\
\textsf{2)} In the basis from the proof of Lemma \ref{darb}, $\so^{*}(4n+2)$  carries a $|2|$-grading 
\[
\so^{*}(4n)_{-2}\oplus \so^{*}(4n)_{-1}\oplus  \so^{*}(4n)_{0}\oplus  \so^{*}(4n)_{1}\oplus \so^{*}(4n)_{2}
\]
 represented by the following matrix
\[
\left(\begin{smallmatrix}
A &Y& C\\
X& uj & jY^*\\
B & -X^*j& -A^*
\end{smallmatrix}\right),
\]
for $A,B,C\in \gl(n,\Hn)$, $A\in \so^{*}(4n)_{0}, B^*=B\in  \so^{*}(4n)_{-2},C^*=C\in \so^{*}(4n)_{2}$, $X,Y\in \Hn^n$, $X\in \so^{*}(4n)_{-1}$, $Y\in \so^{*}(4n)_{1}$ and $u\in \R\subset \so^{*}(4n)_{0}$. This depth 2-gradation corresponds to the last two nodes in the corresponding Satake diagram.\\
\textsf{3)} There is a diagonal Cartan subalgebra of $\so^{*}(4n)$ and $\so^{*}(4n+2)$ in both of these matrix representations.
\end{prop}
\begin{proof}
We can use the  transition  matrix $\mathscr{C}_{2n}$ posed in the proof of Lemma \ref{darb} on our representation of $\so^{*}(4n)$ and $\so^{*}(4n+2)$ and obtain the claimed matrices.  Also,  it is a simple observation that the claimed decompositions are $|1|$-gradings and $|2|$-gradings, respectively, and that there is a diagonal Cartan subalgebra.
\end{proof}
We should mention that the $|2|$-grading described above, differs from the unique contact $|2|$-grading  corresponding to the second node of the Satake diagram associated to $\fr{so}^*(2n)$, see \cite{CS}. In particular, a decomposition of  the  matrix given  in Proposition \ref{gradbas} for the $\fr{so}^*(4n)$-case into blocks of size $1$, $(n-1)$, $1$ and $(n-1)$,  provides the contact $|2|$-grading of $\fr{so}^*(4n)$.  Similarly,  a decomposition of  the  matrix given in the second part of Proposition \ref{gradbas}   into blocks of size $1$, $n$, $1$ and  $(n-1)$, determines  the contact $|2|$-grading of $\fr{so}^*(4n+2)$.

{ }
\end{document}